\documentclass{article}
\usepackage{amsmath}
\usepackage{amssymb}
\usepackage{amsfonts}
\usepackage{graphicx}
\usepackage[utf8]{inputenc} 
\usepackage[english]{babel}
\usepackage{dsfont,hyperref}
\usepackage{stackrel,stmaryrd}
\usepackage{url}
\usepackage{vmargin}
\newcommand{\vertiii}[1]{%
	{\left\vert\kern-0.25ex\left\vert\kern-0.25ex\left\vert
		#1
		\right\vert\kern-0.25ex\right\vert\kern-0.25ex\right\vert}}

\setmarginsrb{2cm}{2cm}{2cm}{1cm}{0cm}{0cm}{0cm}{0cm}

\DeclareMathOperator*{\esssup}{ess\,sup}

\newenvironment{class}[1][(2020) Mathematics Subject Classification]{\textbf{#1.} }{}
\newenvironment{MC}[1][Key words and phrases]{\textbf{#1.} }{}
\newtheorem{theorem}{Theorem}

\newtheorem{corollary}[theorem]{Corollary}

\newtheorem{definition}[theorem]{Definition}
\newtheorem{example}[theorem]{Example}

\newtheorem{lemma}[theorem]{Lemma}

\newtheorem{proposition}[theorem]{Proposition}
\newtheorem{remark}[theorem]{Remark}

\newenvironment{proof}[1][Proof]{\noindent\textbf{#1.} }{\ \rule{0.5em}{0.5em}}

\begin{document}
	
	\title{Generalized reflected BSDEs with irregular obstacles driven by RCLL increasing processes on general filtered space}
	
	\author{
		\textbf{Badr ELMANSOURI}$^a$\footnote{Corresponding author} \\
		\small $^a$Cadi Ayyad University (UCA), \\
		\small National School of Applied Sciences of Marrakech (ENSA-M),\\ 
		\small Avenue Abdelkrim Khattabi, BP 575, Guéliz, Marrakech, 40000, Morocco.\\
		\small Email: \url{b.elmansouri@uca.ac.ma}  \\
		\\
		\textbf{Youssef OUKNINE}$^{b,c}$\\
		\small $^b$Cadi Ayyad University (UCA), Faculty of Sciences Semlalia (FSSM),\\ 
		\small Department of Mathematics, B.P. 2390, Marrakesh, 40000, Morocco.\\
		\small $^c$Mohammed VI Polytechnic University, Africa Business School,\\
		\small Avenue Mohammed Ben Abdellah Regragui,\\
		\small Madinat Al Irfane, BP 6380, Rabat, Morocco.\\
		\small Emails: \url{ouknine@uca.ac.ma} \& \url{youssef.ouknine@um6p.ma} 
	}	
	
	\maketitle
	
	\begin{abstract}
		We study generalized backward stochastic differential equations (GBSDEs) and generalized reflected backward stochastic differential equations (GRBSDEs) on a general filtered probability space satisfying the usual conditions, without assuming that the underlying filtration is quasi-left-continuous. The equations are driven by a prescribed predictable, bounded, nondecreasing RCLL process \(A\), which acts as a possibly discontinuous stochastic clock, which we call a driver. We first establish a priori estimates, stability, existence, and uniqueness results for GBSDEs whose generator is Lipschitz continuous with respect to the state variable. Since \(A\) may have jumps, the analysis is carried out in weighted spaces defined through the stochastic exponential \(\mathcal{E}(\beta A)\). We then investigate GRBSDEs with an optional regulated lower obstacle. When the generator is independent of the state variable, we develop two complementary approaches. The first relies on a Snell-envelope representation and optimal stopping arguments, while the second is based on a modified penalization procedure adapted to the discontinuities of the right jumps of the obstacle. The general Lipschitz case is subsequently obtained through a fixed-point argument in an appropriate weighted Banach space. 
	\end{abstract}
	
	\begin{MC}
		 Generalized reflected BSDEs; regulated obstacle; Snell envelope;\\
		 penalization method; general filtration.
	\end{MC}
	
	\begin{class}
		60H05; 60H20; 60H30.
	\end{class}
	
	\section{Introduction} \label{Intro}
	The theory of backward stochastic differential equations (BSDEs) has been extensively developed over the past decades, owing to its wide range of applications. These include mathematical finance, particularly the pricing and hedging of European contingent claims under the no-arbitrage principle
	\cite{BarrieuElKaroui2005,ElKarouiPengQuenez1997,RougeElKaroui2000},
	stochastic control and differential games
	\cite{HamadeneLepeltier1995,HamadeneLepeltierPeng1997,KaratzasLi2012},
	and the probabilistic representation of solutions to partial differential equations
	\cite{Pardoux1999,PardouxPeng1992}.
	For further developments and applications, we refer the reader, among others, to
	\cite{buckdahn2008stochastic,Carmona2016,Crepey2013BSDE,PapapantoleonPossamaiSaplaouras2023,Perninge2024,QuenezSulem2013,QuenezSulem2014,Touzi2013}.
	
	BSDEs were first introduced by Bismut \cite{bismut1973conjugate} in the linear setting as adjoint equations arising from stochastic versions of Pontryagin's maximum principle in control theory. They were subsequently studied in a nonlinear framework by Pardoux and Peng \cite{PardouxPeng1990}. Let $T\in(0,+\infty)$ be a fixed deterministic terminal time, which may be interpreted as the maturity of a financial contract, and let $(\Omega,\mathcal{F},\mathbb{P})$ be a complete probability space. Given a terminal condition $\xi_T$ and a generator
	\[
	f:\Omega\times[0,T]\times\mathbb{R}\times\mathbb{R}
	\longrightarrow \mathbb{R},
	\]
	a solution of the corresponding BSDE is a pair of stochastic processes
	$(Y,Z)=(Y_t,Z_t)_{t\leq T}$ satisfying
	\begin{equation}\label{intro1}
		Y_t
		=
		\xi_T
		+\int_t^T f(s,Y_s,Z_s)\,ds
		-\int_t^T Z_s\,dB_s,
		\qquad t\in[0,T],
	\end{equation}
	where $B$ is a standard Brownian motion defined on
	$(\Omega,\mathcal{F},\mathbb{P})$, and $(Y,Z)$ is adapted to the augmented natural filtration
	$\mathbb{F}=(\mathcal{F}_t)_{t\leq T}$ generated by $B$.
	
	Motivated by applications such as the pricing and hedging of American options, the classical BSDE framework was extended to reflected backward stochastic differential equations (RBSDEs). In this setting, the state process $Y$ is constrained to remain above a given stochastic process
	$\xi=(\xi_t)_{t\leq T}$, called the lower obstacle or barrier. In financial applications, $\xi$ may represent the payoff process of an American option in a given market model, such as the Black--Scholes model. A classical RBSDE takes the form
	\begin{equation}\label{intro2}
		\left\{
		\begin{aligned}
			\textnormal{(i)}\quad
			&Y_t
			=
			\xi_T
			+\int_t^T f(s,Y_s,Z_s)\,ds
			+(K_T-K_t)
			-\int_t^T Z_s\,dB_s,
			\\
			\textnormal{(ii)}\quad
			&Y_t
			\geq \xi_t,
			\qquad 0\leq t\leq T,\quad \textnormal{a.s.},
			\\
			\textnormal{(iii)}\quad
			&\int_0^T (Y_s-\xi_s)\,dK_s
			=0,
			\qquad \textnormal{a.s.}
		\end{aligned}
		\right.
	\end{equation}
	The reflecting process $K$ is a continuous, nondecreasing process that acts only when the constraint is active. More precisely, the Skorokhod minimality condition \eqref{intro2}--\textnormal{(iii)} means that the Stieltjes measure $dK$ is carried by the contact set
	\[
	\{(\omega,t):Y_t(\omega)=\xi_t(\omega)\}.
	\]
	Thus, $K$ increases only when $Y$ reaches the obstacle and would otherwise tend to move below it, thereby ensuring that the constraint in
	\eqref{intro2}--\textnormal{(ii)} is preserved.
	
	This notion was introduced by El Karoui et al.
	\cite{ElKarouiKapoudjianPardouxPengQuenez1997}. The authors established that, when the terminal condition $\xi_T$ is square-integrable, the generator $f$ is uniformly Lipschitz continuous, and the obstacle $\xi$ is continuous, the RBSDE \eqref{intro2} admits a unique solution. The importance of equations of the form \eqref{intro2} in optimal stopping, stochastic control, and mathematical finance has motivated numerous subsequent extensions. We refer, for instance, to
	\cite{ElKarouiPardouxQuenez1997,EsmaeeliImkeller2018,fuhrman2017reflected,Liang2015,PengXu2010}
	and the references therein.
	
	Several extensions of the classical RBSDE \eqref{intro2} have been developed in the literature. A first line of research consists in weakening the assumptions imposed in
	\cite{ElKarouiKapoudjianPardouxPengQuenez1997}, while retaining the same underlying filtered probability space
	$(\Omega,\mathcal{F},\mathbb{F},\mathbb{P})$. In this direction, Hamadène and Lepeltier
	\cite{HamadeneLepeltier2000} considered obstacles that are right-continuous and left-upper semicontinuous and investigated the connection with zero-sum stochastic games. Matoussi \cite{Matoussi1997} studied the well-posedness of \eqref{intro2} when the generator $f$ is continuous and satisfies a linear-growth condition. Hamadène \cite{Hamadene2002} examined RBSDEs with right-continuous and left-limited (RCLL) obstacles and related the resulting equations to stochastic mixed control problems.
	
	A second line of research concerns the nature of the underlying filtration. Beyond the classical Brownian framework, Hamadène and Ouknine
	\cite{HamadeneOuknine2003} studied RBSDEs in a filtration generated by a Brownian motion and an independent Poisson process. In their setting, the obstacle is RCLL and its jumps are restricted to the totally inaccessible jump times generated by the Poisson component. This framework was subsequently extended in
	\cite{Essaky2008,HamadeneOuknine2016}
	to RCLL obstacles whose jumps may be either predictable or totally inaccessible. Further related contributions include
	\cite{ElOtmani2006,ElOtmani2009,EssakyHassani2011,EssakyHassani2013,EssakyHassaniOuknine2013,RenElOtmani2010}.
	
	More recently, Elmansouri and El Otmani
	\cite{ElmansouriElOtmani2026}
	studied RBSDEs in a more general filtration $\mathbb{F}$ supporting a square-integrable RCLL martingale $N$. In this framework, the reflected equation takes the form
	\begin{equation}\label{intro3}
		\left\{
		\begin{aligned}
			\textnormal{(i)}\quad
			&Y_t
			=
			\xi_T
			+\int_t^T f(s,Y_s,Z_s)\,d\langle N\rangle_s
			+(K_T-K_t)
			-\int_t^T Z_s\,dN_s
			-\int_t^T dR_s,
			\\
			\textnormal{(ii)}\quad
			&Y_t
			\geq \xi_t,
			\qquad 0\leq t\leq T,\quad \textnormal{a.s.},
			\\
			\textnormal{(iii)}\quad
			&\int_0^T
			(Y_{s-}-\xi_{s-})\,dK_s
			=0,
			\qquad \textnormal{a.s.}
		\end{aligned}
		\right.
	\end{equation}
	Here, $\langle N\rangle$ denotes the predictable quadratic variation, or angle-bracket process, of $N$\footnote{Equivalently, $\langle M\rangle$ is the dual predictable projection of the quadratic variation process $[N]$.}, while $R$ is a square-integrable martingale strongly orthogonal to $N$. The additional martingale $R$ arises naturally from the Galtchouk--Kunita--Watanabe decomposition; see, for instance,
	\cite{AnselStricker1993}.
	
	Under the assumption that the filtration $\mathbb{F}$ is quasi-left-continuous, martingales cannot have jumps at predictable stopping times; see, for example,
	\cite[Chapter V, Theorem 36]{he2019semimartingale}. In particular, the predictable quadratic variation process $\langle N\rangle$ is continuous; see
	\cite[p.~186]{he2019semimartingale}. Under this assumption, the authors established the existence and uniqueness of a solution
	$(Y,Z,K,R)$ to \eqref{intro3}. They also applied their results to the pricing of American options in a financial market driven by the Azéma martingale under asymmetric information. This framework therefore extends several previous results obtained in Brownian and Brownian--Poisson settings, whose natural filtrations are quasi-left-continuous, including
	\cite{ElKarouiKapoudjianPardouxPengQuenez1997,Essaky2008,HamadeneLepeltier2000,HamadeneOuknine2016,Matoussi1997,NieRutkowski2021,NieRutkowski2022} among others.
	
	A further important extension of the theory concerns obstacles that are
	not necessarily right-continuous. In a Brownian--Poisson framework,
	Grigorova et al. \cite{Miryana} introduced RBSDEs with a lower obstacle
	$\xi$ that is right-upper semicontinuous but not necessarily
	right-continuous. More precisely, $\xi$ is right-upper semicontinuous
	along stopping times if, for every stopping time $\tau$ and every
	nonincreasing sequence of stopping times
	$(\tau_n)_{n\in\mathbb{N}}$ such that
	$\tau_n\downarrow\tau$ almost surely, one has
	\[
	\xi_\tau
	\geq
	\limsup_{n\rightarrow+\infty}\xi_{\tau_n},
	\qquad \text{a.s.}
	\]
	
	In this setting, the reflecting mechanism involves two nondecreasing
	processes. In addition to the predictable right-continuous process
	controlling the continuous part and the left jumps of the reflection,
	one must introduce a right-continuous, adapted, nondecreasing, purely
	discontinuous process $C$ that accounts for the right jumps of the
	solution. More precisely, the reflected equation contains the term
	$C_{T-}-C_{t-}$, and its jumps satisfy
	\[
	\Delta C_t
	=
	Y_t-Y_{t+}.
	\]
	This additional process ensures that the constraint $Y\geq\xi$ remains
	satisfied when the obstacle is not right-continuous. The authors proved
	existence and uniqueness of the solution and characterized the value
	process of a nonlinear optimal stopping problem in terms of the first
	component of the associated RBSDE. Further contributions concerning
	RBSDEs with obstacles that are not necessarily right-continuous and
	their connections with optimal stopping and stochastic control include
	\cite{AkdimHaddadiOuknine2020,ArharasOuknine2024,BaadiOuknine2017,BaadiOuknine2018,BouhadouHilbertOuknine2022,BouhadouOuknine2021,BouhadouHilbertOuknine2023,
		Elmansouri2025,ElmansouriElOtmani2026OptimalControl,grigorova2020optimal,klimsiak2019reflected,Klimsiak2012,Klimsiak2015,
		KobylanskiQuenez2012,Marzougue2021,MarzougueElOtmani2019}.
	
	In the present work, we study generalized backward stochastic
	differential equations (GBSDEs) and generalized reflected backward
	stochastic differential equations (GRBSDEs) on a general filtered
	probability space
	\[
	(\Omega,\mathcal{F},\mathbb{F},\mathbb{P}),
	\qquad
	\mathbb{F}=(\mathcal{F}_t)_{0\leq t\leq T},
	\]
	where the filtration is fixed but otherwise arbitrary. In particular,
	we do not assume that $\mathbb{F}$ is quasi-left-continuous. We only
	require that it satisfy the usual conditions of right-continuity and
	completeness.
	
	Let $A=(A_t)_{0\leq t\leq T}$ be a given predictable, RCLL,
	nondecreasing process, which plays the role of a stochastic clock, and
	let $\xi_T$ be a terminal condition. Given a generator
	$f(\omega,t,y)$, we first consider the GBSDE associated with the data
	$(\xi_T,f,A)$:
	\begin{equation}\label{intro4}
		Y_t
		=
		\xi_T
		+\int_{t}^T f(s,Y_{s-})\,dA_s
		-\left(M_T-M_t\right),
		\qquad t\in[0,T].
	\end{equation}
	Here, $Y$ is an $\mathbb{F}$-adapted RCLL and $M$ is an RCLL square-integrable
	$\mathbb{F}$-martingale.
	
	The formulation \eqref{intro4} is intrinsic to the underlying
	filtration and does not require the prior specification of a driving
	martingale enjoying a martingale representation property. In
	particular, since the generator depends only on the state variable
	$y$, the martingale component may be treated directly as an unknown
	square-integrable martingale.
	
	Nevertheless, suppose that the filtration $\mathbb{F}$ has the strong
	predictable representation property with respect to a given
	square-integrable martingale $N$; see, for instance,
	\cite[Chapter XIII]{he2019semimartingale}. This means that every
	square-integrable $\mathbb{F}$-martingale $\mathsf{M}$ admits a
	representation of the form
	\[
	\mathsf{M}_t
	=
	\mathsf{M}_0+\int_0^t Z_s\,dN_s,
	\qquad t\in[0,T],
	\]
	for some $\mathbb{F}$-predictable process $Z$ satisfying
	\[
	\mathbb{E}\left[
	\int_0^T |Z_s|^2\,d\langle N\rangle_s
	\right]
	<+\infty.
	\]
	In this case, the martingale component of \eqref{intro4} may be
	represented as an integral with respect to $N$, and the equation becomes
	\begin{equation}\label{intro5}
		Y_t
		=
		\xi_T
		+\int_{t}^T f(s,Y_{s-})\,dA_s
		-\int_{t}^T Z_s\,dN_s,
		\qquad t\in[0,T].
	\end{equation}
	
	Among the works dealing with BSDEs beyond the
	quasi-left-continuous setting and with a discontinuous stochastic
	driver, we mention Bandini \cite{bandini2015existence}. The author
	considered a BSDE driven by an integer-valued random measure
	$\mu(dt,dx)$ on $[0,T]\times E$, where $E$ is a Lusin space and
	$\mathbb{F}$ is the natural filtration of $\mu$\footnote{That is, the
		smallest right-continuous filtration with respect to which $\mu$ is
		optional. Under the assumptions imposed in
		\cite{bandini2015existence}, this filtration satisfies an appropriate
		martingale representation theorem.}. The compensator of $\mu$ is
	assumed to admit the disintegration
	\[
	\nu(dt,dx)
	=
	\phi_t(dx)\,dA_t,
	\]
	where $A$ is a predictable, RCLL, nondecreasing process satisfying
	$A_0=0$. The corresponding BSDE takes the form
	\begin{equation}\label{intro6}
		\begin{aligned}
			Y_t
			={}&
			\xi_T
			+\int_{t}^T
			f\bigl(s,Y_{s-},Z_s(\cdot)\bigr)\,dA_s
			\\
			&-
			\int_{t}^T \int_E
			Z_s(x)\,(\mu-\nu)(ds,dx),
			\qquad t\in[0,T].
		\end{aligned}
	\end{equation}
	
	To establish existence and uniqueness for \eqref{intro6}, Bandini
	imposed an additional technical condition relating the jumps of $A$ to
	the Lipschitz constant of the generator with respect to the state
	variable $y$. More precisely, if $K$ denotes this Lipschitz constant with respect to $y$, it is assumed that there exists $\varepsilon\in(0,1)$ such that
	\begin{equation}\label{intro7}
		2K^2
		\left|\Delta_-A_t\right|^2
		\leq
		1-\varepsilon,
		\qquad
		\mathbb{P}\text{-a.s.},\quad
		t\in[0,T],
	\end{equation}
	where
	\[
	\Delta_-A_t:=A_t-A_{t-}
	\]
	denotes the jump of $A$ at time $t$, with the convention
	$\Delta_-A_0=0$. In that framework, one also has
	$\Delta_-A_t\leq1$. Conditions of this type also appear, in different
	forms and levels of generality, in
	\cite{CohenElliott2012,confortola2016backward,
		possamai2024reflections}.
	
	The relevance of such jump conditions is illustrated by the
	counterexamples in
	\cite[Section 4.3]{confortola2016backward}, which show that a BSDE with
	a discontinuous integrator may fail to admit a solution or may admit
	more than one solution when the corresponding jump-wise solvability
	condition is violated. These difficulties are related to the explicit
	dependence of the generator $f$ on the jump component value $Z_{t}$. 
	For other relevant contributions to BSDEs in general filtrations and related decomposition results, we refer to
	\cite{BouchardPossamaiTan2016,CarboneFerrarioSantacroce2008,CeciCretarolaRusso2014,PapapantoleonPossamaiSaplaouras2018}.

	The study of GBSDE of the form \eqref{intro4} is directly motivated by credit-risk models using progressive enlargements of filtrations. In this setting, one begins with a reference filtration $\mathbb{F}=(\mathcal{F}_t)_{t\leq T}$, representing the information generated by the underlying market factors, and then enlarges it by a larger filtration $\mathbb{G}=(\mathcal{G}_t)_{t\leq T}$, by progressively adding the observation of one or several random times, usually interpreted as default times (see  \cite{JeanblancLi2020} for more study on default times). In \cite{kharroubi2014progressive}, BSDEs with random marked jumps, driven by the enlarged filtration $\mathbb{G}$, are related to a recursive system of BSDEs in the reference filtration $\mathbb{F}$. Their approach is based on a decomposition of $\mathbb{G}$-predictable and $\mathbb{G}$-optional processes into families of $\mathbb{F}$-adapted processes indexed by the successive default times and their marks. Under the density hypothesis imposed there, the compensator of the marked point process is absolutely continuous with respect to the Lebesgue measure and has the form $ \nu(dt,de)=\lambda_t(e)\,de\,dt.$ Thus the effect of the random times enters in the reduced equations through finite-variation terms generated by the compensator. A related reduction procedure is developed in \cite{crepey2016counterparty} in the context of counterparty and funding risk. The valuation problem is first set up as a BSDE in the full filtration $\mathbb{G}$ over a random horizon given by the first default time of the bank or its counterparty. Under suitable assumptions , we reduce the original equation to a BSDE in the smaller filtration $\mathbb{F}$, on the deterministic horizon $T$. The reduced driver captures the essential components of counterparty-risk valuation, such as the default intensity, the close-out exposure, recovery payments, collateral and non-linear funding costs.The reduction is especially important when the market exposure and the default risk are correlated as it is the case for wrong-way risk and gap-risk models. These constructions suggest to replace the deterministic drift $dt$ by a more general predictable stochastic driver $dA_t$. Indeed, it is well known that the compensator is not necessarily absolutely continuous with respect to the Lebesgue measure in a general marked point process model. It may instead permit a disintegration of the form
$
		\nu(dt,dx)=\phi_t(dx)\,dA_t,
$
	where $A$ is a predictable, nondecreasing, RCLL process. After the
	equation is projected or reduced to the reference filtration, the terms arising
	from the compensator are then naturally integrated with respect to $dA_t$. This
	leads precisely to a BSDE of the form \eqref{intro4}. 
	
	To clarify this point, let $\tau$ be a random time that is not necessarily a $\mathbb{F}$-stopping time and let $\mathbb{G}$ be the progressive enlargement of $\mathbb{F}$ with $\tau$. We consider a reduced BSDE in the reference filtration $\mathbb{F}$ containing both an ordinary time component and a term driven by a $\mathbb{F}$-predictable cumulative hazard process $\Gamma$:
	\begin{equation}\label{reduced-bsde-two-clocks}
		Y_t
		=
		\xi_T
		+
		\int_t^T g\bigl(s,Y_{s-}\bigr)\,ds
		+
		\int_t^T h\bigl(s,Y_{s-}\bigr)\,d\Gamma_s
		-
		\bigl(M_T-M_t\bigr),
		\qquad t\in[0,T].
	\end{equation}
	Here, $\xi_T$ is $\mathcal{F}_T$-measurable, $M$ is an
	$\mathbb{F}$-martingale, and $\Gamma$ is an $\mathbb{F}$-predictable,
	nondecreasing finite-variation process with $\Gamma_0=0$.
	
	Equation \eqref{reduced-bsde-two-clocks} can be seen as a generalization of the intensity-based frameworks studied in \cite{crepey2016counterparty,kharroubi2014progressive}. In these environments, the compensators corresponding to the relevant default times or marked point processes are absolutely continuous with respect to Lebesgue measure. So for a single default time the form of the cumulative hazard process is
	\[
	\Gamma_t=\int_0^t \lambda_s\,ds,
	\]
	where $\lambda$ denotes the corresponding intensity process, and hence
	$
	d\Gamma_t=\lambda_t\,dt.
	$
	The formulation in \eqref{reduced-bsde-two-clocks} allows $(\Gamma_t)_{t \leq T}$ to have a
	more general predictable finite-variation structure, including a singular
	continuous component or predictable jumps. We now introduce the common increasing driver
	\[
	A_t=t+\Gamma_t,
	\qquad t\in[0,T].
	\]
	Since both the Lebesgue measure $dt$ and the Stieltjes measure $d\Gamma_t$ are
	absolutely continuous with respect to $dA_t$, there exist nonnegative
	$\mathbb{F}$-predictable processes $a=(a_t)_{t\in[0,T]}$ and
	$b=(b_t)_{t\in[0,T]}$ such that
	\[
	dt=a_t\,dA_t,
	\qquad
	d\Gamma_t=b_t\,dA_t.
	\]
	Equivalently, $a$ and $b$ are the Radon--Nikodym densities
	$
	a_t=\frac{dt}{dA_t},$ and $
	b_t=\frac{d\Gamma_t}{dA_t}.
	$
	Moreover, since
	$
	dA_t=dt+d\Gamma_t,
	$
	we have
	$
	a_t+b_t=1,~ dA_t\text{-a.e.}
	$	Consequently, equation \eqref{reduced-bsde-two-clocks} can be rewritten as
	\[
	Y_t
	=
	\xi_T
	+
	\int_t^T
	\left[
	a_s g\bigl(s,Y_{s-}\bigr)
	+
	b_s h\bigl(s,Y_{s-}\bigr)
	\right]dA_s
	-
	\bigl(M_T-M_t\bigr).
	\]
	Defining
	\[
	f(s,y)
	=
	a_s g(s,y)
	+
	b_s h(s,y),
	\]
	we recover equation \eqref{intro4}, whose well-posedness is established in
	Theorem~\ref{GBSDE-general_thm} below.
	
	This common-driver formulation offers a handy approach to address, all in one equation, both the usual time evolution and the finite-variation effects brought about by the default mechanism. When $\Gamma\equiv0$, we revert to the classical BSDE \eqref{intro1}, where $A_t=t$ and $M_t=\int_{0}^{t}Z_s dB_s$ in the natural filtration of the Brownian motion $B$. In the context of absolutely continuous intensity,
	$$
	\Gamma_t=\int_0^t\lambda_s\,ds,
	$$
	and therefore
     $$
	A_t=t+\int_0^t\lambda_s\,ds,
	\qquad
	dA_t=(1+\lambda_t)\,dt.
	$$
	In this case, one may take
	$$
	a_t=\frac{1}{1+\lambda_t},
	\qquad
	b_t=\frac{\lambda_t}{1+\lambda_t},
	$$
	so that
	$$
	f(t,y)
	=
	\frac{g(t,y)+\lambda_t h(t,y)}
	{1+\lambda_t}.
	$$
	
	Let us also point out that, in a market subject to default, such as the
	one considered in \cite{DumitrescuGrigorovaQuenezSulem2018}, the linear
	pricing of a European contingent claim can be reduced to the study of a
	linear BSDE with a default jump whose generator depends only on the
	state variable $y$. This formulation is discussed in
	\cite[Subsection~5.2.1]{Elmansouri2026BSDEs}; see also
	\cite[Proposition~5.6]{Elmansouri2026BSDEs}. This provides further
	motivation for studying GBSDEs of the form \eqref{intro5} directly in
	the progressive enlargement $\mathbb{G}$ of the reference filtration
	$\mathbb{F}$ with a default time $\tau$. For further contributions on
	the financial applications of BSDEs and RBSDEs with default, we refer the reader
	to
	\cite{AksamitLiRutkowski2023,BlanchetScallietEyraudLoiselRoyerCarenzi2010,DumitrescuQuenezSulem2018,ElmansouriElOtmani2026,
		EyraudLoiselRoyerCarenzi2010,GrigorovaQuenezSulem2020,
		LiangLyonsQian2011}
	and the references therein.
	
	The second main objective of this paper is to study the generalized
	reflected backward stochastic differential equation (GRBSDE)
	associated with the data $(\xi,f,A)$ and given by
	\begin{equation}\label{intro8}
		\left\{
		\begin{aligned}
			\textnormal{(i)}\quad
			&Y_t
			=
			\xi_T
			+\int_{t}^T f(s,Y_{s-})\,dA_s
			+(K_T-K_t)
			-(M_T-M_t),
			\\
			\textnormal{(ii)}\quad
			&Y_t
			\geq \xi_t,
			\qquad 0\leq t\leq T,\quad \textnormal{a.s.},
			\\
			\textnormal{(iii)}\quad
			&\int_0^T
			\left(Y_{s-}-\xi_{s-}\right)\,dK_s^{\ast}
			+
			\sum_{0\leq s<T}
			\left(Y_s-\xi_s\right)\Delta_+K_s
			=0,
			\qquad \textnormal{a.s.}
		\end{aligned}
		\right.
	\end{equation}
	Here, $\xi$ is the lower obstacle, while $K$ is the reflecting process
	satisfying the Skorokhod minimality condition
	\eqref{intro8}--\textnormal{(iii)}. This condition ensures that the
	constraint \eqref{intro8}--\textnormal{(ii)} is maintained and that
	$K$ acts only when the solution is in contact with the obstacle.
	
	Throughout the paper, $A=(A_t)_{0\leq t\leq T}$ is a prescribed
	predictable, RCLL, nondecreasing process, whereas $\xi$ is an
	$\mathbb{F}$-optional regulated process. Recall that a regulated
	process has finite left and right limits at every point of $[0,T]$;
	see Definition \ref{regulated process} below. This irregularity of the
	obstacle naturally leads to regulated state and reflecting processes
	$Y$ and $K$. In particular, the right jumps of $K$ control the right
	jumps of $Y$ through the relation
	\[
	\Delta_+Y_t=-\Delta_+K_t,
	\qquad 0\leq t<T.
	\]
	
	The GRBSDE \eqref{intro8} contains, at the level of its intrinsic
	martingale formulation, several previously studied models as special
	cases. Indeed, when $A$ is continuous and $\xi$ is RCLL, the
	right-jump component of the reflecting process vanishes, that is,
	$\Delta_+K=0$; see, for instance,
	\cite[Remark~3.3]{Miryana}. Moreover, by taking
	\[
	A=\langle N\rangle
	\]
	and identifying the martingale component in \eqref{intro8} with
	\[
	\int_0^\cdot Z_s\,dN_s+R,
	\]
	one recovers a formulation of the type \eqref{intro3} when the
	generator is independent of the martingale integrand. Under
	quasi-left-continuity of the filtration and continuity of
	$\langle N\rangle$, the corresponding framework may also accommodate
	generators depending on the $z$-variable.
	
	Among the closest related contributions, we mention
	\cite{foresta2021optimal} and \cite{li2024penalization}, where more
	general forms of reflected equations are considered, but under the
	assumption that the increasing stochastic driver $A$ is continuous. In
	\cite{foresta2021optimal}, the author studies equations driven by a
	continuous increasing process associated with a point process, in a
	filtration generated by a marked point process on a Borel space and a
	Brownian motion. More recently, Li et al.
	\cite{li2024penalization} considered generalized reflected BSDEs on a
	general filtered probability space, again assuming that the process
	$A$ is continuous. Although continuity eliminates several technical
	difficulties caused by the jumps of the stochastic driver, the resulting
	problem remains highly nontrivial. In the present work, we allow $A$ to
	be a general predictable RCLL nondecreasing process.
	
	Our proof of existence and uniqueness is divided into two main steps.
	We first consider a generator that is independent of the state
	variable $y$. In this case, we develop two complementary approaches.
	The first relies on the Snell envelope and the theory of optimal
	stopping, whereas the second is based on a modified penalization
	procedure that explicitly takes into account the right jumps of the
	obstacle. Our approach therefore differs from that of
	\cite{li2024penalization}, where a classical penalization scheme is
	used in connection with an associated optimal stopping problem.
	
	As in \cite{li2024penalization}, we assume that the stochastic driver
	$A$ is bounded; see Assumption \textbf{(H3)} and Example
	\ref{exemple}, where several relevant situations satisfying this
	condition are presented. The arguments may also be adapted to certain
	square-integrable, but not necessarily bounded, processes $A$, provided
	that additional structural assumptions are imposed on the generator.
	This includes, in particular, generators of class \textnormal{(P)} in
	the sense of \cite[Definition~3.13]{li2024penalization}. We do not
	pursue this extension in the present paper.
	
	Let us explain why a direct application of the classical penalization
	method is problematic when the stochastic driver $A$ has jumps.
	Consider, for simplicity, an RCLL obstacle $\xi$ and a BSDE driven by
	an integer-valued random measure. If the generator also depends on a
	martingale integrand $z$, one is naturally led to consider a penalized
	equation of the form\footnote{As in the work by Fakhouri and Ouknine \cite{fakhouri20192}.}
	\begin{equation}\label{intro9}
		\begin{aligned}
			Y_t^n
			={}&
			\xi_T
			+\int_{(t,T]}
			\left\{
			f\bigl(s,Y_{s-}^n,Z_s^n(\cdot)\bigr)
			+n\bigl(Y_{s-}^n-\xi_{s-}\bigr)^-
			\right\}\,dA_s
			\\
			&-
			\int_{(t,T]}\int_E
			Z_s^n(x)\,(\mu-\nu)(ds,dx),
			\qquad n\geq1,\quad t\in[0,T].
		\end{aligned}
	\end{equation}
	Suppose that $f$ is Lipschitz continuous with respect to $y$, with
	Lipschitz constant $K$. The penalized generator in
	\eqref{intro9} then has a Lipschitz constant with respect to $y$ that
	is bounded by $K+n$. Consequently, applying a jump condition of the
	type used in
	\cite{bandini2015existence,confortola2016backward}
	would require the existence of $\varepsilon\in(0,1)$ such that
	\begin{equation}\label{intro10}
		2(K+n)^2
		\left|\Delta_-A_t\right|^2
		\leq
		1-\varepsilon,
		\qquad
		\mathbb{P}\text{-a.s.},\quad
		t\in[0,T].
	\end{equation}
	Equivalently,
	\[
	\left|\Delta_-A_t\right|
	\leq
	\frac{\sqrt{1-\varepsilon}}
	{\sqrt{2}\,(K+n)},
	\qquad
	\mathbb{P}\text{-a.s.},\quad
	t\in[0,T].
	\]
	If such an estimate were required for every $n\geq1$, then letting
	$n\rightarrow+\infty$ would yield
	\[
	\Delta_-A_t=0,
	\qquad
	\mathbb{P}\text{-a.s.},\quad t\in[0,T],
	\]
	and hence $A$ would necessarily be continuous. Such a conclusion is
	incompatible with the general framework considered here. For example,
	when $T>1$, one may take
	\[
	A_t:=\mathds{1}_{\{t\geq1\}},
	\qquad t\in[0,T],
	\]
	which is a predictable, bounded, nondecreasing, purely discontinuous
	RCLL process.
	
	This observation shows that the classical penalization scheme cannot
	be combined directly with an $n$-dependent Lipschitz condition of the
	form \eqref{intro10} when $A$ has nonzero jumps. It therefore motivates
	the modified approximation procedure developed in this paper. In the
	case of a generator independent of the state variable, we construct a
	penalization scheme adapted to the discontinuous stochastic driver and
	to the right jumps of the obstacle. Under the assumptions introduced
	below, this scheme avoids the $n$-dependent jump restriction appearing
	in \eqref{intro10}. The general Lipschitz case, in which the generator
	depends on $y$, is subsequently treated by means of a fixed-point
	argument in an appropriate Banach space.
	
	Another important feature of our analysis is the choice of the
	weighted spaces. Since $A$ is only RCLL, we work with weights defined
	through the stochastic exponential
	\[
	\mathcal{E}(\beta A),
	\]
	which is the solution of the forward equation
	\[
	d\mathcal{E}(\beta A)_t
	=
	\beta\mathcal{E}(\beta A)_{t-}\,dA_t,
	\qquad
	\mathcal{E}(\beta A)_0=1.
	\]
	When $A$ is continuous, this stochastic exponential reduces to the
	ordinary exponential $e^{\beta A}$. By contrast, when $A$ has jumps,
	the stochastic exponential incorporates the jump factors
	$1+\beta\Delta_-A$. This constitutes a significant difference from
	the framework of \cite{li2024penalization}, where the continuity of
	$A$ allows the authors to work with a classical exponential weight.
	
	The remainder of the paper is organized as follows. Section
	\ref{sec2} is devoted to the preliminary material, the main
	assumptions, and the auxiliary results needed throughout the paper. In
	Section \ref{sec3}, we study GBSDEs of the form \eqref{intro4} under a
	Lipschitz condition on the generator and suitable square-integrability
	assumptions on the data. Finally, Section \ref{sec4} is devoted to the
	well-posedness of the GRBSDE \eqref{intro8}. We first address the case
	of a generator independent of the state variable and establish a
	characterization in terms of a Snell envelope and an associated optimal
	stopping problem. We then introduce a modified penalization
	approximation, combining the results of Section \ref{sec3} with
	optimal stopping techniques. The general case of a generator depending
	on $y$ is obtained by applying a fixed-point argument in an appropriate
	weighted Banach space.

	\section{Preliminaries}\label{sec2}
	
	Let $T\in(0,+\infty)$ be a fixed deterministic time horizon. We work on
	a complete probability space $(\Omega,\mathcal{F},\mathbb{P})$ endowed
	with a filtration
	\[
	\mathbb{F}:=(\mathcal{F}_t)_{0\leq t\leq T}
	\]
	satisfying the usual conditions of right-continuity and completeness.
	Let $A$ be a predictable, RCLL, nondecreasing process such that
	$A_0=0$. We denote by $\mathcal{P}$ and $\mathbf{Prog}$ the predictable
	and progressive $\sigma$-fields on $\Omega\times[0,T]$, respectively.
	
	For $0\leq s<t\leq T$, the notation $\int_s^t$ is understood as
	$\int_{(s,t]}$, both for Stieltjes-type integrals and for It\^o
	integrals of predictable processes with respect to RCLL real-valued
	martingales.
	
	We denote by $\mathcal{T}_{[0,T]}$ the set of all
	$\mathbb{F}$-stopping times taking values in $[0,T]$. For
	$\eta_1,\eta_2\in\mathcal{T}_{[0,T]}$ such that
	$\eta_1\leq\eta_2$ almost surely, we set
	\[
	\mathcal{T}_{[\eta_1,\eta_2]}
	:=
	\left\{
	\tau\in\mathcal{T}_{[0,T]}:
	\eta_1\leq\tau\leq\eta_2
	\ \textnormal{a.s.}
	\right\}.
	\]
	
	For brevity, we suppress the dependence on $\omega$ of generators and
	other random coefficients whenever no confusion may arise. Unless
	otherwise stated, equalities and inequalities between stochastic
	processes are understood up to evanescence.
	
	We begin by recalling several definitions that will be used throughout
	the paper.
	
	\begin{definition}[Strong supermartingale]\label{o11}
		An $\mathbb{F}$-optional process $\mathsf{Y}$ is called a
		\emph{strong supermartingale} if $\mathsf{Y}_\tau$ is integrable for
		every $\tau\in\mathcal{T}_{[0,T]}$ and
		\[
		\mathsf{Y}_\eta
		\geq
		\mathbb{E}\left[
		\mathsf{Y}_\tau
		\,\middle|\,
		\mathcal{F}_\eta
		\right]
		\qquad\textnormal{a.s.}
		\]
		for every $\eta,\tau\in\mathcal{T}_{[0,T]}$ satisfying
		$\eta\leq\tau$ almost surely. Here,
		\[
		\mathcal{F}_\eta
		:=
		\left\{
		B\in\mathcal{F}:
		B\cap\{\eta\leq t\}\in\mathcal{F}_t
		\textnormal{ for every }t\in[0,T]
		\right\}.
		\]
	\end{definition}
	
	Definition \ref{o11} is taken from
	\cite[p.~407]{DellacherieMeyer1980}.
	
	\begin{definition}[Snell envelope]\label{o22}
		Let $\xi$ be an $\mathbb{F}$-optional process. An optional process
		$\mathsf{Y}$ is called the \emph{Snell envelope} of $\xi$ if the
		following conditions hold:
		\begin{enumerate}
			\item $\mathsf{Y}$ is a strong supermartingale;
			
			\item $\mathsf{Y}$ dominates $\xi$, that is,
			\[
			\mathsf{Y}\geq\xi;
			\]
			
			\item $\mathsf{Y}$ is the smallest strong supermartingale
			dominating $\xi$. More precisely, if
			$\widehat{\mathsf{Y}}$ is any strong supermartingale satisfying
			$\widehat{\mathsf{Y}}\geq\xi$, then
			\[
			\widehat{\mathsf{Y}}\geq\mathsf{Y}.
			\]
		\end{enumerate}
	\end{definition}
	
	Definition \ref{o22} is taken from
	\cite[p.~113]{ElKaroui1981}.
	
	\begin{definition}[Regulated processes]\label{regulated process}
		\begin{itemize}
			\item A process
			\[
			\mathcal{X}:\Omega\times[0,T]\longrightarrow\mathbb{R}
			\]
			is called a \emph{regulated process} if, for
			$\mathbb{P}$-almost every $\omega\in\Omega$, the mapping
			\[
			t\longmapsto\mathcal{X}_t(\omega)
			\]
			admits a finite right limit at every $t\in[0,T)$ and a finite
			left limit at every $t\in(0,T]$.
			
			\item Let $\mathcal{X}$ be a process with regulated trajectories.
			For $s\in(0,T]$, we define its left limit and left jump by
			\[
			\mathcal{X}_{s-}
			:=
			\lim_{u\uparrow s}\mathcal{X}_u,
			\qquad
			\Delta_-\mathcal{X}_s
			:=
			\mathcal{X}_s-\mathcal{X}_{s-},
			\]
			with the convention $\Delta_-\mathcal{X}_0=0$.
			
			For $s\in[0,T)$, we define its right limit and right jump by
			\[
			\mathcal{X}_{s+}
			:=
			\lim_{u\downarrow s}\mathcal{X}_u,
			\qquad
			\Delta_+\mathcal{X}_s
			:=
			\mathcal{X}_{s+}-\mathcal{X}_s,
			\]
			with the convention $\Delta_+\mathcal{X}_T=0$.
			
			The trajectories of a regulated process have at most countably
			many discontinuities; see
			\cite[Corollary~II.2.2]{DudleyNorvaisa2011}.
			
			\item Let
			\[
			\mathsf{K}:\Omega\times[0,T]\longrightarrow\mathbb{R}
			\]
			be a regulated process of finite variation. Then $\mathsf{K}$
			admits the decomposition
			\[
			\mathsf{K}
			=
			\mathsf{K}^c+\mathsf{K}^d+\mathsf{K}^g,
			\]
			where $\mathsf{K}^c$ is continuous, the càdlàg pure-jump process
			$\mathsf{K}^d$ is defined by
			\[
			\mathsf{K}^d_t
			:=
			\sum_{0<s\leq t}\Delta_-\mathsf{K}_s,
			\qquad t\in[0,T],
			\]
			and the càglàd pure-jump process $\mathsf{K}^g$ is defined by
			\[
			\mathsf{K}^g_t
			:=
			\sum_{0\leq s<t}\Delta_+\mathsf{K}_s,
			\qquad t\in[0,T].
			\]
			The above sums converge absolutely along each trajectory since
			$\mathsf{K}$ has finite variation.
			
			In particular,
			\[
			\mathsf{K}_t
			=
			\mathsf{K}^{\ast}_t
			+
			\sum_{0\leq s<t}\Delta_+\mathsf{K}_s,
			\qquad t\in[0,T],
			\]
			where
			\[
			\mathsf{K}^{\ast}
			:=
			\mathsf{K}-\mathsf{K}^g
			=
			\mathsf{K}^c+\mathsf{K}^d
			\]
			is the càdlàg, or right-continuous, part of $\mathsf{K}$.
			The process $\mathsf{K}^g$ is the left-continuous pure-jump part
			of $\mathsf{K}$.
		\end{itemize}
	\end{definition}
	
	Definition \ref{regulated process} is adapted from
	\cite[Definition~1.4]{Galchuk1982}.
	
	\begin{remark}
		Let $\xi=\left(\xi_t\right)_{t\leq T}$ be an $\mathbb{F}$-optional process with regulated trajectories. Then, by Theorem 1.14 in \cite{gal1981optional}, there exists a sequence of stopping times $\left\{\sigma_n\right\}_{n\in\mathbb{N}}$ whose graphs cover all right-jump times of $\xi$. More precisely,
		\begin{equation*}
			\left\{(\omega,t)\in\Omega\times[0,T):\Delta_{+}\xi_t(\omega)\neq 0\right\}
			=
			\bigcup_{n\in\mathbb{N}}\llbracket \theta_n\rrbracket,
		\end{equation*}
		where
		$
		\llbracket \sigma_n\rrbracket
		=
		\left\{(\omega,t)\in\Omega\times[0,T]:\sigma_n(\omega)=t\right\}.
		$
		\label{existence}
	\end{remark}
	
	For a given RCLL semimartingale $\mathcal{X}$, we denote by
	$\mathcal{E}(\mathcal{X}):=\left(\mathcal{E}(\mathcal{X})_t\right)_{t \leq T}$
	the Doléans--Dade stochastic exponential of $\mathcal{X}$.
	
	Let $\beta \geq 0$. By Theorem 37 in \cite[p.~84]{Protter2004}, 
	$\mathcal{E}(\beta A)$ is the unique right-continuous adapted process satisfying
	$$
	\mathcal{E}(\beta A)_t
	=
	1+\beta \int_{0}^{t}\mathcal{E}(\beta A)_{s-}\,dA_s,
	\quad t \in [0,T], \quad \text{a.s.}
	$$
	Since $A$ is of finite variation, the stochastic exponential admits the following closed-form representation:
	$$
	\mathcal{E}(\beta A)_t
	=
	e^{\beta A_t}
	\prod_{0<s \leq t}
	\left(1+\beta\Delta A_s\right)e^{-\beta\Delta A_s},
	\quad t \in [0,T], \quad \text{a.s.}
	$$
	
	\begin{remark}\label{rmq4}
		\begin{itemize}
		\item[(i)] Since $\beta \Delta A_t \geq 0$ for all $t \in [0,T]$, a.s., and for every $\beta \geq 0$, we have
		$\mathcal{E}(\beta A)_t \geq 1$ for all $t \in [0,T]$, a.s. Moreover, using
		$$
		\Delta \mathcal{E}(\beta A)_{t}
		=
		\beta \mathcal{E}(\beta A)_{t-}\Delta A_t,
		$$
		we obtain
		$$
		\mathcal{E}(\beta A)_{t-}
		\leq
		\mathcal{E}(\beta A)_{t},
		\quad t \in [0,T], \quad \text{a.s.}
		$$
		Thus, $\mathcal{E}(\beta A)$ is an RCLL predictable\footnote{A consequence of  Proposition 3.5 in \cite[p. 28]{JacodShiryaev2013}.} increasing process. In particular, it is locally bounded\footnote{Let
			$\tau_n:=\inf\{t\in(0,T]: \mathcal{E}(\beta A)_t\ge n\}$. Since $\mathcal{E}(\beta A)$ is increasing, we have
			$\mathcal{E}(\beta A)_t\le n$ on $[0,\tau_n)$, and $\tau_n \nearrow T$. Hence $\mathcal{E}(\beta A)$ is locally bounded. Consequently, for any RCLL optional increasing process $\Gamma$, we have
			$\int_{0}^{\tau_n}\mathcal{E}(\beta A)_s\,d\Gamma_s \le n\int_{0}^{\tau_n} d\Gamma_s$, up to the usual convention at the terminal jump.}. Furthermore, applying the change-of-variables formula, see Theorem 31 in \cite[p.~78]{Protter2004}, yields
		\begin{equation}\label{promis}
			\begin{split}
				\frac{1}{\mathcal{E}(\beta A)_t}
				=
				1
				-\beta\int_{0}^{t}
				\frac{1}{\mathcal{E}(\beta A)_{s-}}\,dA_s
				+\sum_{0<s\leq t}
				\left\{
				\frac{1}{\mathcal{E}(\beta A)_{s}}
				-\frac{1}{\mathcal{E}(\beta A)_{s-}}
				+\beta \frac{\Delta A_s}{\mathcal{E}(\beta A)_{s-}}
				\right\}.
			\end{split}
		\end{equation}
		Since
		$$
		\mathcal{E}(\beta A)_{s}
		=
		\mathcal{E}(\beta A)_{s-}\left(1+\beta \Delta A_s\right),
		$$
		or equivalently
		\begin{equation}\label{matla}
		\mathcal{E}(\beta A)_{s-}
		=
		\frac{\mathcal{E}(\beta A)_{s}}{1+\beta \Delta A_s},
		\end{equation}
		we get
		$$
		\frac{1}{\mathcal{E}(\beta A)_{s}}
		-\frac{1}{\mathcal{E}(\beta A)_{s-}}
		+\beta \frac{\Delta A_s}{\mathcal{E}(\beta A)_{s-}}
		=
		\beta^2
		\frac{1}{\mathcal{E}(\beta A)_{s}}
		\left(\Delta A_s\right)^2.
		$$
		Moreover,
		$$
		\beta
		\frac{1}{\mathcal{E}(\beta A)_{s-}}\,dA_s
		=
		\beta
		\frac{1}{\mathcal{E}(\beta A)_{s}}\,dA_s
		+
		\beta^2
		\frac{1}{\mathcal{E}(\beta A)_{s}}
		\left(\Delta A_s\right)^2.
		$$
		Substituting these identities into \eqref{promis}, we obtain
		\begin{equation}\label{code}
			\int_{0}^{t}
			\frac{1}{\mathcal{E}(\beta A)_{s}} \, dA_s
			=
			\frac{1}{\beta}
			\left(
			1-\frac{1}{\mathcal{E}(\beta A)_t}
			\right)
			\leq
			\frac{1}{\beta},
			\quad t \in [0,T], \quad \text{a.s.}
		\end{equation}
		

\item[(ii)]  Let $V$ be an RCLL predictable process with integrable variation. Then, for any stopping time $\tau \in \mathcal{T}_{0,T}$, we have $V_\tau \in \mathcal{F}_{\tau-}$, see, for instance, Corollary 3.23 in \cite[p. 91]{he2019semimartingale}; in particular, $\Delta_- V_\tau$ is $\mathcal{F}_{\tau-}$-measurable, and hence $\mathbb{E}\left[\Delta_- V_\tau \mid \mathcal{F}_{\tau-}\right]=\Delta_- V_\tau$. Moreover, if $M$ is an RCLL martingale, then, for every predictable stopping time $\tau \in \mathcal{T}^{p}_{0,T}$, we have $\mathbb{E}\left[\Delta_- M_\tau \mid \mathcal{F}_{\tau-}\right]=0$, see, for instance, Chapter I, Lemma (1.21) in \cite{Jacod1979}. Therefore, since $\Delta_- V_\tau$ is $\mathcal{F}_{\tau-}$-measurable, it follows that $\mathbb{E}\left[\Delta_- V_\tau \Delta_- M_\tau \mid \mathcal{F}_{\tau-}\right]=\Delta_- V_\tau \mathbb{E}\left[\Delta_- M_\tau \mid \mathcal{F}_{\tau-}\right]=0$. Consequently, by summing over the predictable jump times of $V$, and assuming that the series $\sum_{0<s\leq T}\Delta_- V_s \Delta_- M_s$ is integrable, we obtain $\mathbb{E}\left[\sum_{0<s\leq T}\Delta_- V_s \Delta_- M_s\right]=0$.
		\end{itemize}
	\end{remark}

	\paragraph{Assumptions on the data $(\xi,f,A)$:}
We assume that the data $(\xi,f,A)$ of the RGBSDE \eqref{basic equation} satisfy the following conditions. Let $\beta>0$.
\begin{itemize}
	\item[\textbf{(H1)}] The obstacle $\xi=(\xi_t)_{t \leq T}$ is an  $\mathbb{R}$-valued, $\mathbb{F}$-optional, regulated process such that
	$$
	\mathbb{E}\left[\esssup_{\tau \in \mathcal{T}_{[0,T]}}| \mathcal{E}(\beta A)_\tau  \xi_\tau|^2\right]<+\infty.
	$$
	
	\item[\textbf{(H2)}] The generator $f : \Omega \times [0,T] \times \mathbb{R} \rightarrow \mathbb{R}$ satisfies:
	\begin{itemize}
		\item[(i)] The random field
		$
		(\omega,t,y)\longmapsto f(\omega,t,y)
		$
		is
		$
		\mathcal{P}\otimes\mathcal{B}(\mathbb{R})
		\big/
		\mathcal{B}(\mathbb{R})\text{-measurable}.
		$
		\footnote{This measurability assumption is in the spirit of the conditions imposed by Bandini \cite[Assumption~(i)]{bandini2015existence} and Confortola et al. \cite[Condition~(2.8)]{confortola2016backward}. In these works, two classes of equations are considered in the filtration generated by a point process, without assuming quasi-left-continuity, and the corresponding generators satisfy a similar measurability property; see also the discussion on pp.~1747--1748 of \cite{confortola2016backward}. In the classical framework, where the process appearing as the integrator in the Lebesgue--Stieltjes term is continuous and the underlying filtration is quasi-left-continuous, we refer to Pardoux \cite{Pardoux1997} and El Otmani \cite{ElOtmani2006}; see also \cite{elmansouri2023generalized,elmansouri2024generalized}.}
		
		\item[(ii)] There exist a positive constant $K>0$ such that
		$$
		\left|f(\omega,t,y)-f(\omega,t,y')\right| \leq K |y-y'|
		$$
		for all $(\omega,t) \in \Omega$, $y,y' \in \mathbb{R}^2$.
		\item[(iii)] $\mathbb{E}\left[\int_{0}^{T}\mathcal{E}(\beta A)_s |f(s,0)|^2 dA_s\right]<+\infty$.
		
	\end{itemize}
	
\item[\textbf{(H3)}] The process $A$ is bounded, namely there exists a deterministic constant $\mathfrak{C}_A>0$ such that $A_T\leq \mathfrak{C}_A$ {a.s.}

\end{itemize}

\begin{example}\label{exemple}
	There are many examples of RCLL predictable nondecreasing processes with bounded jumps, i.e satisfying assumption \textbf{\textnormal{\textbf{(H3)}}}.
	\begin{enumerate}
		\item A simple deterministic example is
		$$
		A_t:=\mathds{1}_{\{t\geq q\}}, \qquad t\in[0,T],
		$$
		with $q \in [0,T]$ a fixed deterministic time.
		
		\item Any deterministic nondecreasing
		RCLL function defines a predictable process. For instance, by fixing some $N \in \mathbb{N}^\ast$, one may take
		$$
		A_t=t+\mathds{1}_{\{t\geq q\}}, \qquad t\in[0,T],
		$$
		or
		$$
		A_t=C_t+\sum_{n= 1}^N a_n\mathds{1}_{\{t\geq t_n\}},
		\qquad t\in[0,T],
		$$
		where $C$ is a nonnegative deterministic continuous function, $0<t_1<t_2<\cdots\leq T$, and $0 \leq a_n \leq \Lambda$ for some $\Lambda>0$. In this case $A$ is RCLL, predictable,
		nondecreasing, and bounded by $\sup_{0\leq t\leq T} C_t+ N\Lambda$.

		\item Another standard class is given by
		$$
		A_t=\int_0^t C_s\,ds+\sum_{n= 1}^N \eta_n\mathds{1}_{\{t\geq \tau_n\}} ,
		\quad t\in[0,T],\quad N \in \mathbb{N}^\ast,
		$$
		where $C$ is a nonnegative bounded predictable process, $(\tau_n)_{n\geq1}$
		are predictable stopping times, and $\eta_n\geq 0$ are 
		$\mathcal F_{\tau_n-}$-measurable bounded random variables. Under this consideration, 
		the process $A$ is again RCLL, predictable, nondecreasing and 
		bounded.
	\end{enumerate}

	Finally, from assumption \textnormal{\textbf{(H3)}}, we derive that for any $\beta>0$
	\begin{equation}\label{essential condition-r1}
		\frac{1}{1+ \beta \mathfrak{C}_A }\leq \frac{1}{1+\beta \Delta_- A_s} \leq 1
	\end{equation}
\end{example}

	\begin{remark}\label{rmq2}
		\begin{enumerate}
	\item Let $Y$ be any given regulated process. 
	Using the left-continuity of the process $\left(\mathcal{E}(\beta A)_{t-} Y_{t-}\right)_{t \in [0,T]}$ and \eqref{matla} along with assumption \textbf{(H3)}, we get
	\begin{equation*}
		\begin{split}
			\sup_{s \in [0,T]}  \mathcal{E}(\beta A)_s\left|Y_{s-} \right|^2 \leq (1+\beta \mathfrak{C}_A) \sup_{s \in [0,T]}  \mathcal{E}(\beta A)_{s-}\left|Y_{s-} \right|^2
			& = (1+\beta \mathfrak{C}_A)\sup_{s \in \mathbb{Q} \cap [0,T]}  \mathcal{E}(\beta A)_{s-}\left|Y_{s-} \right|^2\\ 
			&\leq (1+\beta \mathfrak{C}_A)\esssup_{\tau \in \mathcal{T}_{[0,T]}} \mathcal{E}(\beta A)_\tau\left|Y_{\tau} \right|^2 \quad \mathbb{P}\textnormal{-a.s.}
		\end{split}
	\end{equation*}
	
	\item Using the closed formula for stochastic exponential (see, e.g., \cite[Ch II. Theorem 37 ]{Protter2004}) and  $A_t=A^c_t+\sum_{0 < s \leq t} \Delta_-A_s$ (see, e.g., \cite[p. 99]{he2019semimartingale}), we can write
	\begin{equation*}
		\begin{split}
			\mathcal{E}(\beta A)_t&=\exp\left\{\beta A_t\right\}\prod_{0 < s \leq t}\left(1+\beta \Delta_-A_s\right)\exp\left\{-\beta \Delta_-A_s\right\}\\
			&=\exp\left\{\beta A^c_t+\beta\sum_{0 < s \leq t}\Delta_-A_s\right\}\prod_{0 < s \leq t}\left(1+\beta \Delta_-A_s\right)\exp\left\{-\beta \Delta_-A_s\right\}\\
			&=\exp\left\{\beta A^c_t\right\}\prod_{0 < s \leq t}\left(1+\beta \Delta_-A_s\right)
		\end{split}
	\end{equation*}
Using 
$$
1+x \leq e^x,\quad \forall x \in \mathbb{R}^+,
$$
we get
$$
\prod_{0 < s \leq t}\left(1+\beta \Delta_-A_s\right) \leq \prod_{0 < s \leq t} \exp\left\{\Delta_-A_s\right\}=\exp\left\{-\beta \sum_{0 < s \leq t} \Delta_- A_s\right\}
$$
Therefore, from assumption \textbf{(H3)}, we derive
\begin{equation*}
	\mathcal{E}(\beta A)_t \leq \exp\left\{\beta A_t\right\} \leq e^{\beta \mathfrak{C}_A}\quad \text{ a.s.}
\end{equation*}
In particular,
\begin{equation}\label{bounded_exp_A}
	\sup_{0\leq t\leq T} 	\mathcal{E}(\beta A)_t \leq \exp\left\{\beta A_t\right\} \leq e^{\beta \mathfrak{C}_A},\quad \text{ a.s.}
\end{equation}
Henceforth, since in addition $\inf_{0\leq t\leq T} 	\mathcal{E}(\beta A)_t \geq 1$, we can deduce that, the norms $
\|\cdot\|_{\mathcal{S}^2_\beta}
$ and $\|\cdot\|_{\mathcal{S}^2}
$ are equivalent.
		\end{enumerate}
\end{remark}

\paragraph*{Terminology.}
The conditions \textbf{(H1)}, \textbf{(H2)} and \textbf{(H3)} will be collectively denoted by \textbf{(H)}. Thus, when we say that a triplet of data $(\xi,f,A)$ satisfies assumption \textbf{(H)}, this means that $\xi$ satisfies \textbf{(H1)}, the generator $f$ satisfies \textbf{(H2)}, and the increasing process $A$ (driver) satisfies \textbf{(H3)}.

In the case of a classical generalized BSDE, see \eqref{GBSDE} below, when we say that a triplet $(\xi_T,f,A)$ satisfies assumption \textbf{(H)}, the same terminology is used with $\xi_T$ as the terminal condition. More precisely, this means in particular that
\begin{description}
\item[(H1)$_T$] A terminal condition $\xi_T$ corresponds to
\begin{itemize}
	\item[(i)] $\xi_T$ is an $\mathcal{F}_T$-measurable random variable;
	\item[(ii)] $\mathbb{E}\left[|\mathcal{E}(\beta A)_T \xi_T|^2\right]<+\infty$.
\end{itemize}
\end{description}

\begin{remark}\label{rmq1}
	Let us first recall that, for any RCLL semimartingale $X$ and any right-continuous process $V$ of finite variation, the quadratic covariation between $X$ and $V$ is giving by $d[X,V]_s=\Delta_- V_s dX_s$ on $[0,T]$; see, for instance, Proposition 4.49 in \cite[p.~52]{JacodShiryaev2013}. 
	
	Moreover, if $Y$ is RCLL and adapted, then its left-limit process $Y_{-}$ is predictable. Since, by assumption \textbf{(H2)}-(i), the random field
	$$
	(\omega,t,y)\longmapsto f(\omega,t,y)
	$$
	is $\mathcal{P}\otimes\mathcal{B}(\mathbb{R})/\mathcal{B}(\mathbb{R})$-measurable, it follows that the process
	$$
	(\omega,t)\longmapsto f(\omega,t,Y_{t-}(\omega))
	$$
	is predictable. Indeed, define
	$$
	\Phi:\Omega\times[0,T]\longrightarrow \Omega\times[0,T]\times\mathbb{R}
	$$
	by
	$$
	\Phi(\omega,t)=\bigl(\omega,t,Y_{t-}(\omega)\bigr).
	$$
	For every $A\in\mathcal{P}$ and every $B\in\mathcal{B}(\mathbb{R})$, we have
	$$
	\Phi^{-1}(A\times B)
	=
	A\cap\{(\omega,t):Y_{t-}(\omega)\in B\}.
	$$
	Since $Y_{-}$ is predictable, the set
	$$
	\{(\omega,t):Y_{t-}(\omega)\in B\}
	$$
	belongs to $\mathcal{P}$. Hence
	$$
	\Phi^{-1}(A\times B)\in\mathcal{P}.
	$$
	Therefore, $\Phi$ is $\mathcal{P}/\bigl(\mathcal{P}\otimes\mathcal{B}(\mathbb{R})\bigr)$-measurable. Consequently,
	$$
	(\omega,t)\longmapsto f\bigl(\omega,t,Y_{t-}(\omega)\bigr)
	=
	f\circ\Phi(\omega,t)
	$$
	is $\mathcal{P}$-measurable. This proves the predictability of the process
	$f(\cdot,Y_{-})$.
\end{remark}

For the remainder of the paper, we introduce the following functional spaces. Let $\beta>0$.
\begin{itemize}
	\item $\mathcal{S}^{2}_{\beta}$: The space of one-dimensional $\mathbb{F}$-optional processes $(Y_t)_{t \leq T}$ with regulated trajectories such that
	$$
	\left\| Y\right\|^2_{\mathcal{S}^{2}_{\beta}}:=
	\mathbb{E}\left[ \esssup_{\eta \in \mathcal{T}_{[0,T]}} \mathcal{E}(\beta A)_\eta \left| Y_\eta \right|^2\right]  <+\infty,
	$$
	with the convention $\mathcal{S}^2:=\mathcal{S}^{2}_{0}$.

	\item Let $\mu_\beta$ be the measure defined on the predictable sigma-field $\mathcal{P}$ by
	$$
	\mu_\beta(B)
	:=
	\mathbb{E}\left[
	\int_{0}^T
	\mathcal{E}(\beta A)_s(\omega)\mathds{1}_{B}(\omega,s)\,dA_s(\omega)
	\right],
	\quad B\in\mathcal{P}.
	$$
	We then define
	$$
	\mathcal{S}^{2,A}_{\beta}
	:=
	\mathbb{L}^2\left(\Omega\times(0,T],\mathcal{P},\mu_\beta\right).
	$$
	Equivalently, $\mathcal{S}^{2,A}_{\beta}$ is the space of one-dimensional predictable processes $U=(U_t)_{0<t\leq T}$ such that
	$$
	\left\| U\right\|^2_{\mathcal{S}^{2,A}_{\beta}}
	:=
	\mathbb{E}\left[
	\int_{0}^T
	\mathcal{E}(\beta A)_s |U_s|^2\,dA_s
	\right]
	<+\infty.
	$$
	As usual in $\mathbb{L}^2$, two predictable processes $U$ and $\widetilde U$ are identified whenever
	$$
	\mathbb{E}\left[
	\int_{0}^T
	\mathcal{E}(\beta A)_s |U_s-\widetilde U_s|^2\,dA_s
	\right]
	=0.
	$$

	\item $\mathcal{M}^2_{\beta}$: the space of one-dimensional $\mathbb{F}$-martingales $(M_t)_{t \leq T}$ such that $M_0=0$ a.s. and
	$$
	\left\| M\right\|^2_{\mathcal{M}^2_{\beta}}:=
	\mathbb{E}\left[ \int_{0}^{T}\mathcal{E}(\beta A)_s d\left[M\right]_s \right] <+\infty,
	$$
	with the convention $\mathcal{M}^2_{}:=\mathcal{M}^2_{0}$.
	
	\item $\mathcal{K}^2$: The space of one-dimensional $\mathbb{F}$-optional non-decreasing processes $(K_t)_{t \leq T}$ with regulated trajectories such that
	$$
	\left\| K\right\|^2_{\mathcal{K}^2}:=\mathbb{E}\left| K_T \right|^2  <+\infty.
	$$
\end{itemize}

\begin{lemma}\label{ouma}
	Let $(Y,M)\in \mathcal{S}^2_\beta \times \mathcal{M}^2_\beta$. Then the stochastic integral
	$$
	\int_{0}^{\cdot}\mathcal{E}(\beta A)_sY_{s-}\,dM_s
	$$
	is a uniformly integrable martingale. Moreover, for every $\varepsilon>0$, there exists a universal constant $c>0$ such that
	\begin{equation}\label{BDG-res}
	\mathbb{E}\left[\sup_{0\leq t\leq T}\left|\int_{t}^{T}\mathcal{E}(\beta A)_sY_{s-}\,dM_s\right|\right]
	\leq
	\frac{\varepsilon}{2}\mathbb{E}\left[\sup_{0\leq t\leq T}\mathcal{E}(\beta A)_t|Y_t|^2\right]
	+
	\frac{c^2}{2\varepsilon}\mathbb{E}\left[\int_{0}^{T}\mathcal{E}(\beta A)_s\,d[M]_s\right].
   \end{equation}
\end{lemma}

\begin{proof}
	By Remark \ref{rmq4}-(i), the process $\mathcal{E}(\beta A)$ is predictable and locally bounded. Since $Y$ is RCLL and adapted, the left-limit process $Y_-$ is predictable. Therefore,
	$$
	\int_{0}^{\cdot}\mathcal{E}(\beta A)_sY_{s-}\,dM_s
	$$
	is a well-defined local martingale. Applying the Burkholder--Davis--Gundy (B-D-G) inequality, see, for instance, Theorem 48 in \cite[p.~193]{Protter2004}, and using the inequality
	$$
	ab\leq \frac{\varepsilon}{2}a^2+\frac{1}{2\varepsilon}b^2,
	\quad \varepsilon>0,
	$$
	we obtain
	\begin{equation*}
	\begin{split}
		\mathbb{E}\left[\sup_{0\leq t\leq T}\left|\int_{t}^{T}\mathcal{E}(\beta A)_sY_{s-}\,dM_s\right|\right]
		&\leq
		c\,\mathbb{E}\left[
		\left(\int_{0}^{T}\mathcal{E}(\beta A)_s^2 |Y_{s-}|^2\,d[M]_s\right)^{\frac12}
		\right]  \\
		&\leq
		c\,\mathbb{E}\left[
		\sup_{0\leq t\leq T}\sqrt{\mathcal{E}(\beta A)_t}|Y_t|
		\left(\int_{0}^{T}\mathcal{E}(\beta A)_s\,d[M]_s\right)^{\frac12}
		\right]  \\
		&\leq
		\frac{\varepsilon}{2}\mathbb{E}\left[\sup_{0\leq t\leq T}\mathcal{E}(\beta A)_t|Y_t|^2\right]
		+
		\frac{c^2}{2\varepsilon}\mathbb{E}\left[\int_{0}^{T}\mathcal{E}(\beta A)_s\,d[M]_s\right]<+\infty.
	\end{split}
	\end{equation*}
	The right-hand side is finite because $(Y,M)\in \mathcal{S}^2_\beta \times \mathcal{M}^2_\beta$. Hence, by Theorem 51 in \cite[p.~38]{Protter2004}, the local martingale
	$
	\int_{0}^{\cdot}\mathcal{E}(\beta A)_sY_{s-}\,dM_s
	$
	 is therefore a uniformly integrable martingale. The announced estimate follows from the preceding inequalities.
\end{proof}

\begin{lemma}\label{ouma1}
	Let $(Y,M)\in \mathcal{S}^2_\beta \times \mathcal{M}^2_\beta$, and let $f$ be a generator satisfying condition \textbf{(H2)}. Then, for every $t\in[0,T]$, we have
	$$
	\mathbb{E}\left[
	\sum_{t<s\leq T}
	\mathcal{E}(\beta A)_s f(s,Y_{s-}) \Delta_- M_s \Delta_- A_s
	\right]
	=0.
	$$
\end{lemma}

\begin{proof}
	We adopt the notation $d\|\alpha\|$ to denotes the total variation measure corresponding to the measure $d\alpha$.
	
	First, we prove that the jump sum is absolutely integrable. By Remark \ref{rmq1}-(ii), the Kunita--Watanabe inequality, see, e.g., Theorem 25 in \cite[p.~69]{Protter2004}, the fact that $(Y,M)\in \mathcal{S}^2_\beta \times \mathcal{M}^2_\beta$, and assumption \textbf{(H2)}-(iii), we obtain
	\begin{equation}\label{driver_part_Tonelli}
		\begin{split}
		&\mathbb{E}\left[
		\sum_{t<s\leq T}
		\left|
		\mathcal{E}(\beta A)_s f(s,Y_{s-}) \Delta_- M_s \Delta_- A_s
		\right|
		\right] \\
		&\leq
		\mathbb{E}\left[
		\int_0^T
		\mathcal{E}(\beta A)_s
		\left|f(s,Y_{s-})\right|
		\,d\|[M,A]\|_s
		\right] \\
		&\leq
		\mathbb{E}\left[
		\left(
		\int_0^T
		\mathcal{E}(\beta A)_s
		\left|f(s,Y_{s-})\right|^2\,dA_s
		\right)^{\frac12}
		\left(
		\int_0^T
		\mathcal{E}(\beta A)_s\,d[M]_s
		\right)^{\frac12}
		\right] \\
		&\leq
		\left(
		\mathbb{E}\left[
		\int_0^T
		\mathcal{E}(\beta A)_s
		\left|f(s,Y_{s-})\right|^2\,dA_s
		\right]
		\right)^{\frac12}
		\left(
		\mathbb{E}\left[
		\int_0^T
		\mathcal{E}(\beta A)_s\,d[M]_s
		\right]
		\right)^{\frac12} \\
		&\leq
		\frac12
		\mathbb{E}\left[
		\int_0^T
		\mathcal{E}(\beta A)_s
		\left|f(s,Y_{s-})\right|^2\,dA_s
		\right]
		+
		\frac12
		\mathbb{E}\left[
		\int_0^T
		\mathcal{E}(\beta A)_s\,d[M]_s
		\right] \\
		&\leq
		K^2
		\mathbb{E}\left[
		\int_0^T
		\mathcal{E}(\beta A)_s
		|Y_{s-}|^2\,dA_s
		\right]
		+
		\mathbb{E}\left[
		\int_0^T
		\mathcal{E}(\beta A)_s
		|f(s,0)|^2\,dA_s
		\right]
		+
		\frac12
		\mathbb{E}\left[
		\int_0^T
		\mathcal{E}(\beta A)_s\,d[M]_s
		\right]
		<+\infty.
		\end{split}
	\end{equation}
%
%
	Thus, the jump sum is integrable, and we may interchange the expectation and the summation.
	
	Next, we may choose a sequence of predictable stopping times $\{\tau_n\}_{n\geq 1}\subset \mathcal{T}^p_{0,T}$ whose graphs exhaust the left jumps of the process $A$. Therefore,
	$$
	\begin{aligned}
		&\mathbb{E}\left[
		\sum_{t<s\leq T}
		\mathcal{E}(\beta A)_s f(s,Y_{s-}) \Delta_- M_s \Delta_- A_s
		\right] \\
		&=
		\sum_{n\geq 1}
		\mathbb{E}\left[
		\mathds{1}_{\{t<\tau_n\leq T\}}
		\mathcal{E}(\beta A)_{\tau_n}
		f(\tau_n,Y_{\tau_n-})
		\Delta_- M_{\tau_n} \Delta_- A_{\tau_n}
		\right].
	\end{aligned}
	$$
	For each $n\geq 1$, the random variable
	$$
	\mathds{1}_{\{t<\tau_n\leq T\}}
	\mathcal{E}(\beta A)_{\tau_n}
	f(\tau_n,Y_{\tau_n-})
	\Delta_- A_{\tau_n}
	$$
	is $\mathcal{F}_{\tau_n-}$-measurable. Since $\tau_n$ is predictable and $M$ is a martingale\footnote{This points are justified in Remark \ref{rmq4}-(ii).}, we have
	$$
	\mathbb{E}\left[\Delta_- M_{\tau_n}\mid \mathcal{F}_{\tau_n-}\right]=0.
	$$
	Consequently,
	$$
	\mathbb{E}\left[
	\mathds{1}_{\{t<\tau_n\leq T\}}
	\mathcal{E}(\beta A)_{\tau_n}
	f(\tau_n,Y_{\tau_n-})
	\Delta_- M_{\tau_n} \Delta_- A_{\tau_n}
	\right]
	=0.
	$$
	Summing over $n\geq 1$ gives
	$$
	\mathbb{E}\left[
	\sum_{t<s\leq T}
	\mathcal{E}(\beta A)_s f(s,Y_{s-}) \Delta_- M_s \Delta_- A_s
	\right]
	=0,
	\quad \forall t\in[0,T].
	$$
	This completes the proof.
\end{proof}

\paragraph*{Notation.}
Throughout this paper, $\mathfrak{C}$ denotes a positive constant whose value may change from one line to another. When needed, we write $\mathfrak{C}_\gamma$ to indicate that the constant depends on the parameter or collection of parameters $\gamma$.

	\section{Generalized BSDE}\label{sec3}
	Let $\xi_T$ be a terminal value of the lower obstacle $\xi$. 
	In this first part, our aim is to study the solvability of the following GBSDE associated with data $(\xi_T,f,A)$:
	\begin{equation}\label{GBSDE}
		Y_t=\xi_T+\int_{t}^{T}f(s,Y_{s-})dA_s-\left(M_T-M_t\right),\quad t \in [0,T].
	\end{equation}

\begin{remark}\label{Otage}
	By taking the conditional expectation with respect to $\mathcal F_t$ on both sides of \eqref{GBSDE}, and using the martingale property of $M$, we obtain
	\begin{equation}\label{GBSDE.write.Y}
		Y_t=
		\mathbb{E}\left[
		\xi_T+\int_{t}^Tf(s,Y_{s-})\,dA_s
		\mid \mathcal{F}_t
		\right],
		\quad t \in [0,T].
	\end{equation}
	Conversely, if a process $Y$ satisfies \eqref{GBSDE.write.Y}, then we may define
	\begin{equation*}
		M_t:=
		\mathbb{E}\left[
		\xi_T+\int_{0}^T f(s,Y_{s-})\,dA_s
		\mid \mathcal F_t
		\right],
		\quad t\in[0,T].
	\end{equation*}
	Then $M$ is a martingale, $Y_T=\xi_T$ and
	\begin{equation*}
		Y_t=-\int_{0}^T f(s,Y_{s-})\,dA_s+M_t.
	\end{equation*}
	Consequently,
	\begin{equation*}
		Y_t
		=
		\xi_T+\int_{t}^Tf(s,Y_{s-})\,dA_s
		-\left(M_T-M_t\right),
	\end{equation*}
	which shows that the pair $(Y,M)$ satisfies the GBSDE \eqref{GBSDE}.
	Thus, solving \eqref{GBSDE} is equivalent to finding a process $Y$ satisfying \eqref{GBSDE.write.Y}, with the martingale $M$ defined as above.
\end{remark}

	We now introduce the notion of a solution to the GBSDE
	\eqref{GBSDE} associated with the data \((\xi_T,f,A)\).
	\begin{definition}\label{def-GBSDE}
		A pair of processes \((Y,M)\) is said to be a solution to the
		GBSDE \eqref{GBSDE} associated with the data
		\((\xi_T,f,A)\) if
		\begin{itemize}
			\item $
			(Y,M)\in
			\mathcal{S}_{\beta}^{2}\times\mathcal{M}_{\beta}^{2}
			$ for some $\beta>0$;
			
			\item $\mathbb{P}\left(
			Y_t
			=
			\xi_T
			+\int_t^T f(s,Y_{s-})\,dA_s
			-\left(M_T-M_t\right),
			~ \textnormal{ for all } t\in[0,T]
			\right)=1.$
		\end{itemize}
		
	\end{definition}

\subsection{Special case: a generator independent of the \(y\)-variable}
We first consider the particular case in which the generator \(f\) does not
depend on the \(y\)-variable. More precisely, we assume that there exists a
process
\[
\mathfrak{f}:\Omega\times[0,T]\longrightarrow\mathbb{R}
\]
such that
\[
f(\omega,t,y)=\mathfrak{f}(\omega,t),
\qquad
(\omega,t,y)\in\Omega\times[0,T]\times\mathbb{R}.
\]
Equivalently, one may set
\[
\mathfrak{f}(\omega,t):=f(\omega,t,0),
\qquad
(\omega,t)\in\Omega\times[0,T].
\]

We further assume that
\[
\mathbb{E}\left[
\int_0^T
\mathcal{E}(\beta A)_s
|\mathfrak{f}(s)|^2\,dA_s
\right]
<+\infty.
\]

In this setting, we consider the following GBSDE:
\begin{equation}\label{GBSDE.Indp.y}
	Y_t
	=
	\xi_T
	+\int_t^T\mathfrak{f}(s)\,dA_s
	-\left(M_T-M_t\right),
	\qquad t\in[0,T].
\end{equation}

\subsubsection{A priori estimates and uniqueness}
%

\begin{proposition} \label{GBSDE-propo1}
	Let $(\xi_T,f,A)$ be a set of data satisfying assumption \textbf{(H)}. Let $(Y,M)$ be a solution of the GBSDE \eqref{GBSDE.Indp.y} associated with $(\xi_T,f,A)$. Then there exists a constant $\mathfrak{C}_{\beta}>0$ such that
	\begin{equation*}
		\begin{split}
			&\mathbb{E}\left[\sup_{0\leq t\leq T}\mathcal{E}(\beta A)_t |Y_t|^2\right]+\mathbb{E}\left[ \int_{0}^{T}\mathcal{E}(\beta A)_s |Y_{s-}|^2 dA_s\right] +\mathbb{E}\left[\int_{0}^{T}{\mathcal{E}(\beta A)_s}d[M]_s\right] \\ 
			& \leq \mathfrak{C}_{\beta} \left(  \mathbb{E}\left[\mathcal{E}(\beta A)_T |\xi_T|^2\right] +\mathbb{E}\left[\int_{0}^{T}\mathcal{E}(\beta A)_{s} |\mathfrak{f}(s)|^2 dA_s\right]\right).
		\end{split}
	\end{equation*}
\end{proposition}

	\begin{proof}
		 From It\^o's formula applied to the process $\mathcal{E}(\beta A)_t |Y_t|^2$ along with the use of claim from Remark \ref{rmq1} and \eqref{matla} it follows that
		 \begin{equation}\label{Ito1}
		 	\begin{split}
		 		d\left(\mathcal{E}(\beta A)_s |Y_s|^2\right)&=\mathcal{E}(\beta A)_{s-}d|Y_{s}|^2+|Y_{s-}|^2 d\mathcal{E}(\beta A)_s+\Delta \mathcal{E}(\beta A)_{s-} d|Y_{s}|^2\\
		 		&=\mathcal{E}(\beta A)_{s}d|Y_{s}|^2+|Y_{s-}|^2 d\mathcal{E}(\beta A)_s\\
		 		&=2\mathcal{E}(\beta A)_{s}Y_{s-}dY_s+\mathcal{E}(\beta A)_{s}d[Y]^c_s+\mathcal{E}(\beta A)_{s}\left(\Delta_- Y_s\right)^2+\beta \frac{\mathcal{E}(\beta A)_s}{1+\beta \Delta A_s} |Y_{s-}|^2 dA_s
		 	\end{split}
		 \end{equation}
		 From the GBSDE \eqref{GBSDE.Indp.y} and using Theorem 13 in \cite[p. 60]{Protter2004}, we have
		 $$
		 [Y]^c_s=[M]^c_s
		 $$
		 and
		 \begin{equation}\label{jump of Y}
		 	\Delta_- Y_s=-\mathfrak{f}(s)\Delta_- A_s+\Delta_- M_s
		 \end{equation}
		 Thus
		 \begin{equation}\label{jump of Y square}
		  \left(\Delta_- Y_s\right)^2=|\mathfrak{f}(s)|^2 |\Delta_- A_s|^2-2\mathfrak{f}(s)\Delta_- A_s \Delta_-M_s +|\Delta_- M_s|^2
		 \end{equation}
		 Therefore, we get
		 \begin{equation}\label{Plug1}
		 	\begin{split}
		 		\mathcal{E}(\beta A)_{s}d[Y]_s&=\mathcal{E}(\beta A)_{s}d[Y]^c_s+\mathcal{E}(\beta A)_{s}\left(\Delta_- Y_s\right)^2\\
		 		&=\mathcal{E}(\beta A)_{s}|\mathfrak{f}(s)|^2 |\Delta_- A_s|^2-2\mathcal{E}(\beta A)_{s}\mathfrak{f}(s)\Delta_- A_s \Delta_-M_s +\mathcal{E}(\beta A)_{s}d [M]_s
		 	\end{split}
		 \end{equation}
		 Plugging \eqref{Plug1} into \eqref{Ito1} and using \eqref{essential condition-r1}, we find
		 \begin{equation}\label{Ito-r1}
		 	\begin{split}
		 	&\mathcal{E}(\beta A)_t |Y_t|^2+\frac{\beta}{1+\beta \mathfrak{C}_A }\int_{t}^{T}\mathcal{E}(\beta A)_s |Y_{s-}|^2 dA_s+\int_{t}^{T}{\mathcal{E}(\beta A)_s}d[M]_s\\
		 	&\leq	\mathcal{E}(\beta A)_T |\xi_T|^2+2\int_{t}^{T}\mathcal{E}(\beta A)_{s}Y_{s-} \mathfrak{f}(s)dA_s-2\int_{t}^{T}\mathcal{E}(\beta A)_{s}Y_{s-}dM_s\\
		 	&\qquad-\sum_{t < s \leq T} \mathcal{E}(\beta A)_s|\mathfrak{f}(s)|^2 |\Delta_- A_s|^2+2\sum_{t < s \leq T} \mathcal{E}(\beta A)_s \mathfrak{f}(s)\Delta_- A_s \Delta_-M_s\\
		 	& \leq \mathcal{E}(\beta A)_T |\xi_T|^2+2\int_{t}^{T}\mathcal{E}(\beta A)_{s}Y_{s-} \mathfrak{f}(s)dA_s-2\int_{t}^{T}\mathcal{E}(\beta A)_{s}Y_{s-}dM_s
		 	+\sum_{t < s \leq T} \mathcal{E}(\beta A)_s \mathfrak{f}(s)\Delta_- A_s \Delta_-M_s.
		 	\end{split}
		 \end{equation}
		 From the basic inequality $2 ab \leq \epsilon a^2 +\frac{1}{\epsilon} b^2$, $\forall \epsilon >0$, we get the following inequality 
		 \begin{equation}\label{Ito-r2}
		 	2Y_{s-} \mathfrak{f}(s) \leq  \dfrac{\beta}{2(1+\beta \mathfrak{C}_A)}|Y_{s-}|^2+\dfrac{2(1+\beta \mathfrak{C}_A)}{\beta} |\mathfrak{f}(s)|^2 .
		 \end{equation}
	 Plugging \eqref{Ito-r2} into \eqref{Ito-r1}, we get
	 \begin{equation}\label{Ito-r3}
	 	\begin{split}
	 		&\mathcal{E}(\beta A)_t |Y_t|^2+\frac{\beta}{2 \left(1+\beta \mathfrak{C}_A\right) }\int_{t}^{T}\mathcal{E}(\beta A)_s |Y_{s-}|^2 dA_s+\int_{t}^{T}{\mathcal{E}(\beta A)_s}d[M]_s\\
	 		& \leq \mathcal{E}(\beta A)_T |\xi_T|^2+\int_{t}^{T}\mathcal{E}(\beta A)_{s} |\mathfrak{f}(s)|^2 dA_s-2\int_{t}^{T}\mathcal{E}(\beta A)_{s}Y_{s-}dM_s\\
	 		&\qquad+\sum_{t < s \leq T} \mathcal{E}(\beta A)_s \mathfrak{f}(s)\Delta_- A_s \Delta_-M_s.\\
	 	\end{split}
	 \end{equation}
	 Taking expectation on both sides of \eqref{Ito-r3} along with the application of Lemmas \ref{ouma} and \ref{ouma1}, we derive
	 	 \begin{equation}\label{Ito-r4}
	 	\begin{split}
	 		&\sup_{0\leq t\leq T} \mathbb{E}\left[  \mathcal{E}(\beta A)_t |Y_t|^2\right] +\frac{\beta}{2 \left(1+\beta \mathfrak{C}_A\right) }\mathbb{E}\left[ \int_{0}^{T}\mathcal{E}(\beta A)_s |Y_{s-}|^2 dA_s\right] +\mathbb{E}\left[\int_{0}^{T}{\mathcal{E}(\beta A)_s}d[M]_s\right] \\
	 		& \leq \mathbb{E}\left[\mathcal{E}(\beta A)_T |\xi_T|^2\right] +\mathbb{E}\left[\int_{0}^{T}\mathcal{E}(\beta A)_{s} |\mathfrak{f}(s)|^2 dA_s\right].
	 	\end{split}
	 \end{equation}
	
	Consequently, we derive the existence of a constant $\mathfrak{C}_{\beta}>0$ such that 
	 \begin{equation}\label{Ito-r5}
		\begin{split}
			&\sup_{0\leq t\leq T} \mathbb{E}\left[  \mathcal{E}(\beta A)_t |Y_t|^2\right] +\mathbb{E}\left[ \int_{0}^{T}\mathcal{E}(\beta A)_s |Y_{s-}|^2 dA_s\right] +\mathbb{E}\left[\int_{0}^{T}{\mathcal{E}(\beta A)_s}d[M]_s\right] \\
			& \leq \mathfrak{C}_{\beta,\mathfrak{C}_{A}} \left(  \mathbb{E}\left[\mathcal{E}(\beta A)_T |\xi_T|^2\right] +\mathbb{E}\left[\int_{0}^{T}\mathcal{E}(\beta A)_{s} |\mathfrak{f}(s)|^2 dA_s\right]\right).
		\end{split}
	\end{equation}
		 Returning now to \eqref{Ito-r3} and applying \eqref{BDG-res} with $\varepsilon=\frac{1}{2}$, see Lemma \ref{ouma}, we obtain
		  \begin{equation}\label{Ito-r6}
		 	\begin{split}
		 		\frac{1}{2} \mathbb{E}\left[\sup_{0\leq t\leq T}\mathcal{E}(\beta A)_t |Y_t|^2\right] 
		 		& \leq \mathbb{E}\left[\mathcal{E}(\beta A)_T |\xi_T|^2\right] +2\mathbb{E}\left[\int_{0}^{T}\mathcal{E}(\beta A)_{s} |\mathfrak{f}(s)|^2 dA_s\right] +2c^2\mathbb{E}\left[\int_{0}^{T}\mathcal{E}(\beta A)_s\,d[M]_s\right] \\
		 		&\qquad+\mathbb{E}\left[\sum_{t < s \leq T} \mathcal{E}(\beta A)_s |\mathfrak{f}(s)| \Delta_- A_s |\Delta_-M_s|\right].\\
		 	\end{split}
		 \end{equation}
		For the last term on the right-hand side of \eqref{Ito-r6}, we apply the same argument as in Lemma \ref{ouma1}, based on the Kunita--Watanabe inequality. Combining this with the estimate obtained in \eqref{Ito-r5}, we deduce that there exists a constant $\mathfrak{C}_{\beta}>0$ such that
		 \begin{equation*}
		 	\begin{split}
		 		\mathbb{E}\left[\sup_{0\leq t\leq T}\mathcal{E}(\beta A)_t |Y_t|^2\right] 
		 		& \leq \mathfrak{C}_{\beta,\mathfrak{C}_{A}} \left(  \mathbb{E}\left[\mathcal{E}(\beta A)_T |\xi_T|^2\right] +\mathbb{E}\left[\int_{0}^{T}\mathcal{E}(\beta A)_{s} |\mathfrak{f}(s)|^2 dA_s\right]\right).
		 	\end{split}
		 \end{equation*}
	 Completing the proof.
	 
	\end{proof}

Following the argument used in Proposition \ref{GBSDE-propo1}, we obtain the following more general a priori estimate.

\begin{proposition}\label{GBSDE-propo2}
	Let $(\xi^1_T,\mathfrak{f}^1(\cdot),A)$ and $(\xi^2_T,\mathfrak{f}^2(\cdot),A)$ be two sets of data satisfying assumption \textbf{(H)}. Let $(Y^1,M^1)$ and $(Y^2,M^2)$ be two solutions of the GBSDE \eqref{GBSDE} associated with $(\xi^1_T,\mathfrak{f}^1(\cdot),A)$ and $(\xi^2_T,\mathfrak{f}^2(\cdot),A)$, respectively. Then there exists a constant $\mathfrak{C}_{\beta}>0$ such that
	\begin{equation*}
		\begin{split}
			&\mathbb{E}\left[\sup_{0\leq t\leq T}\mathcal{E}(\beta A)_t |Y^1_t-Y^2_t|^2\right]
			+\mathbb{E}\left[\int_{0}^{T}\mathcal{E}(\beta A)_s |Y^1_{s-}-Y^2_{s-}|^2\,dA_s\right] 			
			+\mathbb{E}\left[\int_{0}^{T}\mathcal{E}(\beta A)_s\,d[M^1-M^2]_s\right] \\
			&\leq
			\mathfrak{C}_{\beta,\mathfrak{C}_{A}}
			\left(
			\mathbb{E}\left[\mathcal{E}(\beta A)_T |\xi^1_T-\xi^2_T|^2\right]
			+
			\mathbb{E}\left[\int_{0}^{T}\mathcal{E}(\beta A)_s
			|\mathfrak{f}^1(s)-\mathfrak{f}^2(s)|^2\,dA_s\right]
			\right).
		\end{split}
	\end{equation*}
\end{proposition}

\begin{proof}
	For any quantity $\mathfrak{R}\in\{\xi_T,\mathfrak{f},Y,M\}$, we set
	$$
	\bar{\mathfrak{R}}:=\mathfrak{R}^1-\mathfrak{R}^2.
	$$
	Then, by subtracting the two GBSDEs, the process $\bar{Y}$ satisfies
	\begin{equation*}
		\bar{Y}_t
		=
		\bar{\xi}_T
		+
		\int_t^T
		\bar{\mathfrak{f}}(s)
		\,dA_s
		-
		\left(\bar{M}_T-\bar{M}_t\right),
		\quad t\in[0,T].
	\end{equation*}
	We then follow the proof of Proposition \ref{GBSDE-propo1}. 
	
	Replacing the corresponding estimate in the proof of Proposition \ref{GBSDE-propo1} by the above inequality and repeating the same arguments yield the desired estimate.
\end{proof}

As a consequence of Proposition \ref{GBSDE-propo2}, we obtain the following uniqueness result.
\begin{corollary}\label{GBSDE-uniq}
	Assume that the data $(\xi_T,\mathfrak{f}(\cdot),A)$ satisfy \textbf{(H)}. Then there exists at most one pair of processes $(Y,M)\in \mathcal{S}^2_\beta \times \mathcal{M}^2_\beta$, solving the GBSDE \eqref{GBSDE.Indp.y} associated with $(\xi_T,\mathfrak{f}(\cdot),A)$.
\end{corollary}

It is worth mentioning that the uniqueness statement in Corollary \ref{GBSDE-uniq} can also be recovered through a different argument, namely by means of a comparison principle. This method relies on a direct examination of the trajectories of the process $Y$ and requires an additional monotonicity assumption on the generator $f$. The proof is inspired by the strategy developed in Proposition 3.6 of \cite{li2024penalization} in the particular case where the driver $A$ is continuous.

\begin{proposition}\label{GBSDE-propo3}
	Let $\xi^1_T$ and $\xi^2_T$ be two terminal conditions\footnote{In the sense of condition \textbf{(H1)$_T$} above. In the present comparison result, only the measurability condition \textbf{(H1)$_T$}-(i) is needed; no integrability assumption is used in the proof.}, and let
	$\mathfrak{f}^1,\mathfrak{f}^2:\Omega\times[0,T]\rightarrow\mathbb{R}$
	be two generators satisfying condition \textbf{(H2)}-(i). Assume that the GBSDE \eqref{GBSDE} admits two solutions $(Y^1,M^1)$ and $(Y^2,M^2)$ associated with the data $(\xi^1_T,\mathfrak{f}^1(\cdot),A)$ and $(\xi^2_T,\mathfrak{f}^2(\cdot),A)$, respectively. Suppose that
	\begin{itemize}
		\item[(i)] $
		\xi^1_T\geq \xi^2_T$, $\mathbb{P}\text{-a.s.}$
		
		\item[(ii)] $\mathfrak{f}^1(\omega,t)\geq \mathfrak{f}^2(\omega,t)$ for every $(\omega,t,y)\in\Omega\times[0,T]\times\mathbb{R}$;
	\end{itemize}
	Then
	$$
	Y^1_t\geq Y^2_t,
	\quad \forall t\in[0,T],
	\quad \mathbb{P}\text{-a.s.}
	$$
\end{proposition}

\begin{proof}
	We follow the proof of Proposition 3.6 in \cite{li2024penalization}. We also refer to \cite[Lemma 8.3]{Peng2004NonlinearEvaluations}.
	
	For a fixed $\varepsilon>0$, we consider the $\mathbb{F}$-stopping time
	$$
	\tau^{\varepsilon}:=\inf \left\{t \geqslant 0: Y^1_t \leqslant {Y}^2_t-\varepsilon\right\},
	$$
	with the convention that $\inf \emptyset=T$. By the argument used in the proof of Proposition 3.6 in \cite{li2024penalization}, if for every $\varepsilon>0$ we have $\mathbb{P}\left(\tau^{\varepsilon}=T\right)=1$, then
	$$
	Y^1_t\geq Y^2_t,\quad \forall t\in[0,T],\quad \mathbb{P}\text{-a.s.}
	$$
	
	We now argue by contradiction. Assume that the inequality $Y^1\geqslant Y^2$ does not hold. Then, by the previous step, there exists $\varepsilon>0$ such that $\mathbb{P}(E)>0$, where
	$
	E:=\left\{\tau^{\varepsilon}<T\right\}\in \mathcal{F}_{\tau^{\varepsilon}}.
	$
	We fix such an $\varepsilon$ and define
	$
	\tau:=\tau^{\varepsilon}\mathds{1}_E+T\mathds{1}_{E^c}
	$
	and
	$
	\nu:=\inf \left\{t \geqslant \tau: Y^1_t \geqslant Y^2_t\right\}.
	$
	Then $\nu\leqslant T$, since $\xi^1_T\geqslant \xi^2_T$. Since $Y^1$ and $Y^2$ are RCLL processes, the interval $\llbracket \tau,\nu \llbracket$ is nonempty on $E$, because $Y^1_\tau<Y^2_\tau$ on $E$. Moreover, the inequality $Y^1_\nu\geqslant Y^2_\nu$ holds, and we have
	$
	Y^1_{t-}<Y^2_{t-}$, $t\in]\tau,\nu]$, \text{a.s.}

	Let us set $\bar{f}:=f^1-f^2$ and $\bar{M}:=M^1-M^2$. Using the GBSDE \eqref{GBSDE}, we can write
	\begin{equation}\label{DM1}
		\begin{split}
			(Y^1_\tau-Y^2_\tau)\mathds{1}_E
			&=(Y^1_\nu-Y^2_\nu)\mathds{1}_E
			+\int_{\tau}^{\nu} \bar{\mathfrak{f}}(s)\mathds{1}_E\,dA_s 
			-\int_{\tau}^{\nu}\mathds{1}_E\,d\bar{M}_s .
		\end{split}
	\end{equation}
	From the preceding arguments and assumptions  (ii) on $\mathfrak{f}^1(\cdot)$ and $\mathfrak{f}^2(\cdot)$, we deduce that $\mathds{1}_E\bar{\mathfrak{f}}(\cdot)$ is predictable and nonnegative on $\rrbracket \tau,\nu \rrbracket$. Therefore, after taking the conditional expectation with respect to $\mathcal{F}_\tau$ in \eqref{DM1}, and using the fact that $E\in\mathcal{F}_{\tau^\varepsilon}\subseteq\mathcal{F}_\tau$, we obtain
	$
	Y^1_\tau\geq Y^2_\tau$, \text{a.s. on } $E$.\footnote{Indeed, the process $Y^1-Y^2$ is an RCLL supermartingale on $\llbracket \tau,\nu \rrbracket$. Its Doob--Meyer decomposition (see Theorem 8 in \cite[p. 111]{Protter2004}) is given by \eqref{DM1}, where the predictable finite variation part is nondecreasing on $\rrbracket \tau,\nu \rrbracket$ by assumptions (i) and (ii).}
	This contradicts our assumption. Hence the claim follows.
\end{proof}

\begin{remark}
Proposition \ref{GBSDE-propo3} can also be proved by applying the Meyer--Tanaka formula, see, e.g., Theorem 68 in \cite[p.~213]{Protter2004}, to the process $(Y^2-Y^1)^+$ on $[t,T]$, for any fixed $t\in[0,T]$. 
\end{remark}

Under the assumptions imposed on the data $(\xi_T,\mathfrak{f}(\cdot),A)$ of the GBSDE \eqref{GBSDE}, we obtain the following uniqueness result.

\begin{corollary}
	Let $\xi_T$ be a terminal condition, and let $\mathfrak{f}(\cdot)$ be a generator that is independent of $y$, as in Proposition \ref{GBSDE-propo3}. If a solution $(Y,M)$ to the GBSDE \eqref{GBSDE.Indp.y} associated with $(\xi_T,\mathfrak{f}(\cdot),A)$ exists, then it is unique.
\end{corollary}

\begin{proof}
	Let $(Y^1,M^1)$ and $(Y^2,M^2)$ be two solutions to the GBSDE \eqref{GBSDE} associated with the same data $(\xi_T,f,A)$. By Proposition \ref{GBSDE-propo3}, we obtain
	$Y^1_t=Y^2_t$, $\forall t\in[0,T]$, \text{a.s.}\\
	Returning to the GBSDE \eqref{GBSDE}, it follows that
	$M^1_t=M^2_t$, $\forall t\in[0,T]$, \text{a.s.}
	
	Hence the solution is unique.
	
\end{proof}

\subsubsection{Existence}

\begin{theorem}\label{GBSDE-special_thm}
	Assume that the data \((\xi_T,\mathfrak{f},A)\) satisfy
	assumption \textbf{(H)}. Then, the GBSDE \eqref{GBSDE}
	admits a unique solution $
	(Y,M)\in\mathcal{S}_{\beta}^{2}
	\times\mathcal{M}_{\beta}^{2}.
	$ Moreover, \(Y\in\mathcal{S}_{\beta}^{2,A}\).
\end{theorem}

\begin{proof}
	The uniqueness follows directly from Corollary \ref{GBSDE-uniq}.
	
	We now prove existence. Define the state process \(Y\) by
	\begin{equation}\label{GBSDE.write.Y.E}
		Y_t
		:=
		\mathbb{E}\left[
		\xi_T+\int_t^T\mathfrak{f}(s)\,dA_s
		\Bigm|\mathcal{F}_t
		\right],
		\qquad t\in[0,T].
	\end{equation}
	Equivalently,
	\[
	Y_t+\int_0^t\mathfrak{f}(s)\,dA_s
	=
	\mathbb{E}\left[
	\xi_T+\int_0^T\mathfrak{f}(s)\,dA_s
	\Bigm|\mathcal{F}_t
	\right].
	\]
	
	We therefore define
	\begin{equation}\label{GBSDE.write.M.E}
		\begin{split}
			M_t
			&:=
			Y_t-Y_0+\int_0^t\mathfrak{f}(s)\,dA_s \\
			&=
			\mathbb{E}\left[
			\xi_T+\int_0^T\mathfrak{f}(s)\,dA_s
			\Bigm|\mathcal{F}_t
			\right]-Y_0,
			\qquad t\in[0,T].
		\end{split}
	\end{equation}
	Clearly, \(M_0=0\). Moreover, under the usual conditions on the
	filtration, \(M\) admits an RCLL modification; see, for instance,
	\cite[Theorem 9, p.~8]{Protter2004}.
	
	By assumptions \textbf{(H1)}\(_T\), \textbf{(H2)}-(iii), and
	\textbf{(H3)}, the random variable
	\[
	\xi_T+\int_0^T\mathfrak{f}(s)\,dA_s
	\]
	is square-integrable. Hence, \(M\) is a square-integrable
	martingale.
	
	Furthermore, by Doob's quadratic maximal inequality, Young's inequality,
	and estimates \eqref{code} and \eqref{bounded_exp_A}, there exists a constant
	\(C_{\beta,\mathfrak{C}_A}>0\) such that
	\begin{equation}\label{M.In}
		\begin{split}
			\mathbb{E}\left[
			\sup_{0\leq t\leq T}
			\mathcal{E}(\beta A)_t|M_t|^2
			\right]
			&\leq
			C_{\beta,\mathfrak{C}_A}
			\mathbb{E}\left[
			\mathcal{E}(\beta A)_T|\xi_T|^2
			+
			\int_0^T
			\mathcal{E}(\beta A)_s
			|\mathfrak{f}(s)|^2\,dA_s
			\right] 
			<+\infty.
		\end{split}
	\end{equation}
	Consequently, \(M\in\mathcal{M}_{\beta}^{2}\).
	
	From \eqref{GBSDE.write.M.E}, we have
	\[
	Y_t
	=
	Y_0
	-\int_0^t\mathfrak{f}(s)\,dA_s
	+M_t.
	\]
	Since \(M\) is RCLL and
	\(t\mapsto\int_0^t\mathfrak{f}(s)\,dA_s\) is RCLL, it follows
	that \(Y\) is also RCLL. Moreover, using the preceding identity
	together with \eqref{M.In}, we obtain
	\[
	\mathbb{E}\left[
	\sup_{0\leq t\leq T}
	\mathcal{E}(\beta A)_t|Y_t|^2
	\right]
	<+\infty.
	\]
	Hence,
	\[
	Y\in\mathcal{S}_{\beta}^{2}.
	\]
	
	Setting \(t=T\) in \eqref{GBSDE.write.Y.E} yields
	\[
	Y_T=\xi_T.
	\]
	On the other hand, evaluating \eqref{GBSDE.write.M.E} at \(T\)
	and subtracting its value at \(t\), we obtain
	\[
	M_T-M_t
	=
	\xi_T-Y_t+\int_t^T\mathfrak{f}(s)\,dA_s.
	\]
	Therefore,
	\[
	Y_t
	=
	\xi_T
	+\int_t^T\mathfrak{f}(s)\,dA_s
	-\left(M_T-M_t\right),
	\qquad t\in[0,T].
	\]
	Thus, \((Y,M)\) is a solution of the GBSDE \eqref{GBSDE}.
	
	In addition, assumption \textbf{(H3)} gives
	$
	Y\in\mathcal{S}_{\beta}^{2,A}.
	$
	
	Consequently, the GBSDE \eqref{GBSDE} admits a unique solution
	\[
	(Y,M)\in
	\mathcal{S}_{\beta}^{2}
	\times\mathcal{M}_{\beta}^{2}.
	\]
	Completing the proof.
\end{proof}

The uniqueness of Theorem \ref{GBSDE-special_thm} can be obtained using  Corollary \ref{GBSDE-uniq}, or it also can be derived following the proof of the same Theorem. Indeed, from the arguments adopted, we consider
	\((Y^1,M^1)\) and \((Y^2,M^2)\) to be two solutions of
	\eqref{GBSDE.Indp.y}. Set
	\[
	\overline{Y}:=Y^1-Y^2,
	\qquad
	\overline{M}:=M^1-M^2.
	\]
	Since both equations have the same terminal condition and the same
	driver, we have
	\[
	\overline{Y}_t
	=
	-\left(\overline{M}_T-\overline{M}_t\right),
	\qquad t\in[0,T],
	\]
	with
	\[
	\overline{Y}_T=0.
	\]
	Taking the conditional expectation with respect to
	\(\mathcal{F}_t\), and using the martingale property of
	\(\overline{M}\), yields
	\[
	\bar{Y}_t
	=
	\mathbb{E}\left[
	\overline{Y}_T
	\mid\mathcal{F}_t
	\right]
	=0.
	\]
	Therefore, since $\bar{Y}$ has RCLL path, we get
	\[
	Y^1=Y^2
	\]
	in the indistinguishable sense; see, for instance,
	\cite[Theorem 2, p.~4]{Protter2004}.. Substituting this equality into
	the two GBSDEs gives
	\[
	M^1=M^2.
	\]

\subsection{General case of the generator $f$}


\begin{theorem}\label{GBSDE-general_thm}
	The GBSDE \eqref{GBSDE} associated with $(\xi_T,f,A)$ satisfying assumption \textbf{(H)} admits a unique solution $(Y,M)$ in the space $\mathcal{S}^{2}_{\beta}  \times \mathcal{M}^2_{\beta}$.
\end{theorem}

\begin{proof}
%
%
	Let $\mathcal{S}^{2}_{\beta}$, endowed with its natural norm
	$$
	\|Y\|_{\mathcal{S}^{2}_{\beta}} = \left( \mathbb{E}\left[\sup_{0\leq t\leq T}\mathcal{E}(\beta A)_t |Y_{t}|^2 \right] \right)^{\frac12}.
	$$
	Clearly, the space $\mathcal{S}^{2}_{\beta}$ is complete and thus a Banach space.
	
	Now, define the map $\Psi:\mathcal S^2_\beta\to \mathcal S^2_\beta$ as follows. 
	For a given process $y^1\in\mathcal S^2_\beta$, we set $\Psi(y^1)=Y^1$, where $Y^1$ is given by the formulation \eqref{GBSDE.write.Y}. In other word, we have
	$$
	Y_t^1
	=
	\mathbb E\left[
	\xi_T+\int_{t}^T f(s,y^1_{s-})\,dA_s
	\mid \mathcal F_t
	\right],
	\qquad t\in[0,T].
	$$
	Equivalently, by Remark \ref{Otage}, we may associate with $Y^1$ a martingale $M^1$ such that the pair $(Y^1,M^1)$ solves the GBSDE
	$$
	Y_t^1
	=
	\xi_T+\int_{t}^T f(s,y^1_{s-})\,dA_s
	-\left(M_T^1-M_t^1\right),
	\qquad t\in[0,T].
	$$
	In other words, $\Psi$ maps the input process $y^1$ to the first component $Y^1$ of the solution of the GBSDE \eqref{GBSDE} with driver $f(\cdot,y^1_{\cdot})$.
	
	Note that, from assumptions \textbf{(H2)}-(ii)-(iii) and \textbf{(H3)}, we have
	\begin{equation*}
		\begin{split}
		\mathbb{E}\left[
		\int_{0}^T
		\mathcal{E}(\beta A)_s |f(s,y^1_{s-})|^2\,dA_s
		\right] &\leq
		2K^2\mathbb{E}\left[
		\int_{0}^T
		\mathcal{E}(\beta A)_s |y^1_{s-}|^2\,dA_s
		\right] +2\mathbb{E}\left[
		\int_{0}^T
		\mathcal{E}(\beta A)_s |f(s,0)|^2\,dA_s
		\right]\\
		& \leq2K^2\mathfrak{C}_A \mathbb{E}\left[\sup_{0\leq t\leq T}\mathcal{E}(\beta A)_t |{y}^1_{t}|^2 \right]+2\mathbb{E}\left[
		\int_{0}^T
		\mathcal{E}(\beta A)_s |f(s,0)|^2\,dA_s
		\right]
		<+\infty.
		\end{split}
	\end{equation*}
Then Theorem \ref{GBSDE-special_thm} allows us to claim that the mapping $\Psi$ is well-defined.

	Take another process $y^2$ in $\mathcal{S}^{2}_{\beta}$ and set $Y^2 = \Psi(y^2)$, where $Y^2$ is also given by \eqref{GBSDE.write.Y} associated with $(\xi_T,f(t,y^2_{-}),A)$. For any symbol $\mathfrak{S} \in \{Y,y\}$, define the difference $\bar{\mathfrak{S}} := \mathfrak{S}^1 - \mathfrak{S}^2$.
	
	Using the Lipschitz property of $f$, Jensen's inequality and the formulation \eqref{GBSDE.write.Y} of $Y^1$ and $Y^2$, we obtain
	\begin{equation*}
		\begin{split}
			|\bar{Y}_t|&=\left|\mathbb E\left[
			\int_{t}^T \left(f(s,y^1_{s-})-f(s,y^2_{s-})\right)\,dA_s
			\mid \mathcal F_t
			\right]\right|
			 \leq K \mathbb E\left[
			\int_{t}^T |\bar{y}_{s-}|\,dA_s
			\mid \mathcal F_t
			\right],\quad t \in [0,T].
		\end{split}
	\end{equation*}
	Therefore, for any $t \in [0,T]$, we have
	\begin{equation}\label{Uranus}
		\begin{split}
	\sqrt{\mathcal E(\beta A)_t}|\bar{Y}_t| &\leq  K \mathbb E\left[
	\int_{t}^T \sqrt{\mathcal E(\beta A)_s} |\bar{y}_{s-}| \left(\frac{\sqrt{\mathcal E(\beta A)_t}}{\sqrt{\mathcal E(\beta A)_s}}\right)\,dA_s
	\mid \mathcal F_t
	\right]\\
	& \leq K \mathbb E\left[\sup_{0\leq t\leq T} \sqrt{\mathcal E(\beta A)_t} |\bar{y}_{t}|
	\int_{t}^T  \left(\frac{\sqrt{\mathcal E(\beta A)_t}}{\sqrt{\mathcal E(\beta A)_s}}\right)\,dA_s
	\mid \mathcal F_t
	\right]\quad \text{a.s.}.
		\end{split}
\end{equation}
	Now, we set
	$$
	J_t:=\int_{t}^T\frac{\sqrt{\mathcal E(\beta A)_t}}{\sqrt{\mathcal E(\beta A)_s}}\,dA_s,\quad t \in [0,T].
	$$
	For fixed $t \in [0,T]$, we denote
	$$
	q_s:=\frac{\sqrt{\mathcal E(\beta A)_t}}{\sqrt{\mathcal E(\beta A)_s}}, \qquad s\in[t,T].
	$$
	Then $q_t=1$, $0\leq q_s\leq 1$ and $q$ is an RCLL, predictable, nonincreasing process. Moreover,
	$$
	J_t=\int_{(t,T]} q_s\,dA_s.
	$$
	Classically, we split the last integral into its continuous and purely discontinuous parts (see \cite[p. 99]{he2019semimartingale}):
	$$
	J_t=\int_{t}^T q_s\,dA_s^c+\sum_{t<s\leq T}q_s\Delta_-A_s.
	$$
	
	On the continuous part, we have
	$
	dq_s^c=-\frac{\beta}{2}q_s\,dA_s^c.
	$
	Hence
	$$
	q_s\,dA_s^c=-\frac{2}{\beta}\,dq_s^c.
	$$
	Therefore,
	$$
	\int_{t}^T q_s\,dA_s^c
	=
	\frac{2}{\beta}\int_{t}^T(-dq_s^c).
	$$
	
	For the jump part, at each jump time $s$, since
	$
	\mathcal E(\beta A)_s
	=
	\mathcal E(\beta A)_{s-}\left(1+\beta\Delta_-A_s\right).
	$
	one has
	$$
	q_s=\frac{q_{s-}}{\sqrt{1+\beta\Delta_-A_s}}.
	$$
	It follows that
	$$
	q_{s-}-q_s
	=
	q_s\left(\sqrt{1+\beta\Delta_-A_s}-1\right).
	$$
	Moreover,
	$$
	\Delta_-A_s
	=
	\frac{
		\left(\sqrt{1+\beta\Delta_-A_s}-1\right)
		\left(\sqrt{1+\beta\Delta_-A_s}+1\right)
	}{\beta}.
	$$
	Thus
	$$
	q_s\Delta_-A_s
	=
	\frac{\sqrt{1+\beta\Delta_-A_s}+1}{\beta}
	\left(q_{s-}-q_s\right).
	$$
	Since $A_T\leq \mathfrak{C}_A$ from assumption \textbf{(H3)}, we have $\Delta_-A_s\leq \mathfrak{C}_A$. Hence
	$$
	\sqrt{1+\beta\Delta_-A_s}+1
	\leq
	1+\sqrt{1+\beta \mathfrak{C}_A}.
	$$
	Therefore,
	$$
	q_s\Delta_-A_s
	\leq
	\frac{1+\sqrt{1+\beta \mathfrak{C}_A}}{\beta}
	\left(q_{s-}-q_s\right).
	$$
	Summing over \(t<s\leq T\), we obtain
	$$
	\sum_{t<s\leq T}q_s\Delta_-A_s
	\leq
	\frac{1+\sqrt{1+\beta \mathfrak{C}_A}}{\beta}
	\sum_{t<s\leq T}\left(q_{s-}-q_s\right).
	$$
	
	On the other hand, we clearly have
	$
	\frac{2}{\beta}
	\leq
	\frac{1+\sqrt{1+\beta \mathfrak{C}_A}}{\beta}.
	$
	Combining the continuous and jump parts gives
	$$
	J_t
	\leq
	\frac{1+\sqrt{1+\beta \mathfrak{C}_A}}{\beta}
	\left\{
	\int_{t}^T (-dq_s^c)
	+
	\sum_{t<s\leq T}\left(q_{s-}-q_s\right)
	\right\}.
	$$
	Since $q$ is nonincreasing, the term inside the braces is equal to the total decrease of $q$ on $(t,T]$. Hence
	$$
	\int_{t}^T (-dq_s^c)
	+
	\sum_{t<s\leq T}\left(q_{s-}-q_s\right)
	=
	q_t-q_T
	\leq 1.
	$$
	Thus
	$$
	J_t
	\leq
	\frac{1+\sqrt{1+\beta \mathfrak{C}_A}}{\beta}.
	$$
	Taking the supremum over \(t\in[0,T]\), we obtain
\begin{equation}\label{AH}
	\sup_{0\leq t\leq T}J_t
	\leq
	\frac{1+\sqrt{1+\beta \mathfrak{C}_A}}{\beta}\quad \text{a.s.}.
\end{equation}
	
	Coming back to \eqref{Uranus} and using \eqref{AH}, we have
		\begin{equation}\label{Uranus.1}
		\begin{split}
			\sup_{0\leq t\leq T} \sqrt{\mathcal E(\beta A)_t}|\bar{Y}_t| 
			& \leq K \mathbb E\left[\left( \sup_{0\leq t\leq T} \sqrt{\mathcal E(\beta A)_t} |\bar{y}_{t}|\right) 
				\left( \sup_{0\leq t\leq T}J_t\right) 
			\mid \mathcal F_t
			\right]\\
			& \leq K \frac{1+\sqrt{1+\beta \mathfrak{C}_A}}{\beta} \mathbb E\left[ \sup_{0\leq t\leq T} \sqrt{\mathcal E(\beta A)_t} |\bar{y}_{t}|
			\mid \mathcal F_t
			\right]\quad \text{a.s.}.
		\end{split}
	\end{equation}
By squaring both sides of \eqref{Uranus.1}, taking the supremum and then applying Doob's maximal inequality on its right-hand sides (see, e.g., \cite[Theorem 1.43, p. 11]{JacodShiryaev2013}), we derive
$$
\mathbb{E}\left[\sup_{0\leq t\leq T}\mathcal{E}(\beta A)_t |\bar{Y}_{t}|^2 \right] \leq K^2 \left(\frac{1+\sqrt{1+\beta \mathfrak{C}_A}}{\beta}\right)^2 \mathbb{E}\left[\sup_{0\leq t\leq T}\mathcal{E}(\beta A)_t |\bar{y}_{t}|^2 \right]
$$
Using the elementary inequality $ \left(1+\sqrt{1+\beta \mathfrak{C}_A}\right)^2 \leq 4+2\beta \mathfrak{C}_A$, we  get a simpler estimation 
$$
\mathbb{E}\left[\sup_{0\leq t\leq T}\mathcal{E}(\beta A)_t |\bar{Y}_{t}|^2 \right] \leq K^2\left(\frac{4}{\beta^2}+ \frac{2\mathfrak{C}_A}{\beta}\right) \mathbb{E}\left[\sup_{0\leq t\leq T}\mathcal{E}(\beta A)_t |\bar{y}_{t}|^2 \right]
$$
Finally, by choosing $\beta>0$ such that $K^2\left(\frac{4}{\beta^2}+ \frac{2\mathfrak{C}_A}{\beta}\right)<1$, we derive that 
$$
\|\Psi(y^1)-\Psi(y^2)\|_{\mathcal{S}^{2}_{\beta}} \leq \alpha \|y^1-y^2\|_{\mathcal{S}^{2}_{\beta}},
$$
for some $\alpha \in (0,1)$. Therefore, the map $\Psi$ is a contraction and therefore has a unique fixed point $Y$ which lies in $\mathcal{S}^2_{\beta}$ for the chosen $\beta>0$. Therefore, we have $\Psi(Y)=Y$ in the indistinguishable sense. Thus, following Remark \ref{Otage}, this fixed point provides a unique solution $(Y, M)$ to the GBSDE \eqref{GBSDE} associated with the data $(\xi_T,f,A)$. Moreover, since $A$ is bounded by assumption \textbf{(H3)}, we get
$$
\mathbb{E}\left[
\int_{0}^T
\mathcal{E}(\beta A)_s |Y_{s-}|^2\,dA_s
\right] \leq \mathfrak{C}_A \mathbb{E}\left[\sup_{0\leq t\leq T}\mathcal{E}(\beta A)_t |{Y}_{t}|^2 \right]
<+\infty.
$$

This completes the proof.
\end{proof}


\section{Generalized Reflected BSDEs}\label{sec4}

In this section, we consider the following GRBSDE associated with $(\xi,f,A)$:
\begin{equation}\label{basic equation}
	\left\lbrace 
	\begin{split}
		\text{(i) }&~	Y_t=\xi_T+\int_{t}^{T}f(s,Y_s)dA_s+(K_T-K_t)-(M_T-M_t)\\
		\text{(ii) }&~ Y_t \geq \xi_t,~\forall t \leq T,~\text{a.s.,}\\
		\text{(iii) }&~ \int_{0}^{T}\left(Y_{s-}-\xi_{s-}\right)dK^{\ast}_s+\sum_{0 \leq s <T}\left(Y_s-\xi_s\right)\Delta_{+}K_s=0~\text{ a.s.},
	\end{split}
	\right. 
\end{equation}
where $K_\cdot=K^{\ast}_\cdot+\sum_{0 \leq s <\cdot}\Delta_{+} K_s$.

We now introduce the notion of a solution to the GRBSDE
\eqref{basic equation} associated with the data \((\xi,f,A)\).

\begin{definition}\label{def-GRBSDE}
	A triplet of processes \((Y,M,K)\) is said to be a solution to the
	GRBSDE \eqref{basic equation} associated with the data
	$(\xi,f,A)$ if the following conditions hold:
	\begin{itemize}
		\item
		$
		(Y,M,K)\in
		\mathcal{S}_{\beta}^{2}
		\times
		\mathcal{M}_{\beta}^{2}
		\times
		\mathcal{K},
		$
		for some \(\beta>0\);
		
		\item
		$
		\mathbb{P}\left(
		Y_t
		=
		\xi_T
		+\int_t^T f(s,Y_{s-})\,dA_s
		+(K_T-K_t)
		-\left(M_T-M_t\right),
		~ \textnormal{ for all }t\in[0,T]
		\right)=1;
		$
		
		\item the pair \((Y,K)\) satisfies conditions
		\eqref{basic equation}-(ii)--(iii).
	\end{itemize}
\end{definition}

In this section, we study the existence and uniqueness of solutions to the
GRBSDE \eqref{basic equation} associated with the data \((\xi,f,A)\).
Our approach is divided into two main steps.

\begin{enumerate}
	\item We first consider the particular case in which the generator \(f\)
	does not depend on the \(y\)-variable. Two complementary approaches are
	developed: one based on the Snell envelope representation and another
	based on a modified penalization scheme. The analysis builds on tools from
	optimal stopping theory and the theory of reflected BSDEs with regulated
	trajectories; see, among others,
	\cite{Miryana,grigorova2020optimal,klimsiak2019reflected}.
	This first step also provides several auxiliary results that will be used
	in the proof of the main existence and uniqueness theorem.
	
	\item We then address the general case in which the generator may depend
	on the \(y\)-variable. The argument is based on a fixed-point method in a
	suitable Banach space, in the spirit of the approach developed in
	\cite[Proposition~3.2]{Djehiche}.
\end{enumerate}

\subsection{Case of a generator independent of \(y\): approximation by penalization}

We first consider the case in which the generator \(f\) is independent of the
\(y\)-variable. More precisely, we assume that there exists a process
\[
\mathfrak{f}:\Omega\times[0,T]\longrightarrow\mathbb{R}
\]
such that
\[
f(\omega,t,y)=\mathfrak{f}(\omega,t),
\qquad
(\omega,t,y)\in\Omega\times[0,T]\times\mathbb{R}.
\]

In this subsection, we establish the existence and uniqueness of a solution to
the GRBSDE associated with the data \((\xi,\mathfrak{f},A)\):
\begin{equation}\label{basic equation.indp}
	\left\{
	\begin{aligned}
		\textnormal{(i)}\quad
		&Y_t
		=
		\xi_T
		+\int_t^T \mathfrak{f}(s)\,dA_s
		+(K_T-K_t)
		-(M_T-M_t),
		&& t\in[0,T],
		\\
		\textnormal{(ii)}\quad
		&Y_t
		\geq \xi_t,
		&& \text{for all }t\in[0,T],\quad \text{a.s.},
		\\
		\textnormal{(iii)}\quad
		&\int_0^T
		\left(Y_{s-}-\xi_{s-}\right)\,dK_s^{\ast}
		+
		\sum_{0\leq s<T}
		\left(Y_s-\xi_s\right)\Delta_{+}K_s
		=0,
		&& \text{a.s.}
	\end{aligned}
	\right.
\end{equation}	
	
	
	\subsubsection{A priori estimates and uniqueness}
	
	\begin{proposition}\label{GRBSDE-propo}
		Let $(\xi_T,\mathfrak{f}^1(\cdot),A)$ and
		$(\xi_T,\mathfrak{f}^2(\cdot),A)$ be two sets of data satisfying
		Assumption \textbf{(H)}. Let $(Y^1,M^1,K^1)$ and
		$(Y^2,M^2,K^2)$ be two solutions of the reflected GRBSDE
		\eqref{basic equation.indp}, associated respectively with
		$(\xi,\mathfrak{f}^1(\cdot),A)$ and
		$(\xi,\mathfrak{f}^2(\cdot),A)$.
		
		Then, there exists a constant
		$\mathfrak{C}_{\beta,\mathfrak{C}_A}>0$ such that
		\begin{equation*}
			\begin{aligned}
				&\mathbb{E}\left[
				\sup_{0\leq t\leq T}
				\mathcal{E}(\beta A)_t
				|Y^1_t-Y^2_t|^2
				\right]
				+
				\mathbb{E}\left[
				\int_0^T
				\mathcal{E}(\beta A)_s
				|Y^1_{s-}-Y^2_{s-}|^2\,dA_s
				\right]
				+
				\mathbb{E}\left[
				\int_0^T
				\mathcal{E}(\beta A)_s\,
				d[M^1-M^2]_s
				\right]
				\\
				&\leq
				\mathfrak{C}_{\beta,\mathfrak{C}_A}
				\mathbb{E}\left[
				\int_0^T
				\mathcal{E}(\beta A)_s
				|\mathfrak{f}^1(s)-\mathfrak{f}^2(s)|^2\,dA_s
				\right].
			\end{aligned}
		\end{equation*}
	\end{proposition}
	
	\begin{proof}
		Set
		\[
		\bar{Y}:=Y^1-Y^2,\qquad
		\bar{M}:=M^1-M^2,\qquad
		\bar{K}:=K^1-K^2,
		\]
		and
		\[
		\bar{\mathfrak{f}}(s)
		:=\mathfrak{f}^1(s)-\mathfrak{f}^2(s).
		\]
		Since the two equations have the same terminal condition, we have
		$\bar{Y}_T=0$.
		
		Applying Corollary \ref{Application of Ito formula} to
		$\mathcal{E}(\beta A)|\bar{Y}|^2$, we obtain
		\begin{equation*}
			\begin{aligned}
				0={}&
				\mathcal{E}(\beta A)_t|\bar{Y}_t|^2
				+
				\int_t^T
				\frac{\beta}{1+\beta\Delta_-A_s}
				\mathcal{E}(\beta A)_s
				|\bar{Y}_{s-}|^2\,dA_s
				\\
				&-
				2\int_t^T
				\mathcal{E}(\beta A)_s
				\bar{Y}_{s-}\bar{\mathfrak{f}}(s)\,dA_s
				-
				2\int_t^T
				\mathcal{E}(\beta A)_s
				\bar{Y}_{s-}\,d\bar{K}^{\ast}_s
				\\
				&+
				2\int_t^T
				\mathcal{E}(\beta A)_s
				\bar{Y}_{s-}\,d\bar{M}_s
				+
				\int_t^T
				\mathcal{E}(\beta A)_s\,d[\bar{M}]_s
				\\
				&+
				\sum_{t<s\leq T}
				\mathcal{E}(\beta A)_s
				|\bar{\mathfrak{f}}(s)|^2
				|\Delta_-A_s|^2
				-
				2\int_t^T
				\mathcal{E}(\beta A)_s
				\bar{\mathfrak{f}}(s)
				\Delta_-A_s\,d\bar{M}_s
				\\
				&+
				\sum_{t\leq s<T}
				\mathcal{E}(\beta A)_s
				\left(
				|\Delta_+\bar{K}_s|^2
				-2\bar{Y}_s\Delta_+\bar{K}_s
				\right).
			\end{aligned}
		\end{equation*}
		
		We now show that the terms involving the reflecting processes have
		the appropriate sign. By the minimality condition
		\eqref{basic equation.indp}--\textnormal{(iii)}, we have
		\begin{equation}
			\begin{aligned}
				&\int_t^T
				\mathcal{E}(\beta A)_s
				\bar{Y}_{s-}\,d\bar{K}^{\ast}_s
				\\
				={}&
				\int_t^T
				\mathcal{E}(\beta A)_s
				(Y^1_{s-}-\xi_{s-})
				\left(dK_s^{1,c}+\Delta K_s^{1,d}\right)
			-
				\int_t^T
				\mathcal{E}(\beta A)_s
				(Y^2_{s-}-\xi_{s-})
				\left(dK_s^{1,c}+\Delta K_s^{1,d}\right)
				\\
				&-
				\int_t^T
				\mathcal{E}(\beta A)_s
				(Y^1_{s-}-\xi_{s-})
				\left(dK_s^{2,c}+\Delta K_s^{2,d}\right)
				+
				\int_t^T
				\mathcal{E}(\beta A)_s
				(Y^2_{s-}-\xi_{s-})
				\left(dK_s^{2,c}+\Delta K_s^{2,d}\right)
				\\
				={}&
				-
				\int_t^T
				\mathcal{E}(\beta A)_s
				(Y^2_{s-}-\xi_{s-})
				\left(dK_s^{1,c}+\Delta K_s^{1,d}\right)
				-
				\int_t^T
				\mathcal{E}(\beta A)_s
				(Y^1_{s-}-\xi_{s-})
				\left(dK_s^{2,c}+\Delta K_s^{2,d}\right)
				\\
				\leq{}&0.
			\end{aligned}
			\label{reflection-left-sign}
		\end{equation}
		Indeed, the first and fourth terms vanish by the minimality
		conditions, whereas the remaining terms are nonpositive since
		$Y^i_{s-}\geq\xi_{s-}$ and $K^i$ is nondecreasing.
		
		Similarly, for the right jumps, we have
		\begin{equation}
			\begin{aligned}
				&\sum_{t\leq s<T}
				\mathcal{E}(\beta A)_s
				\bar{Y}_s\Delta_+\bar{K}_s
				\\
				={}&
				\sum_{t\leq s<T}
				\mathcal{E}(\beta A)_s
				(Y^1_s-\xi_s)\Delta_+K^1_s
				-
				\sum_{t\leq s<T}
				\mathcal{E}(\beta A)_s
				(Y^2_s-\xi_s)\Delta_+K^1_s
				\\
				&-
				\sum_{t\leq s<T}
				\mathcal{E}(\beta A)_s
				(Y^1_s-\xi_s)\Delta_+K^2_s
				+
				\sum_{t\leq s<T}
				\mathcal{E}(\beta A)_s
				(Y^2_s-\xi_s)\Delta_+K^2_s
				\\
				={}&
				-
				\sum_{t\leq s<T}
				\mathcal{E}(\beta A)_s
				(Y^2_s-\xi_s)\Delta_+K^1_s
			-
				\sum_{t\leq s<T}
				\mathcal{E}(\beta A)_s
				(Y^1_s-\xi_s)\Delta_+K^2_s
				\\
				\leq{}&0.
			\end{aligned}
			\label{reflection-right-sign}
		\end{equation}
		Here again, the first and fourth terms vanish by the minimality
		conditions, while the remaining terms are nonpositive because
		$Y^i_s\geq\xi_s$ and $\Delta_+K^i_s\geq0$.
		
		Consequently,
		\[
		-2\int_t^T
		\mathcal{E}(\beta A)_s
		\bar{Y}_{s-}\,d\bar{K}^{\ast}_s
		\geq0,
		\]
		and
		\[
		\sum_{t\leq s<T}
		\mathcal{E}(\beta A)_s
		\left(
		|\Delta_+\bar{K}_s|^2
		-2\bar{Y}_s\Delta_+\bar{K}_s
		\right)
		\geq0.
		\]
		Therefore, these terms can be discarded when deriving an upper
		estimate.
		
		The remainder of the proof follows by applying Young's inequality,
		the Burkholder--Davis--Gundy inequality, and the same estimates as
		in Proposition \ref{GBSDE-propo1}. Hence,
		\begin{align*}
			&\mathbb{E}\left[
			\sup_{0\leq t\leq T}
			\mathcal{E}(\beta A)_t|\bar{Y}_t|^2
			\right]
			+
			\mathbb{E}\left[
			\int_0^T
			\mathcal{E}(\beta A)_s
			|\bar{Y}_{s-}|^2\,dA_s
			\right]
			+
			\mathbb{E}\left[
			\int_0^T
			\mathcal{E}(\beta A)_s\,d[\bar{M}]_s
			\right]
			\\
			&\leq
			\mathfrak{C}_{\beta,\mathfrak{C}_A}
			\mathbb{E}\left[
			\int_0^T
			\mathcal{E}(\beta A)_s
			|\bar{\mathfrak{f}}(s)|^2\,dA_s
			\right].
		\end{align*}
		This completes the proof.
	\end{proof}

	As a direct consequence of Proposition \ref{GRBSDE-propo}, we obtain the
	following uniqueness result.
	
	\begin{corollary}\label{uniq.coro.indp.y}
		Let $\xi$ be an irregular lower obstacle, and let
		$\mathfrak{f}(\cdot)$ be a generator independent of the state variable
		$y$. Assume that the data $(\xi,\mathfrak{f}(\cdot),A)$ satisfy
		Assumption \textbf{(H)}. If the GRBSDE
		\eqref{basic equation.indp} admits a solution $(Y,M,K)$, then this
		solution is unique.
	\end{corollary}
	
	\begin{proof}
		Let $(Y^1,M^1,K^1)$ and $(Y^2,M^2,K^2)$ be two solutions of
		\eqref{basic equation.indp} associated with the same data
		$(\xi,\mathfrak{f}(\cdot),A)$. Applying Proposition
		\ref{GRBSDE-propo} with
		\[
		\mathfrak{f}^1(\cdot)
		=
		\mathfrak{f}^2(\cdot)
		=
		\mathfrak{f}(\cdot),
		\]
		we obtain
		\[
		\mathbb{E}\left[
		\sup_{0\leq t\leq T}
		\mathcal{E}(\beta A)_t
		|Y^1_t-Y^2_t|^2
		\right]
		+
		\mathbb{E}\left[
		\int_0^T
		\mathcal{E}(\beta A)_s
		\,d[M^1-M^2]_s
		\right]
		=0.
		\]
		Since $\mathcal{E}(\beta A)>0$, it follows that
		\[
		Y^1=Y^2
		\quad\text{and}\quad
		M^1=M^2
		\]
		up to indistinguishability. Finally, the dynamics of the GRBSDE \eqref{basic equation.indp}-(i) imply
		that $K^1=K^2$. Hence, the solution $(Y,M,K)$ is unique.
	\end{proof}

\subsubsection{Characterization via the Snell envelope}

We now provide a characterization of the first component \(Y\) of the solution to the reflected GBSDE \eqref{basic equation.indp} in terms of an appropriate Snell envelope.

\begin{proposition}\label{Snell.Envlp}
	Under the assumptions of Proposition \ref{GRBSDE-propo}, the GRBSDE \eqref{basic equation.indp} has a unique solution $(Y,M,K)$. Moreover, its first component is characterized by
	$$
	Y_t=\esssup_{\tau \in \mathcal{T}_{[t,T]}} \mathbb{E}\left[\xi_\tau+\int_{t}^{\tau} \mathfrak{f}(s)dA_s \mid  \mathcal{F}_t\right],\qquad t \in [0,T].
	$$
\end{proposition}

\begin{proof}
	Let us set
	$$
	Y_t:=\esssup_{\tau \in \mathcal{T}_{[t,T]}} \mathbb{E}\left[\xi_\tau+\int_{t}^{\tau} \mathfrak{f}(s)dA_s \mid  \mathcal{F}_t\right],\qquad t \in [0,T].
	$$
	Then
	\begin{equation}\label{def1}
		\begin{split}
			S_t:=Y_t+\int_{0}^{t}\mathfrak{f}(s)dA_s
			=\esssup_{\tau \in \mathcal{T}_{[t,T]}} \mathbb{E}\left[\widehat{\xi}_\tau \mid  \mathcal{F}_t\right],
		\end{split}
	\end{equation}
	where
	$$
	\widehat{\xi}_t:=\xi_t+\int_{0}^{t} \mathfrak{f}(s)dA_s
	$$
	is the reward process, or the payoff process. In other words, the process $(S_t)_{t \leq T}$ is the Snell envelope of the process $(\widehat{\xi}_t)_{t \leq T}$. Note that $\widehat{\xi}$ is an optional process with regulated trajectories, since $\xi$ is optional with regulated trajectories and $\int_{0}^{\cdot} \mathfrak{f}(s)dA_s$ is RCLL and predictable. Moreover, from \eqref{bounded_exp_A} and assumptions \textbf{(H1)} and \textbf{(H2)}, we have
	$$
	\mathbb{E}\left[\esssup_{\tau \in \mathcal{T}_{[0,T]}}| \mathcal{E}(\beta A)_\tau  \widehat{\xi}_\tau|^2\right] \leq \mathfrak{C}_{\beta,\mathfrak{C}_A}\left( \mathbb{E}\left[  \esssup_{\tau \in \mathcal{T}_{[0,T]}} \left| \mathcal{E}(\beta A)_{\tau} \xi_{\tau} \right|^2\right] +  \mathbb{E}\left[  \int_{0}^{T} \mathcal{E}(\beta A)_{s} \left| \mathfrak{f}(s) \right|^2 dA_s\right] 
	\right) <+\infty .
	$$
	In particular, the process $\widehat{\xi}$ belongs to $\mathcal{S}^2$. By applying the Mertens decomposition to the process $S$; see Lemma 3.2 in \cite{grigorova2020optimal} (see also \cite[Appendix 1, Thm.20, equalities (20.2)]{DellacherieMeyer1980} or \cite[Theorem A.1]{Miryana}), there exist a unique nondecreasing right-continuous predictable process $(\widehat{K}^{\ast}_t)_{t \leq T}$, a unique nondecreasing right-continuous adapted purely discontinuous process $(\widehat{C}_t)_{t \leq T}$, and a unique martingale $(M_t)_{t \leq T}$ such that $\widehat{K}^{\ast}_0=\widehat{C}_{0-}=0$,
	\begin{equation}\label{def2}
		S_t=S_0+M_t-\widehat{K}^{\ast}_t-\widehat{C}_{t-},\qquad \text{for all } t \in [0,T] \text{ a.s.,}
	\end{equation}
	and
	$$
	\mathbb{E}\left[|\widehat{K}^{\ast}_T|^2+|\widehat{C}_T|^2+[M]_T\right]<+\infty.
	$$
	
	Let us set
	$$
	K_t:=\widehat{K}^{\ast}_t+\widehat{C}_{t-}.
	$$
	Then $K$ is a predictable increasing irregular process such that $K^\ast=\widehat{K}^{\ast}$ and $\Delta_+K_t=\Delta \widehat{C}_t$; see Definition \ref{regulated process}. In other words, $K$ is the increasing process obtained from the Mertens decomposition of $S$.
	
	From Lemma 3.26-(iii) and Lemma 3.3 in \cite{grigorova2020optimal}, we have
	$$
	\int_{0}^{T}\left(\xi_{s-}-Y_{s-}\right)dK^{\ast}_s
	=
	\int_{0}^{T}\left(\xi_s-Y_s\right)d\widehat{K}^{\ast,c}_s
	+
	\sum_{0 < s \leq T}\left(\xi_{s-}-Y_{s-}\right) \Delta_-\widehat{K}^\ast_s
	=0.
	$$
	Similarly, using Lemma 3.26-(ii) in \cite{grigorova2020optimal}, we get
	$$
	\sum_{0 \leq s <T}\left(\xi_s-Y_s\right)\Delta_{+}K_s
	=
	\sum_{0 \leq s <T}\left(\xi_s-Y_s\right)\Delta \widehat{C}_s
	=0.
	$$
	Finally, using \eqref{def1} and \eqref{def2}, we derive
	$$
	Y_t=Y_0-\int_{0}^{t} \mathfrak{f}(s)dA_s-K_t+M_t,\qquad t \in [0,T].
	$$
	Since $Y_T=\xi_T$, it follows that
	$$
	Y_t=\xi_T+\int_{t}^{T} \mathfrak{f}(s)dA_s+(K_T-K_t)-(M_T-M_t),\qquad t \in [0,T].
	$$
	In other words, $(Y,M,K)$ is a solution of the GRBSDE associated with $(\xi,\mathfrak{f}(\cdot),A)$. By the uniqueness result given in Corollary \ref{uniq.coro.indp.y}, this constructed solution is unique. 
	
	From the construction of the state process $Y$ and assumptions \textbf{(H1)}--\textbf{(H3)}, we derive that $Y \in \mathcal{S}^2_\beta$. Moreover, the boundedness property \eqref{bounded_exp_A} allows us to conclude that $M \in \mathcal{M}^2_\beta$. This completes the proof.
\end{proof}

\begin{proposition}\label{Snell.Envlp.general}
	Let $(Y,M,K)$ be a solution to the GRBSDE \eqref{basic equation} associated with $(\xi,{f},A)$ 
	Then, for every $t\in[0,T]$,
	$$
	Y_t=\esssup_{\tau \in \mathcal{T}_{[t,T]}}
	\mathbb{E}\left[
	\xi_\tau+\int_{t}^{\tau}\mathfrak{f}(s,Y_{s-})dA_s
	\mid \mathcal{F}_t
	\right].
	$$
\end{proposition}

\begin{proof}
	Let $(Y,M,K)$ be a solution to the GRBSDE associated with $(\xi,\mathfrak{f},A)$.
	Set
	$$
	\mathfrak{f}^Y(s):=\mathfrak{f}(s,Y_{s-}), \qquad s\in[0,T].
	$$
	Then $(Y,M,K)$ is also a solution to the GRBSDE \eqref{basic equation.indp} associated with
	$(\xi,\mathfrak{f}^Y(\cdot),A)$, where the generator $\mathfrak{f}^Y(\cdot)$ is now independent of $y$.
	Indeed, the GBSDE \eqref{basic equation} becomes a version of \eqref{basic equation.indp}:
	$$
	Y_t=\xi_T+\int_t^T \mathfrak{f}^Y(s)dA_s+(K_T-K_t)-(M_T-M_t),
	\qquad t\in[0,T].
	$$
	Set
	$$
	\widehat{Y}_t=\esssup_{\tau \in \mathcal{T}_{[t,T]}}
	\mathbb{E}\left[
	\xi_\tau+\int_{t}^{\tau}\mathfrak{f}^Y(s)dA_s
	\mid \mathcal{F}_t
	\right],
	\qquad t\in[0,T].
	$$
	By Proposition \ref{Snell.Envlp}, applied to the generator $\mathfrak{f}^Y(\cdot)$, we derive that $\widehat{Y}$ is the first component of the solution to the GRBSDE \eqref{basic equation.indp} associated with
	$(\xi,\mathfrak{f}^Y(\cdot),A)$ and the unqiuenss result (see Corollary \ref{uniq.coro.indp.y}) allows us to derive $Y=\widehat{Y}$. Finally, using the definition of $\mathfrak{f}^Y$, we deduce that
	$$
	Y_t=\esssup_{\tau \in \mathcal{T}_{[t,T]}}
	\mathbb{E}\left[
	\xi_\tau+\int_{t}^{\tau}\mathfrak{f}(s,Y_{s-})dA_s
	\mid \mathcal{F}_t
	\right],
	\qquad t\in[0,T].
	$$
	This completes the proof.
\end{proof}

\subsubsection{A variant of the GBSDE \eqref{GBSDE}}

Let $\mathfrak{g}$ be an additional generator. We consider the following
variant of the GBSDE associated with the data
$(\xi_T,\mathfrak{f}(\cdot),\mathfrak{g},A)$:
\begin{equation}\label{GBSDE.G}
	Y_t
	=
	\xi_T
	+\int_t^T \mathfrak{f}(s)\,dA_s
	+\int_t^T \mathfrak{g}(s,Y_{s-})\,ds
	-\left(M_T-M_t\right),
	\qquad t\in[0,T].
\end{equation}

In addition to the assumptions imposed on
$(\xi_T,\mathfrak{f}(\cdot),A)$, we make the following assumption on the
additional generator $\mathfrak{g}$.

\begin{description}
	\item[\textbf{(H4)}]
	The generator
	\[
	\mathfrak{g}:
	\Omega\times[0,T]\times\mathbb{R}
	\longrightarrow\mathbb{R}
	\]
	satisfies the following conditions:
	\begin{itemize}
		\item[(i)]
		The mapping
		\[
		(\omega,t,y)
		\longmapsto
		\mathfrak{g}(\omega,t,y)
		\]
		is
		$\mathcal{P}\otimes\mathcal{B}(\mathbb{R})
		/\mathcal{B}(\mathbb{R})$-measurable.
		
		\item[(ii)]
		There exists a constant $K_{\mathfrak{g}}>0$ such that
		\[
		\left|
		\mathfrak{g}(\omega,t,y)
		-
		\mathfrak{g}(\omega,t,y')
		\right|
		\leq
		K_{\mathfrak{g}}|y-y'|,
		\]
		for $d\mathbb{P}\otimes dt$-almost every $(\omega,t)$ and all
		$y,y'\in\mathbb{R}$.
		
		\item[(iii)]
		The process $\mathfrak{g}(\cdot,0)$ satisfies
		\[
		\mathbb{E}\left[
		\int_0^T
		\mathcal{E}(\beta A)_s
		|\mathfrak{g}(s,0)|^2\,ds
		\right]
		<+\infty.
		\]
	\end{itemize}
\end{description}

The main result of this subsection is stated as follows.

\begin{theorem}\label{GBSDEG-special_thm}
	Assume that the data
	$(\xi_T,\mathfrak{f}(\cdot),A)$ satisfy Assumption \textbf{(H)}
	and that the generator $\mathfrak{g}$ satisfies
	\textbf{(H4)}. Then the GBSDE \eqref{GBSDE.G} admits a unique
	solution $(Y,M)$ in
	\[
	\mathcal{S}^{2}_{\beta}
	\times
	\mathcal{M}^{2}_{\beta}.
	\]
\end{theorem}

\begin{proof}
	The proof of this result follows standard arguments from the theory of BSDEs. Similar results have already been established in several works; we refer, for instance, to \cite[Theorem 1]{eddahbifakhouriouknine2017} and \cite[Appendix A]{Elmansouri2024}. Nevertheless, since the process $A$ is RCLL, and for the convenience of the reader, we provide the main steps of the proof.
	
	The first step consists in establishing existence and uniqueness for the following GBSDE:
	\begin{equation}\label{GBSDE.G.ind}
		Y_t
		=
		\xi_T
		+\int_t^T \mathfrak{f}(s)\,dA_s
		+\int_t^T \mathfrak{g}(s)\,ds
		-\left(M_T-M_t\right),
		\qquad t\in[0,T].
	\end{equation}
	Here, the coefficient $\mathfrak{g}$ is independent of the state variable $y$. The proof of this particular case is standard and follows the same line of reasoning as the proof of Theorem \ref{GBSDE-special_thm}. We therefore omit the details.
	
	In the second step, we consider the general case in which the generator $\mathfrak{g}$ depends on the state variable $y$. We apply a fixed-point argument based on the GBSDE \eqref{GBSDE.G.ind}. The proof is inspired by that of Theorem \ref{GBSDE-general_thm}, with the main difference being the choice of an appropriate weighted space for the state process $Y$.

	Let
	\[
	W_t^{\beta,\gamma}
	:=
	\mathcal{E}(\beta A)_t e^{\gamma t},
	\qquad t\in[0,T],
	\]
	where \(\beta,\gamma>0\), and define
	\[
	\|Y\|_{\mathcal{S}_{\beta,\gamma}^{2}}^{2}
	:=
	\mathbb{E}\left[
	\sup_{0\leq t\leq T}
	W_t^{\beta,\gamma}|Y_t|^{2}
	\right].
	\]
	
	We define again the mapping
	$
	\Psi:\mathcal{S}_{\beta,\gamma}^{2}
	\longrightarrow
	\mathcal{S}_{\beta,\gamma}^{2}
	$
	as follows. For a given process
	\(y\in\mathcal{S}_{\beta,\gamma}^{2}\), we set
	\(\Psi(y)=Y\), where
	\[
	Y_t
	=
	\mathbb{E}\left[
	\xi_T
	+\int_t^T \mathfrak{f}(s)\,dA_s
	+\int_t^T \mathfrak{g}(s,y_{s-})\,ds
	\Bigm|\mathcal{F}_t
	\right],
	\qquad t\in[0,T].
	\]
	
	Equivalently, there exists a martingale \(M\) such that the pair $(Y,Z)$ satisfies
	\[
	Y_t
	=
	\xi_T
	+\int_t^T \mathfrak{f}(s)\,dA_s
	+\int_t^T \mathfrak{g}(s,y_{s-})\,ds
	-\left(M_T-M_t\right),
	\qquad t\in[0,T].
	\]
	The mapping $\Psi$ is well-defined from the fist part of the current proof.
	
	Let
	$
	y^1,y^2\in\mathcal{S}_{\beta,\gamma}^{2},
	$
	and set
	$
	Y^1=\Psi(y^1),$ $
	Y^2=\Psi(y^2).
	$
	For any \(\mathfrak{S}\in\{Y,y\}\), denote
	$
	\overline{\mathfrak{S}}
	:=
	\mathfrak{S}^1-\mathfrak{S}^2.
$
	
	Since the terminal condition and the process \(\mathfrak{f}\) are
	the same in both equations, these terms cancel when taking the
	difference. Consequently,
	\[
	\overline{Y}_t
	=
	\mathbb{E}\left[
	\int_t^T
	\left(
	\mathfrak{g}(s,y^1_{s-})
	-
	\mathfrak{g}(s,y^2_{s-})
	\right)\,ds
	\Bigm|\mathcal{F}_t
	\right].
	\]
	
	Using assumption \textbf{(H4)}-(ii), we obtain
	\[
	|\overline{Y}_t|
	\leq
	K'\,
	\mathbb{E}\left[
	\int_t^T|\overline{y}_{s-}|\,ds
	\Bigm|\mathcal{F}_t
	\right].
	\]
	
	Multiplying both sides by the positive increasing process
	\(\sqrt{W_t^{\beta,\gamma}}\), we get
	\[
	\begin{aligned}
		\sqrt{W_t^{\beta,\gamma}}|\overline{Y}_t|
		&\leq
		K'\,
		\mathbb{E}\left[
		\int_t^T
		\sqrt{W_s^{\beta,\gamma}}
		|\overline{y}_{s-}|
		\frac{\sqrt{W_t^{\beta,\gamma}}}
		{\sqrt{W_s^{\beta,\gamma}}}
		\,ds
		\Bigm|\mathcal{F}_t
		\right].
	\end{aligned}
	\]
	
	To simplify notations, we set
	\[
	\bar{X}
	:=
	\sup_{0\leq t\leq T}
	\sqrt{W_u^{\beta,\gamma}}
	|\overline{y}_u|.
	\]
	Since the trajectories of the processes are RCLL, their sets
	of discontinuity times are at most countable. Therefore, the
	Lebesgue measure \(ds\) does not charge these discontinuity times,
	and
	\[
	\sqrt{W_s^{\beta,\gamma}}
	|\overline{y}_{s-}|
	\leq \bar X,
	\qquad
	ds\otimes d\mathbb{P}\text{-a.e.}
	\]
	Moreover, for \(t\leq s\),
	\[
	\frac{\sqrt{W_t^{\beta,\gamma}}}
	{\sqrt{W_s^{\beta,\gamma}}}
	=
	\frac{\sqrt{\mathcal{E}(\beta A)_t}}
	{\sqrt{\mathcal{E}(\beta A)_s}}
	e^{-\frac{\gamma}{2}(s-t)}.
	\]
	Since $
	\frac{\sqrt{\mathcal{E}(\beta A)_t}}
	{\sqrt{\mathcal{E}(\beta A)_s}}
	\leq 1,
$
	we get
	\[
	\frac{\sqrt{W_t^{\beta,\gamma}}}
	{\sqrt{W_s^{\beta,\gamma}}}
	\leq
	e^{-\frac{\gamma}{2}(s-t)}.
	\]
	Then it follows that
	\[
	\begin{aligned}
		\int_t^T
		\frac{\sqrt{W_t^{\beta,\gamma}}}
		{\sqrt{W_s^{\beta,\gamma}}}
		\,ds
		\leq
		\int_t^T
		e^{-\frac{\gamma}{2}(s-t)}
		\,ds
		=
		\frac{2}{\gamma}
		\left(
		1-e^{-\frac{\gamma}{2}(T-t)}
		\right)
		\leq
		\frac{2}{\gamma}.
	\end{aligned}
	\]
	Consequently,
	\[
	\sqrt{W_t^{\beta,\gamma}}|\overline{Y}_t|
	\leq
	\frac{2K_g}{\gamma}
	\mathbb{E}\left[
	\bar X\mid\mathcal{F}_t
	\right].
	\]
	Therefore, by Doob's maximal inequality, we get
	\[
	\begin{aligned}
		\|\Psi(y^1)-\Psi(y^2)\|_{\mathcal{S}_{\beta,\gamma}^{2}}^2
		=
		\mathbb{E}\left[
		\sup_{0\leq t\leq T}
		W_t^{\beta,\gamma}
		|\overline{Y}_t|^2
		\right]
		\leq
		\frac{16K'^2}{\gamma^2}
		\mathbb{E}\left[|X|^2\right]
		=
		\frac{16K'^2}{\gamma^2}
		\|y^1-y^2\|_{\mathcal{S}_{\beta,\gamma}^{2}}^2.
	\end{aligned}
	\]
	Thus,
	\[
	\|\Psi(y^1)-\Psi(y^2)\|_{\mathcal{S}_{\beta,\gamma}^{2}}
	\leq
	\frac{4K'}{\gamma}
	\|y^1-y^2\|_{\mathcal{S}_{\beta,\gamma}^{2}}.
	\]
	Choosing
	$
	\gamma>4K',
	$
	we conclude that \(\Psi\) is a strict contraction on
	\(\mathcal{S}_{\beta,\gamma}^{2}\). Hence, by the Banach fixed-point
	theorem, there exists a unique fixed point
	$
	Y\in\mathcal{S}_{\beta,\gamma}^{2}.
	$ This fixed point, together with its associated martingale \(M\) (see Remark ),
	is the unique solution of the GBSDE \eqref{GBSDE.G}.
	
	Completing the proof.
\end{proof}
	
	\begin{remark}
		Following the argument used in Theorem \ref{GBSDEG-special_thm}, we may see that the result  still holds true if we replace the measurability assumption of the coefficient $\mathfrak{g}$ from $ \mathcal{P}  \otimes \mathcal{B}(\mathbb{R})/\mathcal{B}(\mathbb{R})$ to $\textnormal{\textbf{Prog}}\otimes \mathcal{B}(\mathbb{R})/\mathcal{B}(\mathbb{R})$. The reason for keeping the predictable measurability is giving in the next discussion, where we may simply apply Theorem \ref{GBSDE-general_thm} to get the result of Theorem \ref{GBSDEG-special_thm}.
	\end{remark}

	Let us note that the GBSDE \eqref{GBSDE.G} can be writing in a similar form as \eqref{GBSDE}. Indeed, let us set\footnote{The introduced process $B$ is similar to the one presented in Example \ref{exemple}-2.}
	$$
	B_t:=A_t+t,\qquad t \in [0,T].
	$$
	Then the process $(B_t)_{t \leq T}$ is predictable RCLL, nondecreasing, starting from $0$ and bounded by $\mathfrak{C}_A+T$. In other word, it satisfies similar properties as the driver $A$ of the GBSDE \eqref{GBSDE}. Moreover, since $dA_t \ll dB_t$ and $dt \ll dB_t$, using  Theorem 5.14 in \cite[Chapter V]{he2019semimartingale}, we can found two predictable processes $(a_t)_{t \leq T}$ and $(b_t)_{t \leq T}$ such that
	\begin{equation}\label{ab}
	dA_t=a_t dB_t,\quad dt=b_t dB_t,\quad a_t+b_t=1,\quad 0 \leq a_t,b_t, \leq 1,
\end{equation}
	and the GBSDE \eqref{GBSDE.G} can be rewriting as:
	\begin{equation}\label{GBSDE.GN}
		Y_t=\xi_T+\int_{t}^{T}\mathfrak{h}(s,Y_{s-})dB_s-\left(M_T-M_t\right),\quad t \in [0,T],
	\end{equation}
	where
	$$
	\mathfrak{h}(t,y)=a_t \mathfrak{f}(t)+b_t\mathfrak{g}(t,y).
	$$
	Furthermore, we have
	\begin{itemize}
		\item[(i)] The random field $(\omega,t,y) \mapsto \mathfrak{h}(\omega,t,y)$ is
		$\mathcal{P} \otimes \mathcal{B}(\mathbb{R})/\mathcal{B}(\mathbb{R})$-measurable.
		\item[(ii)] for all $(\omega,t) \in \Omega$ and $y,y' \in \mathbb{R}^2$, we have
		$$
		\left|\mathfrak{h}(\omega,t,y)-\mathfrak{h}(\omega,t,y')\right| \leq K' |y-y'|
		$$
		
		\item[(iii)] $\mathbb{E}\left[\int_{0}^{T}\mathcal{E}(\beta A)_s |\mathfrak{h}(s,0)|^2 dB_s\right]<+\infty$.
	\end{itemize}

Consequently, by a direct application of Theorem \ref{GBSDE-general_thm} to the GBSDE \eqref{GBSDE.GN}, we obtain the existence and uniqueness result for the GBSDE \eqref{GBSDE.G} established in Theorem \ref{GBSDEG-special_thm}.

		\subsubsection{Existence via penalization approximation}
	\begin{theorem}\label{main-result_indp}
		Under assumption \textbf{(H)}, the GRBSDE associated with $(\xi,\mathfrak{f}(\cdot),A)$ has a unique solution $(Y,M,K) \in \mathcal{S}^2_\beta \times \mathcal{M}^2_\beta \times \mathcal{K}^2$.
	\end{theorem}

	To prove Theorem \ref{main-result_indp}, we will use a new version of penalization schemes that includes the right-jumps of $\xi$. To this end, we consider an approximation of the solution of GRBSDE \eqref{basic equation.indp} associated with $(\xi,\mathfrak{f}(\cdot),A)$ by a modified penalization method. For each $n \geq 1$, the penalization method is of the following form:
	\begin{equation}\label{penalization}
		\begin{split}
		 Y^n_t=\xi_T&+\int_{t}^{T} \big\{\mathfrak{f}(s)dA_s
		 +n \big(Y^n_{s-}-\xi_{s-}\big)^{-}ds\big\}
		-\int_{t}^{T} dM^n_s +\sum_{t \leq \sigma_{n,i} < T} \left(Y^n_{\sigma_{n,i}+}-\xi_{\sigma_{n,i}}\right)^{-},\quad t \in [0,T],
		\end{split}
	\end{equation}
	Here, $\left\{\sigma_{n,i}\right\}$ denotes a suitably constructed array of stopping times exhausting the right jumps of $\xi$. The existence of such an array is discussed in Remark \ref{existence}. We define $\left\{\sigma_{n,i}\right\}$ inductively as follows:
	We first set
	\begin{equation*}
		\left\{
		\begin{split}
			& \sigma_{1,0}=0,\\
			& \sigma_{1,i}=\inf\left\{t > \sigma_{1,i-1}:\Delta_{+}\xi_t <-1 \right\} \wedge T,\quad i=1,2,\cdots,k_1,\\
			& \sigma_{1,k_1+1}=T,
		\end{split}
		\right.
	\end{equation*}
	for some fixed $k_1 \in \mathbb{N}^{\ast}$. Next, for each $n \in \mathbb{N}$, assuming that the array $\left\{\sigma_{n,i}\right\}$ has already been constructed, we define $\sigma_{n+1,0}=0$ and
	$$
	\sigma_{n+1,i}=\inf\left\{t > \sigma_{n+1,i-1}:\Delta_{+}\xi_t <-\frac{1}{n+1} \right\} \wedge T,\quad i=1,2,\cdots,j_{n+1},
	$$
	where the integer $j_{n+1}$ is chosen so that $\mathbb{P}\left(\sigma_{n+1,j_{n+1}} <T\right) \rightarrow 0$ as $n \rightarrow +\infty$\footnote{For example, we can chose $j_{n+1}$ such that $\mathbb{P}\big(\sigma_{n+1,j_{n+1}} <T\big) \leq \frac{1}{n}$ as in \cite[Section 4 (p. 27)]{klimsiak2020reflected} (see also \cite{klimsiak2020reflected,Marzougue2021,Marzougue2023}).}, and
	$$
	\sigma_{n+1,j_{n+1}+i}=\sigma_{n+1,j_{n+1}} \vee \sigma_{n,i},\quad i=1,2,\cdots, k_n, \qquad \text{with} \qquad k_{n+1}=j_{n+1}+k_n.
	$$
	Finally, we set $\sigma_{n+1,k_{n+1}+1}=T$ and $\mathfrak{f}_n(\omega,t,y)=a_t\mathfrak{f}(\omega,t)+nb_t(y-\xi_{t-})^-$, where the processes $(a_t)_{t \leq T}$ and $(b_t)_{t \leq T}$ are those introduced in \eqref{ab}.
	
	To illustrate the structure of the family $\left\{\sigma_{n,i}\right\}$, we may arrange its elements in the following array:
	\[
	\left\{\sigma_{n,i}\right\}
	=
	\begin{pmatrix}
		0 & \sigma_{1,1} & \sigma_{1,2} & \cdots & \sigma_{1,k_1} & T \\
		0 & \sigma_{2,1} & \sigma_{2,2} & \cdots & \sigma_{2,k_2} & T \\
		\vdots & \vdots & \vdots & \ddots & \vdots & \vdots \\
		0 & \sigma_{n,1} & \sigma_{n,2} & \cdots & \sigma_{n,k_n} & T \\
		\vdots & \vdots & \vdots & \ddots & \vdots & \vdots
	\end{pmatrix}
	\quad
	\begin{array}{l}
		\longrightarrow Y^1 \\
		\longrightarrow Y^2 \\
		\vdots \\
		\longrightarrow Y^n \\
		\vdots
	\end{array}
	\]
	
		Let us note that, for each $n \geq 1$, the dynamic of the state process $(Y^n_t)_{t \leq T}$ of the GBSDE \eqref{penalization} can written as follows:
	\begin{equation}\label{used_dyn}
		Y^n_t=Y^{n,\ast}_t+\sum_{t \leq s < T} \Delta_{+}Y^n_s, \qquad t \in [0,T],
	\end{equation}
	where $(Y^{n,\ast}_t)_{t \leq T}$ is an RCLL $\mathbb{F}$-adapted process defined as the solution of the following GBSDE:
	$$
	Y^{n,\ast}_t=\xi+\int_{t}^{T} \big\{\mathfrak{f}(s)dA_s
	+n \big(Y^n_{s-}-\xi_{s-}\big)^{-}ds\big\}
	-\int_{t}^{T} dM^n_s,\qquad t \in [0,T]
	$$
	and
	$$
	\sum_{t \leq s < T}\Delta_{+}Y^n_s:=-\sum_{t \leq \sigma_{n,i} < T} \left(Y^n_{\sigma_{n,i}+}-\xi_{\sigma_{n,i}}\right)^{-},\qquad t \in [0,T].
	$$
	Note also that from \eqref{penalization}, we have in particular:
$$
	\Delta_+ Y^n_t
	=
	-\sum_{i=0}^{k_n}
	\mathds{1}_{\{t=\sigma_{n,i}\}}
	\left(Y^n_{t+}-\xi_t\right)^{-}
	=
	-\sum_{\substack{0\leq i\leq k_n\\ \sigma_{n,i}=t}}
	\left(Y^n_{\sigma_{n,i}+}-\xi_{\sigma_{n,i}}\right)^{-},
	\qquad t\in[0,T].
	$$
	
	Note that the penalization scheme equations \eqref{penalization} are well-posed. Indeed, we observe that, on each interval $(\sigma_{n,i-1},\sigma_{n,i}]$ for $i=1,2,\cdots,k_n+1$, the GBSDE \eqref{penalization} becomes a classical  GBSDEs of the form:
	\begin{equation}
		\begin{split}
			Y^n_t=&\xi_{\sigma_{n,i}} \vee Y^n_{\sigma_{n,i}+}+\int_{t}^{\sigma_{n,i}} \mathfrak{f}_n(s,Y^n_{s-})dB_s-\int_{t}^{\sigma_{n,i}} dM^n_s,\qquad t \in (\sigma_{n,i-1},\sigma_{n,i}],
		\end{split}
		\label{Penelization equations local}
	\end{equation}
		with the convention that $Y^n_{0}= Y^n_{0+}\vee \zeta_{0}$ and $Y^n_T=\xi_T$. In other words, to solve the GBSDE \eqref{penalization}, we consider it as a classical GBSDE \eqref{Penelization equations local} on each subinterval: $[0, \sigma_{n,1}], (\sigma_{n,1}, \sigma_{n,2}], \cdots, (\sigma_{n,k_n}, T]$, and then construct the corresponding solution backwardly starting from $(\sigma_{n,k_n}, T]$.  We therefore have
	\begin{equation*}
		\begin{split}
			&\mathbb{E}\left[ \int_{0}^{T} \mathcal{E}(\beta A)_s \left| {\mathfrak{f}_n(s,0)}\right|^2 dB_s\right] 
			\leq2\left( \mathbb{E}\left[ \int_{0}^{T} \mathcal{E}(\beta A)_s\left| {\mathfrak{f}(s)}\right|^2 dA_s\right] + {n^2 T}\mathbb{E}\left[ \esssup_{\tau \in \mathcal{T}_{0,T}} \mathcal{E}(\beta A)_\tau  |\xi_\tau|^2 \right]   \right)<+\infty .
		\end{split}
	\end{equation*}
	Hence, from Theorem \ref{GBSDEG-special_thm}, the GBSDE \eqref{Penelization equations local} associated with $(\xi,\mathfrak{f}_n)$ has a unique solution $\left(Y^n,M^n\right)\in \mathcal{S}^2_{\beta} \times \mathcal{M}^2_{\beta}$ on $(\sigma_{n,k_n}, T]$. Following an inductive process, starting now from $(\sigma_{n,k_n-1}, \sigma_{n,k_n}]$ and so on until  $[0, \sigma_{n,1}]$,  we construct a unique solution $\left(Y^n,M^n\right)\in \mathcal{S}^2_{\beta} \times \mathcal{M}^2_{\beta}$ to the GBSDE \eqref{penalization} for each $n \geq 1$.

	The solution of the GBSDE \eqref{penalization} may be stated in the following form:
	\begin{equation}
		\begin{split}
			Y^n_t=\xi&+\int_{t}^{T} \mathfrak{f}(s)dA_s+\int_{t}^{T}dK^{n}_s-\int_{t}^{T}  dM^n_s,\qquad t \in [0,T],
		\end{split}
		\label{Short dynamic of Yn}
	\end{equation}
	where $\left(K^{n}_t\right)_{t \leq T}$ is a regulated process with decomposition:
	\begin{equation}
		K^{n}_t:=K^{n,\ast}_t+\sum_{0 \leq s < t} \Delta_{+} K^{n}_s:=n\int_{0}^{t} \left(Y^n_{s-}-\xi_{s-}\right)^{-}ds+\sum_{0 \leq \sigma_{n,i} < t} \left(Y^n_{\sigma_{n,i}+}-\xi_{\sigma_{n,i}} \right)^{-},\quad t \in [0,T].
		\label{K-n}
	\end{equation}

	We shall need the following auxiliary results toward  the rest of the paper.
	\begin{lemma}\label{result needed for majoration}
		For each $n \geq 1$, and for any $t \in [0,T]$, we have
		$$
		\int_{t}^{T}\left(Y^n_{s-}-\xi_{s-}\right)dK^{n,\ast}_s+\sum_{t \leq s < T}\left(Y^n_s-\xi_s\right)\Delta_{+}K^{n}_s \leq 0,~\text{a.s.}
		$$
	\end{lemma}

	\begin{proof}
		From the definition of the process $\left(\Delta_{+} K^n_t\right)_{t \in [0,T]}$, we have
		\begin{equation}
			\sum_{t \leq s < T}\left(Y^n_s-\xi
			_s\right)\Delta_{+}K^{n}_s=\sum_{t \leq \sigma_{n,i} < T} \left(Y^n_{\sigma_{n,i}}-\xi_{\sigma_{n,i}}\right)\left(Y^n_{\sigma_{n,i}+}-\xi_{\sigma_{n,i}} \right)^{-}.
			\label{in virtue of this equation}
		\end{equation}
		Now, let $n \in \mathbb{N}$ be fixed and assume that there exists an index $i \in \{1,2,\cdots,k_n\}$ such that $t \leq \sigma_{n,i} < T$ and $\left(Y^n_{\sigma_{n,i}}-\xi_{\sigma_{n,i}}\right)\left(Y^n_{\sigma_{n,i}+}-\xi_{\sigma_{n,i}} \right)^{-}>0$. Therefore, we necessarily have $Y_{\sigma_{n,i}}>\xi_{\sigma_{n,i}}$ and $Y^n_{\sigma_{n,i}+}<\xi_{\sigma_{n,i}}$. Moreover, from the penalization scheme \eqref{penalization}, we have  $\Delta_{+}Y^n_{\sigma_{n,i}}=-\Delta_{+} K^{n}_{\sigma_{n,i}}$. Hence, $Y^n_{\sigma_{n,i}}=\xi_{\sigma_{n,i}}$, which leads to a contradiction. Consequently, for every $i \in \{1,2,\cdots,k_n\}$, we have $\big(Y_{\sigma_{n,i}}-\xi_{\sigma_{n,i}}\big)\left(Y^n_{\sigma_{n,i}+}-\xi_{\sigma_{n,i}} \right)^{-}\leq 0$. This inequality, in particular, yields  $\sum_{t \leq s < T}\left(Y^n_s-\xi_s\right)\Delta_{+}K^{n}_s \leq 0$ by virtue of equality \eqref{in virtue of this equation}. On the other hand, we have
		$$
		\int_{t}^{T}\left(Y^n_{s-}-\xi_{s-}\right)dK^{n,\ast}_s=n\int_{t}^{T}\left(Y^n_{s-}-\xi_{s-}\right)\left(Y^n_{s-}-\xi_{s-}\right)^{-}ds=-n\int_{t}^{T}\left(  \left(Y^n_{s-}-\xi_{s-}\right)^{-}\right)^2 ds \leq 0.
		$$
		Completing the proof.
	\end{proof}

	We now state a comparison result for the penalized GBSDE of the form \eqref{penalization}, where right-jumps are included. We shall give a general results that will in particular be applied to the penalized GBSDE with the constructed arrays $\{\sigma_{n,i}\}$ of stopping times. 
	\begin{proposition}\label{comp_pena}
		Let be $\{\sigma_{q,i}\}_{i=0,1,\cdots,k_q+1}$ be a finite sequence of arrays stopping times such that $0=\sigma_{q,0} <\sigma_{q,1}<\sigma_{q,2}<\cdots<\sigma_{q,k_q+1}=T$ for $q \in \{1,2\}$. Let $\{(Y^q,M^q)\}_{q=1,2}$ be a solution of the following GBSDE associated with $(B,\xi_T,\mathfrak{f}_q)$:
		\begin{equation}\label{penalization.comp}
			\begin{split}
				Y^q_t=\xi_T&+\int_{t}^{T} \mathfrak{f}_q(s,Y^q_{s-})dB_s
				-\int_{t}^{T} dM^q_s +\sum_{t \leq \sigma_{q,i} < T} \left(Y^q_{\sigma_{q,i}+}-\xi_{\sigma_{q,i}}\right)^{-},\qquad t \in [0,T].
			\end{split}
		\end{equation}
		Suppose that
		$$
		\xi^1_T\geq \xi^2_T,
		\quad \mathbb{P}\text{-a.s.},
		$$
		and that the following conditions hold:
		\begin{itemize}
			\item[\textnormal{(i)}] $\mathfrak{f}_1(\omega,t,y)\geq \mathfrak{f}_2(\omega,t,y)$ for every $(\omega,t,y)\in\Omega\times[0,T]\times\mathbb{R}$;
			\item[\textnormal{(ii)}] for every $(\omega,t)\in\Omega\times[0,T]$, the mapping $y\mapsto \mathfrak{f}_1(\omega,t,y)$ is nonincreasing;
			\item[\textnormal{(iii)}] $\bigcup_i \llbracket \sigma_{2,i}  \rrbracket \subset \bigcup_i \llbracket \sigma_{1,i}  \rrbracket$
		\end{itemize}
		Then
		$$
		Y^1_t\geq Y^2_t,
		\qquad \forall t\in[0,T],
		\quad \mathbb{P}\text{-a.s.}
		$$
	\end{proposition}

	\begin{proof}
		By applying \cite[Corollary A.5]{klimsiak2019reflected} to the process $(Y^2_t-Y^1_t)^+$ with dynamic that follows \eqref{used_dyn} using \eqref{penalization.comp}, we have
		\begin{equation}\label{oumabb}
			\begin{split}
				\left(Y^2_t-Y^1_t\right)^+
				&\leq
				\left(\xi^2_T-\xi^1_T\right)^+
				+\int_{t}^{T}
				\mathds{1}_{\{Y^2_{s-}>Y^1_{s-}\}}
				\left(\mathfrak{f}_2(s,Y^2_{s-})-\mathfrak{f}_1(s,Y^2_{s-})\right)dB_s \\
				&\quad
				+\int_{t}^{T}
				\mathds{1}_{\{Y^2_{s-}>Y^1_{s-}\}}
				\left(\mathfrak{f}_1(s,Y^2_{s-})-\mathfrak{f}_1(s,Y^1_{s-})\right)dB_s 
				-\int_{t}^{T}
				\mathds{1}_{\{Y^2_{s-}>Y^1_{s-}\}}\,d(M^2_s-M^1_s)\\
				&\quad +\sum_{t \leq \sigma_{2,i} < T}	\mathds{1}_{\{Y^2_{\sigma_{2,i}}>Y^1_{\sigma_{2,i}}\}} \left(Y^2_{\sigma_{2,i}+}-\xi_{\sigma_{2,i}}\right)^{-}- \sum_{t \leq \sigma_{1,i} < T}	\mathds{1}_{\{Y^2_{\sigma_{1,i}}>Y^1_{\sigma_{1,i}}\}} \left(Y^1_{\sigma_{1,i}+}-\xi_{\sigma_{1,i}}\right)^{-}.
			\end{split}
		\end{equation}
	Since	$\bigcup_i \llbracket \sigma_{2,i}  \rrbracket \subset \bigcup_i \llbracket \sigma_{1,i}  \rrbracket$, we have
		\begin{equation*}
			\begin{split}
			&\sum_{t \leq \sigma_{2,i} < T}	\mathds{1}_{\{Y^2_{\sigma_{2,i}}>Y^1_{\sigma_{2,i}}\}} \left(Y^2_{\sigma_{2,i}+}-\xi_{\sigma_{2,i}}\right)^{-}- \sum_{t \leq \sigma_{1,i} < T}	\mathds{1}_{\{Y^2_{\sigma_{1,i}}>Y^1_{\sigma_{1,i}}\}} \left(Y^1_{\sigma_{1,i}+}-\xi_{\sigma_{1,i}}\right)^{-}\\
			& \leq \sum_{t \leq \sigma_{2,i} < T}	\mathds{1}_{\{Y^2_{\sigma_{2,i}}>Y^1_{\sigma_{2,i}}\}}\left\{\left(Y^2_{\sigma_{2,i}+}-\xi_{\sigma_{2,i}}\right)^{-}-\left(Y^1_{\sigma_{2,i}+}-\xi_{\sigma_{2,i}}\right)^{-}\right\}\\
			& \leq 0.
			\end{split}
		\end{equation*}
		Plugging this into \eqref{oumabb} along with the assumption made on the data, then, taking the conditional expectation with respect to $\mathcal{F}_t$  of \eqref{oumabb} and using the martingale property of the stochastic integral on the right-hand side, we get
		$$
		\left(Y^2_t-Y^1_t\right)^+=0,
		\qquad \text{a.s., for each } t\in[0,T].
		$$
		Therefore,
		$$
		Y^2_t\leq Y^1_t,
		\qquad \text{a.s., for each } t\in[0,T].
		$$
		Finally, since $Y^1$ and $Y^2$ are RCLL processes, we conclude that\footnote{A consequence of Theorem 2 in \cite[p. 4]{Protter2004}.}
		$$
		Y^1_t\geq Y^2_t,
		\qquad \forall t\in[0,T],
		\quad \text{a.s.}
		$$
		This completes the proof
	\end{proof}

	After establishing the well-posedness of the penalization schemes \eqref{penalization}, which will be used to prove the existence of a solution to the GRBSDE \eqref{basic equation.indp}, and deriving a useful comparison principle for these schemes in Proposition \ref{comp_pena}, we are now in a position to prove the main existence and uniqueness result stated in Theorem \ref{main-result_indp}. The proof relies on several properties of the sequence $\left\{\left( Y^n,K^{n},M^n\right) \right\}_{n \geq 1}$ established below.\\

	\begin{proof}[Proof of Theorem \ref{main-result_indp}]
		The proof is organized into a sequence of auxiliary lemmas, each establishing a key property of the penalized sequence. Taken together, these results yield the desired existence and uniqueness conclusion. 
		
		We begin by deriving uniform a priori estimates.
		
	\begin{lemma}\label{lemma of unifrom estimation}
		There exists a constant $\mathfrak{C}_{\beta,\mathfrak{C}_A}>0$ independent of $n$ such that for all $\beta >0$
		\begin{equation*}
			\begin{split}
				&\sup_{n \geq 1}\left\lbrace  \mathbb{E}\left[ \esssup_{\tau \in \mathcal{T}_{[0,T]}} \mathcal{E}(\beta A)_{\tau} \left| Y^n_{\tau} \right|^2\right]+ \mathbb{E}\left[  \int_{0}^{T} \mathcal{E}(\beta A)_{s} \left| Y^n_{s-} \right|^2 dA_s\right] +\mathbb{E}\left[  \int_{0}^{T}\mathcal{E}(\beta A)_{s} d \left[M^n\right]_s \right]+\mathbb{E}\left[ \left| K^n_T \right|^2  \right]\right\rbrace \\
				& \leq \mathfrak{C}_{\beta,\mathfrak{C}_A} \left( \mathbb{E}\left[  \esssup_{\tau \in \mathcal{T}_{[0,T]}} \left| \mathcal{E}(\beta A)_{\tau} \xi_{\tau} \right|^2\right] +  \mathbb{E}\left[  \int_{0}^{T} \mathcal{E}(\beta A)_{s} \left| \mathfrak{f}(s) \right|^2 dA_s\right] 
				\right)  . 
			\end{split}
		\end{equation*}
	\end{lemma}

	\begin{proof}
			Let $t \in [0,T]$ be fixed. By applying Corollary \ref{Application of Ito formula} on the time interval $[t,T]$ to the dynamics of the process $Y^n$ given by equation \eqref{Short dynamic of Yn} with right-continuous part
		$
		Y^{n,\ast}_t=Y_0-\int_{0}^{t}\mathfrak{f}(s)dA_s-K^{n,\ast}_t+\int_{0}^{t} dM^n_s
		$ and $\Delta_{+} Y^n = -\Delta_{+} K^{n}$, we can derive 
		\begin{equation}\label{basic Itos formula}
			\begin{split}
				&\mathcal{E}(\beta A)_t \left| Y^n_t \right|^2+ \int_{t}^{T} \dfrac{\beta}{1+\beta \Delta_- A_s} \mathcal{E}(\beta A)_{s} \left| Y^n_{s-} \right|^2 dA_s+\int_{t}^{T} \mathcal{E}(\beta A)_{s} d \left[M^n\right]_s \\
				&   =\mathcal{E}(\beta A)_T \left| \xi_T \right|^2+2\int_{t}^{T}  \mathcal{E}(\beta A)_{s} Y^n_{s-} \mathfrak{f}(s) d A_s+2\int_{t}^{T}  \mathcal{E}(\beta A)_{s} Y^n_{s-}dK^{n,\ast}_s\\
				&\quad-2\int_{t}^{T}  \mathcal{E}(\beta A)_{s} Y^n_{s-}  d M^n_s -\sum_{t < s \leq T}|\mathfrak{f}(s)|^2 |\Delta_- A_s|^2+2\int_{t}^{T}  \mathfrak{f}(s)\Delta_- A_s dM^n_s\\
				&\quad -\sum_{t \leq s < T} \mathcal{E}(\beta A)_{s} \left(\left| \Delta_{+} Y^n_{s}\right|^2-2Y^n_{s}\Delta_{+} K^n_s \right)\\
				& \leq \mathcal{E}(\beta A)_T \left| \xi_T \right|^2+2\int_{t}^{T}  \mathcal{E}(\beta A)_{s} Y^n_{s-} \mathfrak{f}(s) d A_s+2\int_{t}^{T}  \mathcal{E}(\beta A)_{s} Y^n_{s-}dK^{n,\ast}_s\\
				&\quad +2\sum_{t \leq s < T} \mathcal{E}(\beta A)_{s} Y^n_{s}\Delta_{+} K^n_s -2\int_{t}^{T}  \mathcal{E}(\beta A)_{s} Y^n_{s-}  d M^n_s.
			\end{split}
		\end{equation}
	For the generator term, applying the inequality $2ab\leq \epsilon a^2+\frac{1}{\epsilon}b^2$, $
	 \forall (a,b)\in\mathbb{R}^2,$ with
	$\epsilon=\frac{\beta}{2\left(1+\beta\mathfrak{C}_A\right)},$ we obtain
	\begin{equation}\label{hssab0}
		\begin{split}
			2\int_{t}^{T}  \mathcal{E}(\beta A)_{s} Y^n_{s-} \mathfrak{f}(s) d A_s& \leq 	\frac{\beta}{2 \left(1+ \beta \mathfrak{C}_A\right) } \int_{t}^{T}  \mathcal{E}(\beta A)_{s} \left| Y^n_{s-} \right|^2 dA_s+\dfrac{2\left(1+ \beta \mathfrak{C}_A \right)}{\beta}\int_{t}^{T}  \mathcal{E}(\beta A)_{s} \left| \mathfrak{f}(s)\right|^2 dA_s.
		\end{split}
	\end{equation}
	Next, by applying the standard inequalities $a(a-b)^- \leq b^+ (a-b)^-$ and $2 ab \leq 2\epsilon a^2 +\frac{1}{2\epsilon} b^2$, $\forall \epsilon>0$, $\forall (a,b) \in \mathbb{R}^2$, we obtain
	\begin{equation}\label{hssab1}
		\begin{split}
			2\int_{t}^{T}  \mathcal{E}(\beta A)_{s} Y^n_{s-}dK^{n,\ast}_s&=2n\int_{t}^{T}\mathcal{E}(\beta A)_{s} Y^n_{s-} \left(Y^n_{s-}-\xi_{s-}\right)^{-}ds\\
			& \leq 2n\int_{t}^{T}\mathcal{E}(\beta A)_{s} \xi^+_{s-} \left(Y^n_{s-}-\xi_{s-}\right)^{-}ds\\
			& \leq 2\esssup_{\tau \in \mathcal{T}_{0,T}} \mathcal{E}(\beta A)_\tau  |\xi_\tau| \left(K^n_T-K^n_t\right)\\
			&\leq  2\epsilon \esssup_{\tau \in \mathcal{T}_{0,T}} |\mathcal{E}(\beta A)_\tau  \xi_\tau|^2 +\frac{1}{2\epsilon}\left| K^n_T -K^n_t\right|^2.  
		\end{split}
	\end{equation}
	Similarly, using the dynamics of the right jumps of the process $Y^n$ given by \eqref{penalization}, which imply that $Y^n_{\sigma_{n,i}}=Y^n_{\sigma_{n,i}+} \vee \xi_{\sigma_{n,i}}$, we obtain
	\begin{equation}\label{hssab2}
		\begin{split}
			2\sum_{t \leq s < T} \mathcal{E}(\beta A)_{s} Y^n_{s}\Delta_{+} K^n_s &=2\sum_{t \leq \sigma_{n,i} < T} \mathcal{E}(\beta A)_{\sigma_{n,i}} Y^n_{\sigma_{n,i}} \left(Y^n_{\sigma_{n,i}+}-\xi_{\sigma_{n,i}}\right)^{-}\\
			&=2\sum_{t \leq \sigma_{n,i} < T} \mathcal{E}(\beta A)_{\sigma_{n,i}} Y^n_{\sigma_{n,i}+} \vee \xi_{\sigma_{n,i}} \left(Y^n_{\sigma_{n,i}+}-\xi_{\sigma_{n,i}}\right)^{-}\\
			&=2\sum_{t \leq \sigma_{n,i} < T} \mathcal{E}(\beta A)_{\sigma_{n,i}}  \xi_{\sigma_{n,i}} \left(Y^n_{\sigma_{n,i}+}-\xi_{\sigma_{n,i}}\right)^{-}\\
			&\leq  2\epsilon \esssup_{\tau \in \mathcal{T}_{0,T}} |\mathcal{E}(\beta A)_\tau  \xi_\tau|^2 +\frac{1}{2\epsilon}\left| K^n_T -K^n_t\right|^2. 
		\end{split}
	\end{equation}
	Adding \eqref{hssab1} and \eqref{hssab2}, we obtain
	\begin{equation}\label{hssab3}
		2\int_{t}^{T}  \mathcal{E}(\beta A)_{s} Y^n_{s-}dK^{n,\ast}_s+2\sum_{t \leq s < T} \mathcal{E}(\beta A)_{s} Y^n_{s}\Delta_{+} K^n_s \leq 4\epsilon \esssup_{\tau \in \mathcal{T}_{0,T}} |\mathcal{E}(\beta A)_\tau  \xi_\tau|^2 +\frac{1}{\epsilon}\left| K^n_T -K^n_t\right|^2. 
	\end{equation}
	Using \eqref{essential condition-r1} and substituting \eqref{hssab0} and \eqref{hssab3} into \eqref{basic Itos formula}, we obtain
	\begin{equation}
		\begin{split}
			&\mathcal{E}(\beta A)_t \left| Y^n_t \right|^2+\frac{\beta}{2 \left(1+ \beta \mathfrak{C}_A\right) } \int_{t}^{T}  \mathcal{E}(\beta A)_{s} \left| Y^n_{s-} \right|^2 dA_s+\int_{t}^{T} \mathcal{E}(\beta A)_{s} d \left[M^n\right]_s \\
			& \leq \mathcal{E}(\beta A)_T \left| \xi_T \right|^2+\dfrac{2\left(1+ \beta \mathfrak{C}_A \right)}{\beta}\int_{t}^{T}  \mathcal{E}(\beta A)_{s} \left| \mathfrak{f}(s)\right|^2 dA_s+4\epsilon \esssup_{\tau \in \mathcal{T}_{0,T}} |\mathcal{E}(\beta A)_\tau  \xi_\tau|^2\\
			&\quad  +\frac{1}{\epsilon}\left| K^n_T -K^n_t\right|^2 -2\int_{t}^{T}  \mathcal{E}(\beta A)_{s} Y^n_{s-}  d M^n_s
		\end{split}
		\label{basic Itos formula.1}
	\end{equation}
	Taking expectations on both sides of \eqref{basic Itos formula.1} with $t = 0$, and applying Lemmas \ref{ouma} and \ref{ouma1}, we obtain
		\begin{equation}
		\begin{split}
			&\frac{\beta}{2 \left(1+ \beta \mathfrak{C}_A\right) } \mathbb{E}\left[  \int_{0}^{T} \mathcal{E}(\beta A)_{s} \left| Y^n_{s-} \right|^2 dA_s\right] +\mathbb{E}\left[  \int_{0}^{T}\mathcal{E}(\beta A)_{s} d \left[M^n\right]_s \right] \\
			&  \leq \dfrac{2\left(1+ \beta \mathfrak{C}_A \right)}{\beta}  \mathbb{E}\left[  \int_{0}^{T} \mathcal{E}(\beta A)_{s} \left| \mathfrak{f}(s) \right|^2 dA_s\right] 
			+(4\epsilon+1)\mathbb{E}\left[  \esssup_{\tau \in \mathcal{T}_{[0,T]}} \left| \mathcal{E}(\beta A)_{\tau} \xi_{\tau} \right|^2\right] +\dfrac{1}{\epsilon}\mathbb{E}\left[ \left| K^n_T \right|^2\right] .
		\end{split}
		\label{The view of this}
	\end{equation}
	Writing the dynamics of the state process $Y^n$ in \eqref{Short dynamic of Yn} in forward form, then squaring both sides, applying Hölder's inequality together with \eqref{code} and Proposition 4.50 in \cite[p. 53 ]{JacodShiryaev2013}, and finally taking expectations, we obtain
	\begin{equation}
		\begin{split}
			\mathbb{E}\left[ \left| K^n_T \right|^2\right]  & \leq 4 \left(\mathbb{E}\left[\left|Y^n_0 \right|^2\right] +\mathbb{E}\left[ \mathcal{E}(\beta A)_{T}\left|\xi_T \right|^2\right] +\dfrac{1}{\beta}  \mathbb{E} \left[ \int_{0}^{T} \mathcal{E}(\beta A)_{s} \left| \mathfrak{f}(s) \right|^2 dA_s\right] +\mathbb{E}\left[  \int_{0}^{T}\mathcal{E}(\beta A)_{s} d\left[M^n\right]_s\right] \right)
		\end{split}
		\label{Kn est}
	\end{equation}
	Subsequently, choosing $ \epsilon>4$ and substituting this choice into \eqref{The view of this}, we deduce the existence of a constant $\mathfrak{C}_{\beta,\mathfrak{C}_A} > 0$, independent of $n$ and depending only on $\beta$ and $\mathfrak{C}_A$, such that
		\begin{equation}	\label{fr3to lna kna f chofo kidyr}
		\begin{split}
			& \mathbb{E}\left[  \int_{0}^{T} \mathcal{E}(\beta A)_{s} \left| Y^n_{s-} \right|^2 dA_s\right] +\mathbb{E}\left[  \int_{0}^{T}\mathcal{E}(\beta A)_{s} d \left[M^n\right]_s \right] \\
			&  \leq \mathfrak{C}_{\beta,\mathfrak{C}_A} \left( \mathbb{E}\left[  \esssup_{\tau \in \mathcal{T}_{[0,T]}} \left| \mathcal{E}(\beta A)_{\tau} \xi_{\tau} \right|^2\right] +  \mathbb{E}\left[  \int_{0}^{T} \mathcal{E}(\beta A)_{s} \left| \mathfrak{f}(s) \right|^2 dA_s\right] 
			\right)  .
		\end{split}
	\end{equation}
	Therefore, combining this estimate with \eqref{Kn est}, we obtain
	\begin{equation}
		\begin{split}
			&\mathbb{E}\left[  \int_{0}^{T} \mathcal{E}(\beta A)_{s} \left| Y^n_{s-} \right|^2 dA_s\right] +\mathbb{E}\left[  \int_{0}^{T}\mathcal{E}(\beta A)_{s} d \left[M^n\right]_s \right]+\mathbb{E}\left[ \left| K^n_T \right|^2  \right] \\
			&  \leq \mathfrak{C}_{\beta,\mathfrak{C}_A} \left( \mathbb{E}\left[  \esssup_{\tau \in \mathcal{T}_{[0,T]}} \left| \mathcal{E}(\beta A)_{\tau} \xi_{\tau} \right|^2\right] +  \mathbb{E}\left[  \int_{0}^{T} \mathcal{E}(\beta A)_{s} \left| \mathfrak{f}(s) \right|^2 dA_s\right] 
			\right)  .
		\end{split}
		\label{First unif est}
	\end{equation}
It remains to establish a uniform estimate for the sequence of random variables $\left\{\esssup_{\tau \in \mathcal{T}_{[0,T]} }\mathcal{E}(\beta A)_{\tau}\left|Y^n_{\tau} \right|^2 \right\}_{n \geq 1}$. To this end, we return to \eqref{basic Itos formula}, take the essential supremum on both sides, and then apply \eqref{hssab3}. Using arguments similar to those employed in the proof of Proposition \ref{GBSDE-propo1}, together with \eqref{First unif est}, we obtain
\begin{equation*}
	\begin{split}
		\mathbb{E}\left[ \esssup_{\tau \in \mathcal{T}_{[0,T]}} \mathcal{E}(\beta A)_{\tau} \left| Y^n_{\tau} \right|^2\right] 
		&  \leq \mathfrak{C}_{\beta,\mathfrak{C}_A} \left( \mathbb{E}\left[  \esssup_{\tau \in \mathcal{T}_{[0,T]}} \left| \mathcal{E}(\beta A)_{\tau} \xi_{\tau} \right|^2\right] +  \mathbb{E}\left[  \int_{0}^{T} \mathcal{E}(\beta A)_{s} \left| \mathfrak{f}(s) \right|^2 dA_s\right] 
		\right)  . 
	\end{split}
\end{equation*}
This completes the proof of Lemma \ref{lemma of unifrom estimation}.\\
	\end{proof}

	We next establish the monotonicity of the penalized sequence and identify its pointwise limit.
	
	\begin{lemma}\label{increasing}
		There exists an $\mathbb{F}$-optional process $Y:=(Y_t)_{t \leq T}$ such that $Y^n \nearrow Y$ on $[0,T]$ and $\mathcal{E}(\beta A) Y$ is uniformly square integrable.
	\end{lemma}
\begin{proof}
	Note that, since $\mathfrak{f}_{n+1}(s,y) \geq \mathfrak{f}_{n}(s,y)$ for any $s \in [0,T]$, a.s., and for all $y \in \mathbb{R}$, and by the definition of the families $\left\{\sigma_{n,i}\right\}$, due to the implication $\Delta_{+}\xi_t <-\frac{1}{n} $ implies that $\Delta_{+}\xi_t <-\frac{1}{n+1} $, we derive $\bigcup_i \llbracket \sigma_{n,i} \rrbracket \subset \bigcup_i \llbracket \sigma_{n+1,i} \rrbracket $ for each $n \geq 1$. Moreover, since $y \mapsto \mathfrak{f}_n(t,y)$ is nonincreasing for each $n \geq 1$, Proposition \ref{comp_pena} applies and yields $Y^{n+1}_t \geq Y^{n}_t$ for all $t \in [0,T]$ and $n \geq 1$. Consequently, there exists an $\mathbb{F}$-optional process $(Y_t)_{t \leq T}$ such that $Y^n_t \nearrow Y_t$, $\forall t \in [0,T]$ a.s., as $n \rightarrow +\infty$. Finally, applying Fatou's lemma together with the uniform estimate established in Lemma \ref{lemma of unifrom estimation}, we obtain
	\begin{equation*}
		\begin{split}
			\mathbb{E}\left[ \esssup_{\tau \in \mathcal{T}_{[0,T]}} \mathcal{E}(\beta A)_\tau \left| Y_{\tau}\right|^2 \right] 
			&\leq \liminf_{n\rightarrow +\infty} \mathbb{E}\left[ \esssup_{\tau \in \mathcal{T}_{[0,T]}} \mathcal{E}(\beta A)_\tau \left| Y^{n}_{\tau}\right|^2  \right]  \\
			& \leq \mathfrak{C}_{\beta,\mathfrak{C}_A} \left( \mathbb{E}\left[  \esssup_{\tau \in \mathcal{T}_{[0,T]}} \left| \mathcal{E}(\beta A)_{\tau} \xi_{\tau} \right|^2\right] +  \mathbb{E}\left[  \int_{0}^{T} \mathcal{E}(\beta A)_{s} \left| \mathfrak{f}(s) \right|^2 dA_s\right] 
			\right)  . 
		\end{split}
	\end{equation*}
\end{proof}



We now identify the limiting martingale and finite-variation components and establish that the limit process satisfies the GBSDE dynamics.

\begin{lemma}\label{regulated}
	There exists a pair of processes $(M,K) \in \mathcal{M}^2_\beta \times \mathcal{K}^2$  such that the triplet $(Y,M,K)$ is a solution of the GBSDE \eqref{basic equation.indp}-(i), where $Y$ and $K$ has regulated trajectories on $[0,T]$.
\end{lemma}

\begin{proof}
Using the uniform estimation provided in Lemma \ref{lemma of unifrom estimation}, and the Hilbert structure of $\mathcal{M}^2_{\beta}$, we have $\sup_{n \geq 1} \mathbb{E}\left[ \left[ M^n \right]_{T}  \right] \leq \mathfrak{C}_{\beta, \mathfrak{C}_A}$, which implies $\sup_{n \geq 1} \mathbb{E}\left[ \left| M^n_{{T}} \right|^2  \right] \leq \mathfrak{C}_{\beta, \mathfrak{C}_A}$. Therefore, there exists a random variable $\zeta \in \mathbb{L}^2(\Omega,\mathcal{F}_{T},d\mathbb{P})$ such that $M^n_{T} \rightarrow \zeta$ weakly in $\mathbb{L}^2(\Omega,\mathcal{F}_{T},d\mathbb{P})$.

Given that the filtration $\mathbb{F}$ satisfies the standard assumptions, we can consider the RCLL version of the martingale $\left( \mathbb{E}\left[ \zeta \mid \mathcal{F}_t \right]\right)_{t \leq T}$, denoted by $M$. Indeed, for any stopping time $\sigma \in \mathcal{T}_{[0,T]}$ and any $\chi \in \mathbb{L}^2\left(\Omega, \mathcal{F}_{\sigma},d\mathbb{P} \right)$, we have
\begin{equation*}
	\begin{split}
		\mathbb{E}\left[ M^n_{\sigma} \chi  \right] = \mathbb{E}\left[ \mathbb{E}\left[ M^n_{T} \mid \mathcal{F}_{\sigma} \right] \chi  \right]=\mathbb{E}\left[  M^n_{T} \chi   \right] \rightarrow \mathbb{E}\left[ \zeta \chi   \right]
		= \mathbb{E}\left[ \mathbb{E}\left[  \zeta  \mid \mathcal{F}_{\sigma} \right] \chi  \right]=\mathbb{E}\left[ M_{\sigma} \chi  \right].
	\end{split}
\end{equation*}
Then, for any stopping time $\sigma \in \mathcal{T}_{[0,T]}$, $M^n_{\sigma} \rightarrow M_{\sigma}$ weakly in $\mathbb{L}^2\left(\Omega, \mathcal{F}_{\sigma},d\mathbb{P} \right)$.
	
Next, let $\sigma \in \mathcal{T}_{[0,T]}$ be any stopping time. From the equation
	$$
	K^n_\sigma=Y^n_0-Y^n_\sigma-\int_{0}^{t} \mathfrak{f}(s)dA_s+\int_{0}^{\sigma} dM^n_s,
	$$
	we have the following weak convergence
	$$
	K^n_\sigma \rightharpoonup K_\sigma := Y_0-Y_\sigma-\int_{0}^{t} \mathfrak{f}(s)dA_s+\int_{0}^{\sigma} dM_s \quad \text{ as }~ n \rightarrow +\infty.
	$$
	By construction the sequence of process $\{K^n\}_{n \geq 1}$ are increasing predictable with $K^n_0=0$, the limit process $K$ is also predictable increasing with $\mathbb{E}[|K_T|^2]<+\infty$ and $K_0=0$, in fact it is equal to its dual predictable projection, which allows to deduce the described properties. Finally, by the monotone limit theorem for regulated processes; see, for instance, Lemma 1.4 in \cite{Marzougue2021}; the limiting processes $Y$ and $K$ has regulated trajectories on $[0,T]$. In particular, the left limited process $Y_-$ is well defined and we have
	\begin{equation*}
		\begin{split}
			\mathbb{E}\left[  \int_{0}^{T} \mathcal{E}(\beta A)_{s} \left| Y_{s-} \right|^2 dA_s\right] 
			&  \leq \mathfrak{C}_{\beta,\mathfrak{C}_A} \left( \mathbb{E}\left[  \esssup_{\tau \in \mathcal{T}_{[0,T]}} \left| \mathcal{E}(\beta A)_{\tau} \xi_{\tau} \right|^2\right] +  \mathbb{E}\left[  \int_{0}^{T} \mathcal{E}(\beta A)_{s} \left| \mathfrak{f}(s) \right|^2 dA_s\right] 
			\right)  .
		\end{split}
	\end{equation*}
	Consequently, we derive that the limited state process $Y$ satisfies the following GBSDE:
	\begin{equation}\label{blad_beeda}
			Y_t=\xi_T+\int_{t}^{T}\mathfrak{f}(s)dA_s+(K_T-K_t)-(M_T-M_t),\quad t \in [0,T].
	\end{equation}
	
	This completes the proof of Lemma \ref{regulated}.\\
\end{proof}

	We next verify that the limiting process satisfies the obstacle
	condition \eqref{basic equation.indp}-(ii).
	\begin{lemma}\label{dom}
	The limiting state process $Y$ verifies
		$$
		Y_t\geq \xi_t,
		\qquad t\in[0,T].
		$$
	\end{lemma}
	
	\begin{proof}
		Recall that the continuous part of the penalization process is given by
		\[
		K_t^{n,\ast}
		=
		n\int_0^t
		\left(Y_{s-}^n-\xi_{s-}\right)^-\,ds.
		\]
		By Lemma \ref{lemma of unifrom estimation}, there
		exists a constant $\mathfrak{C}_{\beta,\mathfrak{C}_A}>0$, independent of $n$, such that
		$$
		\sup_{n\geq1}
		\mathbb{E}\left[|K_T^n|^2\right]
		\leq \mathfrak{C}_{\beta,\mathfrak{C}_A}.
		$$
		Since \(0\leq K_T^{n,\ast}\leq K_T^n\) by construction, it follows that
	$$
			\mathbb{E}\left[
			\int_0^T
			\left(Y_{s-}^n-\xi_{s-}\right)^-\,ds
			\right]
			=
			\frac{1}{n}\mathbb{E}\left[K_T^{n,\ast}\right] 
			\leq
			\frac{1}{n}
			\left(
			\mathbb{E}\left[|K_T^n|^2\right]
			\right)^{1/2}
			\leq
			\frac{\sqrt{\mathfrak{C}_{\beta,\mathfrak{C}_A}}}{n}.
		$$
		Consequently,
		\[
		\left(Y_{s-}^n-\xi_{s-}\right)^-
		\longrightarrow 0
		\quad\text{in }
		\mathbb{L}^1\bigl(\Omega\times[0,T],
		d\mathbb{P}\otimes ds\bigr).
		\]
		
		Since \(Y^n\) and \(\xi\) have regulated trajectories, each of their
		sample paths has at most countably many discontinuities. Hence,
		\[
		\left(Y_s^n-\xi_s\right)^-
		\longrightarrow 0
		\quad\text{in }
		\mathbb{L}^1\bigl(\Omega\times[0,T],
		d\mathbb{P}\otimes ds\bigr).
		\]
		
		We may therefore extract a subsequence, still denoted by
	$ \{Y^n\}_{n\geq1}$, such that
		\[
		\left(Y_t^n-\xi_t\right)^-
		\longrightarrow 0,
		\qquad
		d\mathbb{P}\otimes dt\text{-a.e.}
		\]
		By Fubini's theorem, there exists a deterministic set
		\(D\subset[0,T]\), having full Lebesgue measure and therefore dense in
		\([0,T]\), such that, for every \(t\in D\),
		\[
		\left(Y_t^n-\xi_t\right)^-
		\longrightarrow0,
		\qquad \mathbb{P}\text{-a.s.}
		\]
		
		Since \(Y_t^n\nearrow Y_t\) for every \(t\in[0,T]\) from Lemma \ref{increasing}, we obtain
		\[
		Y_t\geq\xi_t,
		\qquad t\in D,
		\quad\mathbb{P}\text{-a.s.}
		\]
		By replacing \(D\), if necessary, with a countable dense subset and
		taking a countable intersection of events of probability one, the
		previous inequality may be assumed to hold simultaneously for every
		\(t\in D\).
		
		Let now \(t\in[0,T)\). Choose a sequence
		$\{t_m\}_{m\geq1}\subset D$ such that
		$
		t_m\downarrow t.
		$
		Since both $Y$ and $\xi$ admit right limits, letting
		$m\rightarrow+\infty$ in
		$
		Y_{t_m}\geq\xi_{t_m}
		$
		yields
		\[
		Y_{t+}\geq\xi_{t+}.
		\]
		
		We now prove the desired inequality at time \(t\). From the limiting
		GBSDE \eqref{blad_beeda} and the monotonicity of \(K\), the right jumps of \(Y\) satisfy
		\[
		\Delta_+Y_t=-\Delta_+K_t\leq0.
		\]
		Thus,
		\[
		Y_t
		=
		Y_{t+}-\Delta_+Y_t
		=
		Y_{t+}+\Delta_+K_t
		\geq Y_{t+}.
		\]
		
		\begin{enumerate}
			\item Suppose first that
			\[
			\Delta_+\xi_t\geq0.
			\]
			Then \(\xi_{t+}\geq\xi_t\), and consequently
			\[
			Y_t
			\geq
			Y_{t+}
			\geq
			\xi_{t+}
			\geq
			\xi_t.
			\]

			\item Consider now the case
			\[
			\Delta_+\xi_t<0.
			\]
			Then we have $t \in \bigcup_i \llbracket \sigma_{n,i} \rrbracket$ for sufficiently large $n$. Indeed, by the exhaustion property of the arrays
			\(\{\sigma_{n,i}\}\), the time \(t\) is included among the penalization
			times for all sufficiently large \(n\), on the set where the arrays
			exhaust the negative right jumps of \(\xi\). At such a time,
			$
			\Delta_+K_t^n
			=
			\left(Y_{t+}^n-\xi_t\right)^-.
			$
			Since
			$
			\Delta_+Y_t^n=-\Delta_+K_t^n,
			$
			we have
			$
			Y_t^n
			=
			Y_{t+}^n+\Delta_+K_t^n 
			=
			Y_{t+}^n+
			\left(Y_{t+}^n-\xi_t\right)^- 
			=
			Y_{t+}^n\vee\xi_t.
			$
			In particular,
			\[
			Y_t^n\geq\xi_t
			\]
			for all sufficiently large \(n\). Passing to the increasing limit
			gives
			\[
			Y_t\geq\xi_t.
			\]
		\end{enumerate}
		
		We have therefore proved that
		\[
		Y_t\geq\xi_t,
		\qquad t\in[0,T).
		\]
		Finally, at the terminal time,
		\[
		Y_T=\xi_T,
		\]
		and hence
		\[
		Y_t\geq\xi_t,
		\qquad\text{for every }t\in[0,T],
		\quad\mathbb{P}\text{-a.s.}
		\]
	
	Completing the proof.
		
	\end{proof}
	
	To conclude the proof of Theorem \ref{main-result_indp}. It remains to show the minimality condition  \eqref{basic equation.indp}-(iii) for $K$. To do so, we use the properties of the Snell envelope notion.
	\begin{lemma}\label{SKO}
	 The limited process $K$ satisfies the minimality condition \eqref{basic equation.indp}-(iii).
	\end{lemma}

	\begin{proof}
		From the limiting GBSDE \eqref{blad_beeda}, we can see that the process $\big(Y_t+\int_{0}^{t} \mathfrak{f}(s)dA_s\big)_{t \leq T}$ is a supermartingale that dominates $\big(\xi_t+\int_{0}^{t} \mathfrak{f}(s)dA_s\big)_{t \leq T}$ from Lemma \ref{dom}. Using the definition of the Snell envelope of the process $\big(\xi_t+\int_{0}^{t} \mathfrak{f}(s)dA_s\big)_{t \leq T}$, we derive 
		\begin{equation}\label{f1}
		Y_t \geq \esssup_{\tau \in \mathcal{T}_{[t,T]}} \mathbb{E}\left[\xi_\tau+\int_{t}^{\tau} \mathfrak{f}(s)dA_s \mid  \mathcal{F}_t\right],\qquad t \in [0,T].
		\end{equation}
		Let us now show, the reverse inequality. To this end, observe that the triplet $(Y^n,M^n,K^n)$ constructed in the penalized schemes \eqref{penalization} can be viewed as a solution of the GRBSDE \eqref{basic equation.indp} associated with $(\xi^n,\mathfrak{f}(\cdot),A)$, where $\xi^n:=\xi \wedge Y^n$ for $n \geq 1$. Note that $Y^n \geq \xi^n$, and $Y^n-\xi^n=(Y^n-\xi)^+$ which gives
		\begin{equation*}
			\begin{split}
				&\int_{t}^{T}\left(Y^n_{s-}-\xi^n_{s-}\right)dK^{n,\ast}_s+\sum_{t \leq s < T}\left(Y^n_s-\xi^n_s\right)\Delta_{+}K^{n}_s\\
				&=n\int_{t}^{T}\left(Y^n_{s-}-\xi^n_{s-}\right)\left(Y^n_{s-}-\xi_{s-}\right)^{-}ds+\sum_{t \leq \sigma_{n,i} < T}  (Y^n_{\sigma_{n,i}}-\xi^n_{\sigma_{n,i}}) \left(Y^n_{\sigma_{n,i}+}-\xi_{\sigma_{n,i}}\right)^{-}\\
				&=n\int_{t}^{T}\left(Y^n_{s-}-\xi_{s-}\right)^+\left(Y^n_{s-}-\xi_{s-}\right)^{-}ds+\sum_{t \leq \sigma_{n,i} < T}  (Y^n_{\sigma_{n,i}}-\xi_{\sigma_{n,i}})^+ \left(Y^n_{\sigma_{n,i}+}-\xi_{\sigma_{n,i}}\right)^{-}\\
				&=0.
			\end{split}
		\end{equation*}
		The proof of $\sum_{t \leq \sigma_{n,i} < T}  (Y^n_{\sigma_{n,i}}-\xi_{\sigma_{n,i}})^+ \left(Y^n_{\sigma_{n,i}+}-\xi_{\sigma_{n,i}}\right)^{-}=0$ follows similar argument as those used in Lemma \ref{result needed for majoration}, where we first assume that $\sum_{t \leq \sigma_{n,i} < T}  (Y^n_{\sigma_{n,i}}-\xi_{\sigma_{n,i}})^+ \left(Y^n_{\sigma_{n,i}+}-\xi_{\sigma_{n,i}}\right)^{-} \neq 0$ and then gets a contradiction by the same approach.
		
		Therefore, $(Y^n,M^n,K^n)$ is a solution of the GRBSDE \eqref{basic equation.indp} associated with $(\xi^n,\mathfrak{f}(\cdot),A)$. Then, applying Proposition \ref{Snell.Envlp.general}, we may write 
			$$
		Y^n_t=\esssup_{\tau \in \mathcal{T}_{[t,T]}}
		\mathbb{E}\left[
		\xi^n_\tau+\int_{t}^{\tau}\mathfrak{f}(s)dA_s
		\mid \mathcal{F}_t
		\right],\quad t \in [0,T].
		$$
		By the definition of $\xi^n$, we have $\xi^n \leq \xi$ and then
			$$
		Y^n_t \leq \esssup_{\tau \in \mathcal{T}_{[t,T]}}
		\mathbb{E}\left[
		\xi_\tau+\int_{t}^{\tau}\mathfrak{f}(s)dA_s
		\mid \mathcal{F}_t
		\right],\quad t \in [0,T].
		$$
		Consequently, letting $n \rightarrow +\infty$ and using Lemma \ref{increasing}, we deduce that
		\begin{equation}\label{f2}
		Y_t \leq \esssup_{\tau \in \mathcal{T}_{[t,T]}} \mathbb{E}\left[\xi_\tau+\int_{t}^{\tau} \mathfrak{f}(s)dA_s \mid  \mathcal{F}_t\right],\qquad t \in [0,T].
		\end{equation}
		By combining the inequality \eqref{f1} and \eqref{f2}, we deduce
		$$
		Y_t=\esssup_{\tau \in \mathcal{T}_{[t,T]}}
		\mathbb{E}\left[
		\xi_\tau+\int_{t}^{\tau}\mathfrak{f}(s)dA_s
		\mid \mathcal{F}_t
		\right],\quad t \in [0,T].
		$$
		Moreover, from \eqref{blad_beeda}, we derive that the process $K$ corresponds to the increasing of the Mertens decomposition of the wealth process $\widehat{\xi}=\big(\xi_t+\int_{0}^{t}\mathfrak{f}(s)dA_s\big)_{t \leq T}$ (see the proof of Proposition \ref{Snell.Envlp}). By applying Corollary 3.11 in \cite{klimsiak2019reflected} along with the fact that $\xi$ has regulated trajectories, in particular finite left limits, we conclude
		$$
		\int_{0}^{T}\left(Y_{s-}-\xi_{s-}\right)dK^{\ast}_s+\sum_{0 \leq s <T}\left(Y_s-\xi_s\right)\Delta_{+}K_s=0~\text{ a.s.}
		$$
		Completing the proof of Lemma \ref{SKO}.
	\end{proof}
	
	The proof of Theorem \ref{main-result_indp} is now complete.
		\end{proof}

Let us point out that the conclusion of Lemma \ref{dom} is no longer
guaranteed if the penalization scheme \eqref{penalization}, used to
approximate the solution of the GRBSDE \eqref{basic equation.indp}, is
replaced by
\begin{equation}\label{penalization.1}
	\begin{split}
		Y_t^n
		={}&
		\xi_T
		+\int_t^T
		\left\{
		\mathfrak{f}(s)
		+n\left(Y_{s-}^n-\xi_{s-}\right)^-
		\right\}\,dA_s \\
		&-\int_t^T dM_s^n
		+\sum_{t\leq\sigma_{n,i}<T}
		\left(
		Y_{\sigma_{n,i}+}^n-\xi_{\sigma_{n,i}}
		\right)^-,
		\qquad t\in[0,T].
	\end{split}
\end{equation}

Indeed, although the uniform estimates established in Lemma
\ref{lemma of unifrom estimation} may still be derived for this
modified scheme, they only imply that
\[
\lim_{n\rightarrow+\infty}
\mathbb{E}\left[
\int_0^T
\left(Y_{s-}^n-\xi_{s-}\right)^-
\,dA_s
\right]
=0.
\]
This convergence only controls violations of the obstacle on the
support of the random measure \(dA\). In particular, it provides no
information on intervals that are not charged by \(dA\). Therefore,
one cannot generally deduce that the limiting process \(Y\) satisfies
\[
Y_t\geq\xi_t,
\qquad t\in[0,T],
\]
and condition \eqref{basic equation.indp}-(ii) may fail.

To see this more clearly, consider the setting of
Example \ref{exemple}-(1) with \(q=1\), and let
\[
A_t:=\mathds{1}_{\{t\geq1\}},
\qquad t\in[0,T].
\]
The Stieltjes measure induced by \(A\) is concentrated at the single
time \(t=1\). More precisely,
\[
dA_s=\delta_1(ds),
\]
and hence
\[
\int_0^T
\left(Y_{s-}^n-\xi_{s-}\right)^-
\,dA_s
=
\left(Y_{1-}^n-\xi_{1-}\right)^-.
\]
Consequently, the preceding convergence yields only
\[
\lim_{n\rightarrow+\infty}
\mathbb{E}\left[
\left(Y_{1-}^n-\xi_{1-}\right)^-
\right]
=0.
\]
Provided that \(Y_{1-}^n\) converges to \(Y_{1-}\), one may at most
infer that
\[
Y_{1-}\geq\xi_{1-},
\qquad \mathbb{P}\text{-a.s.}
\]
However, no conclusion can be drawn for times \(t\neq1\).

\begin{example}\label{alahoakbar}
Let \(T=2\), and assume that
\[
\mathfrak{f}\equiv0,
\qquad
\xi_T=\xi_2=0,
\]
and define the optional regulated obstacle $(\xi_t)_{t \leq T}$ by
\[
\xi_t
:=
\begin{cases}
	0, & 0\leq t\leq1,\\[1mm]
	t-1, & 1<t<2,\\[1mm]
	0, & t=2.
\end{cases}
\]
The process \(\xi\) has no negative right jumps on \([0,2)\), so that
the discrete penalization term involving the stopping times
\(\{\sigma_{n,i}\}\) vanishes. Moreover,
\[
\xi_{1-}=0.
\]

For every \(n\geq1\), the pair
\[
Y_t^n\equiv0,
\qquad
M_t^n\equiv0,
\]
satisfies \eqref{penalization.1}. Indeed,
\[
\int_t^2
\left(Y_{s-}^n-\xi_{s-}\right)^-
\,dA_s
=
\mathds{1}_{\{t\leq1\}}
\left(Y_{1-}^n-\xi_{1-}\right)^-
=0,
\]
and the remaining terms in \eqref{penalization.1} are also equal to
zero. It follows that
\[
Y_t^n=0,
\qquad t\in[0,2],
\]
for every \(n\), and therefore
\[
Y_t=0,
\qquad t\in[0,2].
\]

Therefore, we have $Y_{1-}\geq\xi_{1-},$ $\mathbb{P}\text{-a.s}$. Nevertheless,
\[
Y_t=0< t-1=\xi_t,
\qquad t\in(1,2).
\]
Thus, the limiting process does not dominate the obstacle, even
though
\[
\int_0^T
\left(Y_{s-}^n-\xi_{s-}\right)^-
\,dA_s=0
\]
for every \(n \geq 1\).
\end{example}

The above example shows that penalization with respect to \(dA_s\) alone
is insufficient when the measure \(dA\) does not have full support on
\([0,T]\). This is precisely why the continuous penalization term
should instead be integrated with respect to \(ds\), which is our case in \eqref{penalization}.

\subsection{The general case}

We now consider the general setting in which the generator \(f\) may also depend on the $y$-variable.

The main result of this subsection is stated as follows.

\begin{theorem}\label{main result}
	Under assumption \textbf{(H)}, the GRBSDE \eqref{basic equation} associated with the data \((\xi,f,A)\) admits a unique solution \((Y,M,K)\).
\end{theorem}

\begin{proof}
	We define the mapping
	\[
	\Psi:\mathcal{S}_{\beta}^{2}
	\longrightarrow
	\mathcal{S}_{\beta}^{2}
	\]
	as follows. For any \(y\in\mathcal{S}_{\beta}^{2}\), let
	\(\Psi(y)=Y\), where
	\[
	Y_t
	=
	\esssup_{\tau\in\mathcal{T}_{[t,T]}}
	\mathbb{E}\left[
	\xi_\tau
	+
	\int_t^\tau
	{f}(s,y_{s-})\,dA_s
	\Bigm|
	\mathcal{F}_t
	\right],
	\qquad t\in[0,T].
	\]
	
	We first verify that \(\Psi\) is well defined. By the Lipschitz
	continuity of \({f}\), we have
	\[
	\left|{f}(s,y_{s-})\right|^2
	\leq
	2\left|{f}(s,0)\right|^2
	+
	2K^2|y_{s-}|^2.
	\]
	Hence, assumptions \textbf{(H1)}, \textbf{(H2)}-(iii), and
	\textbf{(H3)}, together with
	\(y\in\mathcal{S}_{\beta}^{2}\), yield
	$$
		\mathbb{E}\left[
		\esssup_{\tau\in\mathcal{T}_{[0,T]}}
		\mathcal{E}(\beta A)_\tau|\xi_\tau|^2
		\right]
		+
		\mathbb{E}\left[
		\int_0^T
		\mathcal{E}(\beta A)_s
		\left|{f}(s,y_{s-})\right|^2
		\,dA_s
		\right]
		<+\infty.
	$$
	
	Therefore, Theorem \ref{main-result_indp} ensures that \(Y=\Psi(y)\)
	is well defined and belongs to \(\mathcal{S}_{\beta}^{2}\).
	Moreover, there exists a unique pair
	\[
	(M,K)\in\mathcal{M}_{\beta}^{2}\times\mathcal{K}^{2}
	\]
	such that the triplet \((Y,M,K)\) satisfies
	\begin{equation}\label{basic equation.indp.G}
		\left\{
		\begin{aligned}
			\textnormal{(i)}\quad
			&Y_t
			=
			\xi_T
			+\int_t^T{f}(s,y_{s-})\,dA_s
			+(K_T-K_t)
			-(M_T-M_t),
			&&t\in[0,T],
			\\
			\textnormal{(ii)}\quad
			&Y_t
			\geq\xi_t,
			&&t\in[0,T],\quad\text{a.s.},
			\\
			\textnormal{(iii)}\quad
			&\int_0^T
			\left(Y_{s-}-\xi_{s-}\right)\,dK_s^{\ast}
			+
			\sum_{0\leq s<T}
			\left(Y_s-\xi_s\right)\Delta_+K_s
			=0,
			&&\text{a.s.}
		\end{aligned}
		\right.
	\end{equation}
	
	It remains to prove that \(\Psi\) admits a unique fixed point in the
	complete metric space \(\mathcal{S}_{\beta}^{2}\).
	
	By a classical result from the general theory of processes
	(see, for instance, Theorem IV.84 in
	\cite{DellacherieMeyer1975}), for every
	\(\sigma\in\mathcal{T}_{[0,T]}\), we have
	\[
	\Psi(y)_\sigma
	=
	\esssup_{\tau\in\mathcal{T}_{[\sigma,T]}}
	\mathbb{E}\left[
	\xi_\tau
	+
	\int_\sigma^\tau
	{f}(s,y_{s-})\,dA_s
	\Bigm|
	\mathcal{F}_\sigma
	\right],
	\qquad\text{a.s.}
	\]
	
	Let \(y^1,y^2\in\mathcal{S}_{\beta}^{2}\), and set
	\[
	Y^1:=\Psi(y^1),
	\qquad
	Y^2:=\Psi(y^2).
	\]
	For any \(\mathfrak{S}\in\{Y,y\}\), denote
	\[
	\overline{\mathfrak{S}}
	:=
	\mathfrak{S}^1-\mathfrak{S}^2.
	\]
	
	For any two families of random variables $\{X_\tau : \tau\in\mathcal{T}_{[\sigma,T]}\}$ and $\{Z_\tau : \tau\in\mathcal{T}_{[\sigma,T]}\}$, one has
	\[
	\left|
	\esssup_{\tau\in\mathcal{T}_{[\sigma,T]}}X_\tau
	-
	\esssup_{\tau\in\mathcal{T}_{[\sigma,T]}}Z_\tau
	\right|
	\leq
	\esssup_{\tau\in\mathcal{T}_{[\sigma,T]}}
	|X_\tau-Z_\tau|.
	\]
	Using this inequality and the Lipschitz continuity of
	\(\mathfrak{f}\), we obtain
	\[
	\begin{aligned}
		|\overline{Y}_\sigma|
		&=
		\left|\Psi(y^1)_\sigma-\Psi(y^2)_\sigma\right|
		\\
		&\leq
		\esssup_{\tau\in\mathcal{T}_{[\sigma,T]}}
		\left|
		\mathbb{E}\left[
		\int_\sigma^\tau
		\left(
		{f}(s,y^1_{s-})
		-
		{f}(s,y^2_{s-})
		\right)
		\,dA_s
		\Bigm|
		\mathcal{F}_\sigma
		\right]
		\right|
		\\
		&\leq
		\mathbb{E}\left[
		\esssup_{\tau\in\mathcal{T}_{[\sigma,T]}}
		\int_\sigma^\tau
		\left|
		{f}(s,y^1_{s-})
		-
		{f}(s,y^2_{s-})
		\right|
		\,dA_s
		\Bigm|
		\mathcal{F}_\sigma
		\right]
		\\
		&\leq
		\mathbb{E}\left[
		\int_\sigma^T
		\left|
		{f}(s,y^1_{s-})
		-
		{f}(s,y^2_{s-})
		\right|
		\,dA_s
		\Bigm|
		\mathcal{F}_\sigma
		\right]
		\\
		&\leq
		K\,
		\mathbb{E}\left[
		\int_\sigma^T
		|\overline{y}_{s-}|
		\,dA_s
		\Bigm|
		\mathcal{F}_\sigma
		\right].
	\end{aligned}
	\]
	
	Multiplying both sides by
	\(\sqrt{\mathcal{E}(\beta A)_\sigma}\), we obtain
	\begin{equation}\label{Uranus.11}
		\begin{aligned}
			\sqrt{\mathcal{E}(\beta A)_\sigma}
			|\overline{Y}_\sigma|
			&\leq
			K\,
			\mathbb{E}\left[
			\int_\sigma^T
			\sqrt{\mathcal{E}(\beta A)_s}
			|\overline{y}_{s-}|
			\frac{
				\sqrt{\mathcal{E}(\beta A)_\sigma}
			}{
				\sqrt{\mathcal{E}(\beta A)_s}
			}
			\,dA_s
			\Bigm|
			\mathcal{F}_\sigma
			\right].
		\end{aligned}
	\end{equation}
	
	Set
	\[
	X
	:=
	\esssup_{\tau\in\mathcal{T}_{[0,T]}}
	\sqrt{\mathcal{E}(\beta A)_\tau}
	|\overline{y}_\tau|.
	\]
	Using the equivalent characterization of the
	\(\mathcal{S}_{\beta}^{2}\)-norm for optional processes with
	regulated trajectories, we obtain from \eqref{Uranus.11}
	\[
	\sqrt{\mathcal{E}(\beta A)_\sigma}
	|\overline{Y}_\sigma|
	\leq
	K\,
	\mathbb{E}\left[
	XJ_\sigma
	\Bigm|
	\mathcal{F}_\sigma
	\right],
	\]
	where
	\[
	J_\sigma
	:=
	\int_\sigma^T
	\frac{
		\sqrt{\mathcal{E}(\beta A)_\sigma}
	}{
		\sqrt{\mathcal{E}(\beta A)_s}
	}
	\,dA_s.
	\]
	
	As established in the proof of
	Theorem \ref{GBSDE-general_thm}, the process \(J\) satisfies
	\[
	J_\sigma
	\leq
	C_{\beta,\mathfrak{C}_A}
	:=
	\frac{
		1+\sqrt{1+\beta\mathfrak{C}_A}
	}{\beta},
	\qquad
	\text{a.s.}
	\]
	for every \(\sigma\in\mathcal{T}_{[0,T]}\). Consequently,
	\[
	\sqrt{\mathcal{E}(\beta A)_\sigma}
	|\overline{Y}_\sigma|
	\leq
	K C_{\beta,\mathfrak{C}_A}
	\mathbb{E}\left[
	X\mid\mathcal{F}_\sigma
	\right].
	\]
	
	Let
	\[
	N_t:=\mathbb{E}\left[X\mid\mathcal{F}_t\right],
	\qquad t\in[0,T].
	\]
	Taking the essential supremum over
	\(\sigma\in\mathcal{T}_{[0,T]}\), squaring, and then applying Doob's quadratic maximal inequality, we obtain
	\[
	\begin{aligned}
		\|\Psi(y^1)-\Psi(y^2)\|_{\mathcal{S}_{\beta}^{2}}^2
		&\leq
		K^2C_{\beta,\mathfrak{C}_A}^2
		\mathbb{E}\left[
		\esssup_{\sigma\in\mathcal{T}_{[0,T]}}
		|N_\sigma|^2
		\right]
		\\
		&\leq
		4K^2C_{\beta,\mathfrak{C}_A}^2
		\mathbb{E}\left[X^2\right]
		\\
		&=
		4K^2
		\left(
		\frac{
			1+\sqrt{1+\beta\mathfrak{C}_A}
		}{\beta}
		\right)^2
		\|y^1-y^2\|_{\mathcal{S}_{\beta}^{2}}^2.
	\end{aligned}
	\]
	
	Since
	\[
	\frac{
		1+\sqrt{1+\beta\mathfrak{C}_A}
	}{\beta}
	\longrightarrow0
	\qquad\text{as }\beta\longrightarrow+\infty,
	\]
	we may choose \(\beta>0\) sufficiently large so that
	\[
	\alpha
	:=
	2K
	\frac{
		1+\sqrt{1+\beta\mathfrak{C}_A}
	}{\beta}
	<1.
	\]
	It follows that
	\[
	\|\Psi(y^1)-\Psi(y^2)\|_{\mathcal{S}_{\beta}^{2}}
	\leq
	\alpha
	\|y^1-y^2\|_{\mathcal{S}_{\beta}^{2}},
	\]
	and hence \(\Psi\) is a strict contraction on
	\(\mathcal{S}_{\beta}^{2}\).
	
	By Banach's fixed-point theorem, \(\Psi\) admits a unique fixed point
	\(Y\in\mathcal{S}_{\beta}^{2}\). In particular,
	\[
	\Psi(Y)=Y
	\]
	in the indistinguishable sense. Applying
	\eqref{basic equation.indp.G} with \(y=Y\), there exists a unique
	pair
	\[
	(M,K)\in\mathcal{M}_{\beta}^{2}\times\mathcal{K}^{2}
	\]
	such that \((Y,M,K)\) satisfies
	\[
	Y_t
	=
	\xi_T
	+\int_t^T\mathfrak{f}(s,Y_{s-})\,dA_s
	+(K_T-K_t)
	-(M_T-M_t),
	\]
	together with the obstacle and Skorokhod conditions.
	
	Finally, suppose that
	\((\widetilde{Y},\widetilde{M},\widetilde{K})\) is another solution.
	By the Snell-envelope representation,
	\(\widetilde{Y}\) is also a fixed point of \(\Psi\). The uniqueness
	of the fixed point implies that
	\[
	\widetilde{Y}=Y.
	\]
	Next, using the fact that $\widetilde{Y}=Y$ and writing the GRBSDE \eqref{basic equation}-(i) forwardly, we derive that 
	$$
	K-\widetilde{K}=M-\widetilde{M}.
	$$
	Hence, $M-\widetilde{M}$ is a predictable martingale with finite variation, and then it is equal to zero up to an evanescent set (see. e.g., Corollary.  3.16 in \cite[p. 32]{JacodShiryaev2013}). Therefore, we claim\footnote{This can also be obtained simply from the uniqueness of the Mertens decomposition described in Theorem \ref{main-result_indp} or in Proposition \ref{Snell.Envlp}.}
	\[
	\widetilde{M}=M,
	\qquad
	\widetilde{K}=K.
	\]
	
	Therefore, the reflected GBSDE \eqref{basic equation} admits a unique
	solution.
	
	This completes the proof of Theorem \ref{main result}.
\end{proof}

\section{Perspectives}
In this paper, we have studied generalized backward stochastic differential equations and generalized reflected backward stochastic
differential equations on a general filtered probability space, without assuming that the underlying filtration is quasi-left-continuous. A distinctive feature of our framework is that the equations are driven by a predictable, nondecreasing RCLL process $A$, which can include bounded predictable jumps. We have also allowed the lower obstacle to be an optional regulated process. There are still several natural extensions of this work that remain open and deserve further investigation.

\begin{enumerate}
	\item One initial direction is to relax the boundedness assumption placed on the driver $A$. It would be quite interesting to broaden the current findings to include unbounded or square-integrable increasing processes. Such an extension would require would need more precise a priori estimates, which might involve using localization techniques and incorporating extra integrability or structural conditions on the generator.

	\item  Another key problem to tackle is the consideration of generators that depend not just on the state variable $y$, but also on a martingale integrand or a jump component $z$. When $A$ includes jumps, this becomes quite tricky because the equation creates an implicit relationship at each jump time, involving the solution's jump value. As a result, we need to impose appropriate conditions for jump-wise solvability or invertibility. It would be especially valuable to develop conditions that are weaker than the traditional ones that connect the Lipschitz constants of the generator to the jump sizes of the RCLL driver $A$.
	
	\item Another interesting area to explore is the generalized doubly reflected BSDEs that come with two optional regulated barriers. Here, the state process is kept within the bounds of a lower and an upper obstacle, necessitating the use of two regulated reflecting processes. To prove the existence and uniqueness of solutions, we would need to adapt the Mokobodzki condition to fit a discontinuous stochastic driver $A$, as well as to barriers that have both left and right discontinuities. The equations we derive could also provide a probabilistic representation of nonlinear generalized Dynkin games across various filtrations.

	\item The findings can also be extended to random terminal horizons, as well as to generators that are either monotone or locally Lipschitz. Additionally, the solutions can meet $\mathbb{L}^p$-integrability conditions for values of $p$ that are not equal to 2.
	
	\item From an applied perspective,  the current framework can be used to explore the valuation and hedging of American contingent claims in markets that are incomplete and subject to defaults or other predictable events. In a progressively refined filtration, the process $A$ could serve as the predictable compensator for a default indicator or an integer-valued random measure. The Snell-envelope characterization introduced in this paper may lead the way for nonlinear pricing formulas for American options and other contracts that allow for early exercise, especially in situations with asymmetric information. Additionally, there seem to be exciting possibilities for further applications in dynamic risk measures, nonlinear expectations, and stochastic control problems.
	
	\item At last, diving into numerical approximation methods for the GRBSDEs we've discussed here is a key area for future research. A thorough examination of how quickly the modified penalization scheme converges, along with creating stable time-discretization algorithms, would really enhance the theoretical findings we've laid out in this study.
\end{enumerate}



	\appendix
	\section{Appendix}
	\subsection{It\^o's formula for processes with regulated trajectories}
	The classical It\^o formula for right-continuous semimartingales can be extended to a particular class of possibly non-right-continuous processes with regulated trajectories.
	\begin{theorem}\label{Ito's formula Theorem}
		Let $Y=\left(Y^1,Y^2,\cdots,Y^n\right)$ be an adapted $n$-dimensional process with regulated trajectories admitting the representation
		\begin{equation}
			Y_t=Y^{\ast}_t+\sum_{0 \leq s <t} \Delta_{+} Y_s,\quad \forall t \in [0,T],
			\label{form of the process for Ito}
		\end{equation}
		where $Y^{\ast}=\left(Y^{\ast,1},Y^{\ast,2},\cdots,Y^{\ast,n}\right)$ is an adapted $n$-dimensional RCLL semimartingale and $\sum_{s <t} \left|\Delta_{+}  Y_s\right|  <+\infty$  a.s. Let $F$ be
		a twice continuously differentiable function on $\mathbb{R}^n$. Then the process
		$\left(F\left(Y_t\right)\right)_{t \leq T}$ admits a representation of the form \eqref{form of the process for Ito}. Furthermore, almost surely, for each $n \geq 1$ and all $t \leq T$,
		\begin{equation*}
			\begin{split}
				F\left(Y_t\right)
				&=F\left(Y_0\right)+\sum_{k=1}^n \int_{0}^{t} D^k F\left(Y_{s-}\right)dY^{\ast,k}_s+\dfrac{1}{2}\sum_{k,l=1}^n  \int_{0}^{t} D^k D^l F\left(Y_{s-}\right) d\left[Y^{\ast,k},Y^{\ast,l}\right]^c_s\\
				& \quad+\sum_{0 < s \leq t}\left\{ F\left(Y_s\right)-F\left(Y_{s-}\right)-\sum_{k=1}^n  D^k F\left(Y_{s-}\right)\Delta_{-}Y^{\ast,k}_s\right\}
				+\sum_{0 \leq s < t} \left\{F\left(Y_{s+}\right)-F\left(Y_{s}\right)\right\},
			\end{split}
		\end{equation*}
		where $D^k$ denotes the differentiation operator with respect to the $k$-th coordinate, while $\left[\cdot,\cdot\right]^c$ stands for the continuous part of the quadratic variation $\left[\cdot,\cdot\right]$.
	\end{theorem}

	\begin{proof}
		The proof follows from Theorem A.1 in \cite{klimsiak2019reflected}. we also refer to Theorem A.3 in \cite{Miryana}.
	\end{proof}

	\begin{corollary}\label{Application of Ito formula}
		Let $\left(Y_t\right)_{t \leq T}$ be a one-dimensional adapted process with regulated paths of the form \eqref{form of the process for Ito}. Then, for every $t \in [0,T]$, we have
		\begin{equation*}
			\begin{split}
				\mathcal{E}(\beta A)_{t} \left| Y_t \right|^2&=  \left| Y_0 \right|^2+ \int_{0}^{t} \dfrac{\beta}{1+\beta \Delta_- A_s} \mathcal{E}(\beta A)_{s} \left| Y_{s-} \right|^2 dA_s+2\int_{0}^{t}  \mathcal{E}(\beta A)_{s} Y_{s-} dY^{\ast}_s+\int_{0}^{t} \mathcal{E}(\beta A)_{s} d \left[Y^{\ast}\right]^c_s\\
				&\quad +\sum_{0 < s \leq t} \mathcal{E}(\beta A)_{s} \left| \Delta_{-} Y_{s} \right|^2+\sum_{0 \leq s < t} \mathcal{E}(\beta A)_{s} \left(\left| \Delta_{+} Y_{s}\right|^2+2Y_{s}\Delta_{+} Y_s \right).
			\end{split}
		\end{equation*}
	\end{corollary}
	
	\begin{proof}
		Set $n=2$ and let $F:\mathbb{R}^2 \to \mathbb{R}$ be the twice continuously differentiable function on $\mathbb{R}^2$ defined by $F(x,y)=x \cdot \left| y\right|^2$.
		Applying Theorem \ref{Ito's formula Theorem} to the processes $Y^1:=\left(\mathcal{E}(\beta A)_{t}\right)_{t \leq T}$\footnote{We recall that this process is right-continuous and increasing. Consequently, $\mathcal{E}(\beta A)_{t}=\mathcal{E}(\beta A)_{t+}$ and $d [\mathcal{E}(\beta A)_{\cdot},Y^\ast]_t=\Delta_-\mathcal{E}(\beta A)_{t} dY^\ast_t$; see Remark \ref{rmq1}.} and $Y^2:=\left(Y_t\right)_{t \leq T}$ yields
		\begin{equation}\label{mm1}
			\begin{split}
				\mathcal{E}(\beta A)_{t} \left| Y_t \right|^2&=  \left| Y_0 \right|^2+ \int_{0}^{t}  \mathcal{E}(\beta A)_{s-} \left| Y_{s-} \right|^2 dA_s+2\int_{0}^{t} \mathcal{E}(\beta A)_{s-} Y_{s-} dY^{\ast}_s+\int_{0}^{t} \mathcal{E}(\beta A)_{s} d \left[Y^{\ast}\right]^c_s\\
				&\quad+\sum_{0 < s \leq t}\left\{ \mathcal{E}(\beta A)_{s} \left| Y_s \right|^2-\mathcal{E}(\beta A)_{s-} \left| Y_{s-} \right|^2-\beta\mathcal{E}(\beta A)_{s-} \left| Y_{s-} \right|^2 \Delta_- A_s-2Y_{s-} \mathcal{E}(\beta A)_{s-} \Delta_- Y_s \right\}\\
				&\quad+\sum_{0 \leq s < t} \mathcal{E}(\beta A)_{s} \left\{|Y_{s+}|^2-|Y_{s}|^2\right\}.
			\end{split}
		\end{equation}
		
		By \eqref{matla},
		$
		\mathcal{E}(\beta A)_{s}
		=
		\mathcal{E}(\beta A)_{s-}\left(1+\beta \Delta A_s\right)
		$.
		Hence,
		\begin{equation*}
			\begin{split}
				&\sum_{0 < s \leq t}\left\{ \mathcal{E}(\beta A)_{s} \left| Y_s \right|^2-\mathcal{E}(\beta A)_{s-} \left| Y_{s-} \right|^2-\beta\mathcal{E}(\beta A)_{s-} \left| Y_{s-} \right|^2 \Delta_- A_s \right\}\\
				&=\sum_{0 < s \leq t}\mathcal{E}(\beta A)_{s}\left\{  \left| Y_s \right|^2-\left| Y_{s-} \right|^2\right\}\\
				&=\sum_{0 < s \leq t} \mathcal{E}(\beta A)_{s} \left\{\left| \Delta_{-} Y_{s} \right|^2+2Y_{s-}  \Delta_- Y_s\right\}.
			\end{split}
		\end{equation*}
		Consequently,
		\begin{equation*}
			\begin{split}
				&\sum_{0 < s \leq t}\left\{ \mathcal{E}(\beta A)_{s} \left| Y_s \right|^2-\mathcal{E}(\beta A)_{s-} \left| Y_{s-} \right|^2-\beta\mathcal{E}(\beta A)_{s-} \left| Y_{s-} \right|^2 \Delta_- A_s-2Y_{s-} \mathcal{E}(\beta A)_{s-} \Delta_- Y_s \right\}\\
				&=\sum_{0 < s \leq t}  \left\{\mathcal{E}(\beta A)_{s}\left| \Delta_{-} Y_{s} \right|^2+2Y_{s-}  \Delta_- \mathcal{E}(\beta A)_{s} \Delta_- Y_s\right\}.
			\end{split}
		\end{equation*}
		
		On the other hand, Remark \ref{rmq1}, together with the identity $\Delta_- Y=\Delta_- Y^\ast$, gives
		$$
		2\sum_{0 < s \leq t} Y_{s-}  \Delta_- \mathcal{E}(\beta A)_{s} \Delta_- Y_s=2\int_{0}^{t}Y_{s-}  \Delta_- \mathcal{E}(\beta A)_{s}  dY^\ast_s.
		$$
		It follows that
		\begin{equation}\label{mm2}
			\begin{split}
				&\sum_{0 < s \leq t}\left\{ \mathcal{E}(\beta A)_{s} \left| Y_s \right|^2-\mathcal{E}(\beta A)_{s-} \left| Y_{s-} \right|^2-\beta\mathcal{E}(\beta A)_{s-} \left| Y_{s-} \right|^2 \Delta_- A_s-2Y_{s-} \mathcal{E}(\beta A)_{s-} \Delta_- Y_s \right\}\\
				&=\sum_{0 < s \leq t}  \mathcal{E}(\beta A)_{s}\left| \Delta_{-} Y_{s} \right|^2+2\int_{0}^{t}Y_{s-}  \Delta_- \mathcal{E}(\beta A)_{s}  dY^\ast_s.
			\end{split}
		\end{equation}
		
		Similarly, using the identity
		$\left|Y_{s+}\right|^2-\left|Y_s\right|^2
		=\left|\Delta_+Y_s\right|^2+2Y_s\Delta_+Y_s$, we obtain
		\begin{equation}\label{mm3}
			\sum_{0 \leq s < t} \mathcal{E}(\beta A)_{s} \left\{|Y_{s+}|^2-|Y_{s}|^2\right\}=\sum_{0 \leq s < t} \mathcal{E}(\beta A)_{s} \left(\left| \Delta_{+} Y_{s}\right|^2+2Y_{s}\Delta_{+} Y_s \right).
		\end{equation}
		
		Substituting \eqref{mm2} and \eqref{mm3} into \eqref{mm1} yields
		\begin{equation*}
			\begin{split}
				\mathcal{E}(\beta A)_{t} \left| Y_t \right|^2&=  \left| Y_0 \right|^2+ \int_{0}^{t}  \dfrac{\beta}{1+\beta \Delta_- A_s} \mathcal{E}(\beta A)_{s} \left| Y_{s-} \right|^2 dA_s+2\int_{0}^{t} \mathcal{E}(\beta A)_{s} Y_{s-} dY^{\ast}_s+\int_{0}^{t} \mathcal{E}(\beta A)_{s} d \left[Y^{\ast}\right]^c_s\\
				&\quad+\sum_{0 < s \leq t}  \mathcal{E}(\beta A)_{s}\left| \Delta_{-} Y_{s} \right|^2+\sum_{0 \leq s < t} \mathcal{E}(\beta A)_{s} \left(\left| \Delta_{+} Y_{s}\right|^2+2Y_{s}\Delta_{+} Y_s \right).
			\end{split}
		\end{equation*}
		This completes the proof.
	\end{proof}

Suppose that the right-continuous component $Y^\ast$ of the process $Y$ defined in \eqref{form of the process for Ito} solves the GBSDE \eqref{GBSDE}, with martingale component $M$, associated with the data $(\xi_T,\mathfrak{f}(\cdot),A)$, where the driver does not depend on $y$. Then,
\begin{equation*}
	\begin{split}
		&\int_{0}^{t} \mathcal{E}(\beta A)_{s} d \left[Y^{\ast}\right]^c_s+\sum_{0 < s \leq t}  \mathcal{E}(\beta A)_{s}\left| \Delta_{-} Y_{s} \right|^2\\
		&=\int_{0}^{t} \mathcal{E}(\beta A)_{s} d \left[M\right]_s+\sum_{0 < s \leq t}|\mathfrak{f}(s)|^2 |\Delta_- A_s|^2-2\int_{0}^{t}  \mathfrak{f}(s)\Delta_- A_s dM_s 
	\end{split}
\end{equation*}
Consequently, the formula established in Corollary \ref{Application of Ito formula} can be written, for any $\tau,\eta \in \mathcal{T}_{0,T}$ satisfying $\eta \leq \tau$ a.s., as
\begin{equation*}
	\begin{split}
		\mathcal{E}(\beta A)_{\tau} \left| Y_\tau \right|^2&=  \mathcal{E}(\beta A)_{\eta}\left| Y_\eta \right|^2+ \int_{\eta}^{\tau} \dfrac{\beta}{1+\beta \Delta_- A_s} \mathcal{E}(\beta A)_{s} \left| Y_{s-} \right|^2 dA_s-2\int_{\eta}^{\tau}  \mathcal{E}(\beta A)_{s} Y_{s-} \mathfrak{f}(s) d A_s\\
		&\quad+2\int_{\eta}^{\tau}  \mathcal{E}(\beta A)_{s} Y_{s-}  d M_s +\int_{\eta}^{\tau} \mathcal{E}(\beta A)_{s} d \left[M\right]_s +\sum_{\eta < s \leq \tau}\mathcal{E}(\beta A)_{s}|\mathfrak{f}(s)|^2 |\Delta_- A_s|^2\\
		&\quad -2\int_{\eta}^{\tau} \mathcal{E}(\beta A)_{s} \mathfrak{f}(s)\Delta_- A_s dM_s+\sum_{\eta \leq s < \tau} \mathcal{E}(\beta A)_{s} \left(\left| \Delta_{+} Y_{s}\right|^2+2Y_{s}\Delta_{+} Y_s \right).
	\end{split}
\end{equation*}

\section*{Disclosure statement}
No potential conflict of interest was reported by the authors.

\section*{Funding}
No funding was received.


	
\end{document}